\documentclass[10pt]{amsart}
\usepackage{setspace}
\usepackage{amsmath, amsfonts, amsthm, amstext, xspace}
\usepackage{amssymb,latexsym}
\usepackage{xcolor}
\usepackage{comment}
\usepackage{hyperref}
\usepackage{xcolor}
\definecolor{seagreen}{RGB}{46,139,87}
\definecolor{maroon}{RGB}{128,0,0}
\definecolor{darkviolet}{RGB}{148,0,211}
\definecolor{twelve}{RGB}{100,100,170}
\definecolor{thirteen}{RGB}{100,150,50}
\definecolor{fourteen}{RGB}{200,0,0}
\definecolor{fifteen}{RGB}{0,200,0}
\definecolor{sixteen}{RGB}{0,0,200}
\definecolor{seventeen}{RGB}{200,0,200}
\definecolor{eighteen}{RGB}{0,200,200}

\usepackage{mathrsfs}

\newtheorem{thm}{Theorem}[section]
\newtheorem{lem}[thm]{Lemma}
\newtheorem{cor}[thm]{Corollary}

\newtheorem{prop}[thm]{Proposition}
\newtheorem{conj}[thm]{Conjecture}

\newtheorem{fig}[thm]{Figure}

\newtheorem{thmx}{Theorem}

\theoremstyle{definition}
\newtheorem{defin}[thm]{Definition}
\newtheorem{constr}[thm]{Construction}
\newtheorem{exm}[thm]{Example}
\theoremstyle{remark} \newtheorem{rem2}[thm]{Remark}

\newcommand{\ul}[1]{\underline{#1}}

\def\c{\mathbb{C}}
\def\f{\mathbb{F}}

\def\m{\mathbb{M}}

\def\r{\mathbb{R}}

\def\z{\mathbb{Z}}

\def\ca{\mathcal{A}}

\def\uz{\underline{\mathbb{Z}}}
\newcommand{\oxi}{\overline{\xi}}
\newcommand{\otau}{\overline{\tau}}

\def\Motc{\mathcal{M}ot_\c}

\def\hocolim{\underset{\longrightarrow}{\operatorname{hocolim}}}

\def\Spec{\operatorname{Spec}}

\def\Ext{\operatorname{Ext}}
\def\TExt{\operatorname{TExt}}

\def\colim{\operatorname{colim}}
\def\lim{\operatorname{lim}}
\def\coker{\operatorname{coker}}

\newcommand{\kq}{kq}
\newcommand{\SH}{\mathcal{SH}}

\newcommand{\Gal}{\operatorname{Gal}}

\newcommand{\kqE}{\mbox{}^{kq}E}
\newcommand{\algE}{\mbox{}^{alg}E}

\usepackage{amssymb}
\usepackage{tikz-cd}

\theoremstyle{theoremstyle}

\newtheorem*{theorem*}{Theorem}

\newtheorem*{proposition*}{Proposition}

\newtheorem*{corollary*}{Corollary}

\newtheorem*{remark*}{Remark}
\newtheorem*{conjecture*}{Conjecture}

\newtheorem{defn*}{Definition}
\theoremstyle{definition}

\theoremstyle{theoremstyle}

\author{Dominic Leon Culver}\address{University of Illinois, Urbana-Champaign}\email{dculver@illinois.edu}
\author{J.D. Quigley}\address{University of Notre Dame}\email{jquigle2@nd.edu}
\title[$kq$-resolutions I]{$kq$-resolutions I}

\begin{document}
\maketitle

\section*{Abstract}
Let $kq$ denote the very effective cover of Hermitian K-theory. We apply the $kq$-based motivic Adams spectral sequence, or $kq$-resolution, to computational motivic stable homotopy theory. Over base fields of characteristic not two, we prove that the $n$-th stable homotopy group of motivic spheres is detected in the first $n$ lines of the $kq$-resolution, thereby reinterpreting results of Morel and R{\"o}ndigs-Spitzweck-{\O}stv{\ae}r in terms of $kq$ and $kq$-cooperations. Over algebraically closed fields of characteristic 0, we compute the ring of $kq$-cooperations modulo $v_1$-torsion, establish a vanishing line of slope $1/5$ in the $E_2$-page, and completely determine the $0$- and $1$- lines of the $kq$-resolution. This gives a full computation of the $v_1$-periodic motivic stable stems and recovers Andrews and Miller's calculation of the $\eta$-periodic $\mathbb{C}$-motivic stable stems. We also construct a motivic connective $j$ spectrum and identify its homotopy groups with the $v_1$-periodic motivic stable stems. Finally, we propose motivic analogs of Ravenel's Telescope and Smashing Conjectures and present evidence for both. 

\tableofcontents

\section{Introduction} 

\subsection{Motivation and main theorems}

Motivic stable homotopy groups are the central object of investigation in computational motivic stable homotopy theory. Their study is motivated by deep connections to arithmetic invariants such as Milnor-Witt K-theory \cite{Mor12} and Hermitian K-theory \cite{RSO19}, as well as classical \cite{Isa19}\cite{IWX20} and equivariant \cite{BS19}\cite{BGI20}\cite{BI20}\cite{DI17} stable homotopy groups. The primary tools for accessing motivic stable homotopy groups are the (very) effective slice spectral sequence \cite{Lev08}\cite{SO12}\cite{Voe02} and the motivic Adams(-Novikov) spectral sequence \cite{DI05}\cite{HKO11}\cite{OO14}\cite{WO17}. 

This paper initiates the study of a new spectral sequence in motivic stable homotopy theory. Fix a field $F$ of characteristic not two and let $kq$ denote the very effective cover of Hermitian K-theory \cite{ARO17}.

\begin{thmx}[$kq$-resolution, Theorem \ref{Thm:kqASS}]\label{ThmA}
There is a strongly convergent spectral sequence, the \emph{$kq$-resolution}, of the form
\begin{equation}\label{Eqn:ThmA}
E_1^{n,t,u} = \pi_{t,u}^F(kq \wedge \overline{kq}^{\wedge n}) \Rightarrow \pi_{t-n,u}^F(S^{0,0}),
\end{equation}
where $\overline{kq}$ is the cofiber of the unit map $S^{0,0} \to kq$. The $d_r$-differentials have the form
$$d_r : E_r^{n,t,u} \to E_r^{n+r, t+r-1,u}.$$
There is an analogous strongly convergent spectral sequence of the form
\begin{equation}\label{Eqn:ThmA2}
E_1^{n,t,u} = \pi_{t,u}^F(kq_p^\wedge \wedge \overline{kq_p^\wedge}^{\wedge n}) \Rightarrow \pi_{t-n,u}^F((S^{0,0})_p^\wedge).
\end{equation}
\end{thmx}

Since the $E_1$-page of \eqref{Eqn:ThmA} contains arithmetic input, the $kq$-resolution is naturally suited for relating low-dimensional motivic stable homotopy groups to arithmetic invariants. Let $\Pi^F_k(-):= \bigoplus_{u\in \z}\pi^F_{k+u,u}$ denote the $k$-th homotopy group in the homotopy t-structure on the $F$-motivic stable homotopy category \cite{Mor12}, also known as the $k$-th Milnor-Witt stem. Applying calculations of \cite{RSO19} to show that $kq \wedge \overline{kq}^{\wedge n}$ is $2n$-connective in the homotopy t-structure, we obtain the following from Theorem \ref{ThmA}. 

\begin{thmx}[Milnor-Witt stems, Theorem \ref{Thm:MWStems}]\label{ThmB}
There is an isomorphism of graded abelian groups
\begin{equation}\label{Eqn:ThmB1}
\Pi_0^F(S^{0,0}) \cong \Pi_0^F(kq),
\end{equation}
and a short exact sequence of graded abelian groups
\begin{equation}\label{Eqn:ThmB2}
0 \to \Pi_2((kq \wedge \overline{kq})/d_1) \to \Pi_1^F(S^{0,0}) \to \Pi_1^F(kq)  \to 0.
\end{equation}
More generally, $\Pi^F_k(S^{0,0})$ is detected in $E_\infty^{n,t,u}$ of \eqref{Eqn:ThmA} in the range of tridegrees $(n,t,u)$ satisfying $n \leq k$ and $t-n-u=k$.
\end{thmx}

\begin{remark*}
In \cite{ARO17}, it is remarked that the geometric meaning of $kq$ is not well understood. Using known calculations of $\Pi_n^F(S^{0,0})$ from \cite{Mor12}\cite{RSO19}, Theorem \ref{ThmB} implies that
$$\Pi_0^F(kq) \cong K_*^{MW}(F);$$
compare with \cite{Bac17}. We also find 
$$\Pi_2^F((kq \wedge \overline{kq})/d_1) \cong K_*^M(F)/24.$$
\end{remark*}

We will explore further connections between $kq$-cooperations and arithmetic invariants, as well as other applications of \eqref{Eqn:ThmA}, in future work. 

The primary aim of this paper is to apply the $2$-complete $kq$-resolution \eqref{Eqn:ThmA2} to the study of large-scale periodic phenomena in the $2$-complete motivic stable stems. From now on, we fix an algebraically closed field $F = \overline{F}$ of characteristic 0 and work in the category of $2$-complete cellular motivic spectra. 

Substantial progress has been made in the past decade towards understanding periodicity in the motivic stable stems \cite{And18}\cite{Ghe17b}\cite{GIKR18}\cite{Hor18}\cite{Kra18}\cite{Qui19a}\cite{Qui19c}\cite{Qui19b}. However, even the most elementary form of periodicity classically, $v_1$-periodicity, has not been completely understood over any base field. Our first application of the $kq$-resolution is to completely understand $v_1$-periodicity over algebraically closed fields. 

\begin{thmx}[$v_1$-periodic stable stems, Theorem \ref{Thm:v1periodic}]\label{ThmD}
The $v_1$-torsion free component of the motivic stable stems is given by
\begin{align*}
\m_2[h_0,h_1,v_1^4]/(h_0h_1,h_0v_1^4,\tau h_1^3) \oplus  \bigoplus_{k \geq 0} \Sigma^{4k-1,2k} \z/2^{\rho(k)}[\tau] \oplus \m_2[h_1,v_1^4]/(h_1^3 \tau)
\end{align*}
where $|h_0| = (0,0)$, $|h_1| = (1,1)$, $|v_1^4| = (8,4)$, $|\tau| = (0,-1)$, and $\rho(k)$ is the $2$-adic valuation of $8k$. 
\end{thmx}

In the classical setting, $v_1$-periodicity is closely linked with the J-homomorphism $J : \pi_*(bo) \to \pi_{*-1}(S^0)$. The calculation of the image of J \cite{Ada66}\cite{DM89}\cite{Fri80}\cite{Mah81}\cite{Qui71}\cite{Sul74} produced new connections between algebraic topology, representation theory, and geometry, and moreover, led to the development of chromatic homotopy theory. The image of J can be extracted from the spectrum $j$ which sits in a fiber sequence
\begin{equation}\label{Eqn:jc}
j \to bo \xrightarrow{\psi^3-1} \Sigma^4 bsp.
\end{equation}
Using the theory of $\Gamma_\star(S^0)$-modules \cite{GIKR18} and the $kq$-resolution, we construct a motivic analog of \eqref{Eqn:jc} and compute its homotopy over algebraically closed fields. 

\begin{thmx}[Motivic $j$ spectrum and its homotopy, Theorem \ref{Thm:Jo}]\label{ThmE}
There is a fiber sequence
\begin{equation}\label{Eqn:ThmE}
j_o \to kq \to \Sigma^{4,2}ksp
\end{equation}
where $ksp$ is the very effective cover of $\Sigma^{4,2}KQ$. The homotopy of the fiber is given by
\begin{align*}
\pi_{**}(j_o) \cong & \m_2[h_0,h_1,v_1^4]/(h_0h_1,h_0v_1^4,\tau h_1^3) \oplus  \\
	& \bigoplus_{k \geq 0} \Sigma^{4k-1,2k} \z/2^{\rho(k)}[\tau] \oplus \m_2[h_1,v_1^4]/(h_1^3 \tau)
\end{align*}
where $|h_0| = (0,0)$, $|h_1| = (1,1)$, $|v_1^4| = (8,4)$, $|\tau| = (0,-1)$, and $\rho(k)$ is the $2$-adic valuation of $8k$. 
\end{thmx}

\begin{remark*}
Hu-Kriz-Ormsby define unitary and orthogonal motivic J homomorphisms in \cite{HKO11} and compute the image of the unitary homomorphism. As expected, the order of $2$-torsion in Theorem \ref{ThmE} is closely related to the order of $2$-torsion in their calculations. 
\end{remark*}

\begin{remark*}
The fiber sequence \eqref{Eqn:ThmE} will be constructed over general base fields in forthcoming work of Bachmann-Hopkins. 
\end{remark*}

For our next application, recall that $\eta \in \pi_{1,1}(S^{0,0})$ is not nilpotent. The $\eta$-periodic motivic stable stems have been analyzed frequently in the past decade \cite{AM17}\cite{GI16}\cite{GI15}\cite{RO19}\cite{Ron16}\cite{Wil18}. Using the $kq$-resolution, we prove the following theorem, recovering work of Andrews-Miller \cite{AM17}. 

\begin{thmx}[$\eta$-periodic stable stems, Theorem \ref{Thm:EtaInv}]\label{ThmC}
The $\eta$-periodic motivic stable stems are given by
$$\pi_{**}(\eta^{-1}S^{0,0}) \cong \f_2[h_1^{\pm 1}, v_1^4]\{x,y\}$$
where $|\eta| = (1,1)$, $|v_1^4| = (8,4)$, $|x| = (0,0)$, and $|y| = (8,5)$. 
\end{thmx}

Theorems \ref{ThmD}-\ref{ThmC} provide a foundation for understanding deeper chromatic phenomena. Periodicity is well-understood in the classical \cite{DHS88}\cite{Rav16} and abelian group-equivariant \cite{BS17}\cite{BHNNNS17} stable homotopy theory, but it remains quite mysterious in the motivic setting. 

The Ravenel Conjectures \cite{Rav84} sheperded classical chromatic homotopy theory through its infancy. Guided by our calculations above, we develop motivic analogs of Ravenel's Telescope and Smashing Conjectures over algebraically closed fields. To that effect, we define chromatic Bousfield localization functors $L_{mn}$ using the motivic Morava K-theories constructed by Krause \cite{Kra18} and their corresponding finite Bousfield localizations $L_{mn}^f$. 

\begin{conjecture*}[Motivic Telescope Conjecture, Conjecture \ref{Conj:MotTCLF}]
The natural transformation $L^f_{mn} \to L_{mn}$ is an equivalence. 
\end{conjecture*}

\begin{conjecture*}[Motivic Smashing Conjecture, Conjecture \ref{Conj:MotSC}]
For each $(m,n)$, the localization functor $L_{mn}$ is smashing, i.e. $L_{mn} X \simeq X \wedge L_{mn} S^{0,0}$. 
\end{conjecture*}

For the pairs $(m,n) = (1,0)$ and $(m,n) = (1,-1)$, we resolve these conjectures up to a question about motivic Bousfield classes. 

\subsection{Analysis of the $kq$-resolution}

We now summarize the analysis of the $kq$-resolution which is used to prove Theorems \ref{ThmA}-\ref{ThmC}. 

Theorem \ref{ThmA} follows from considering the $kq$-based Adams resolution and the associated $kq$-based Adams spectral sequence. The description of the $E_1$-term is immediate from the construction, convergence of the uncompleted spectral sequence \eqref{Eqn:ThmA} follows from a connectivity argument using \cite{RSO19}, and convergence of the completed spectral sequence \eqref{Eqn:ThmA2} follows from \cite{Man18}. Theorem \ref{ThmB} is obtained from Theorem \ref{ThmA} by careful bookkeeping and connectivity results from \cite{RSO19}. 

Theorems \ref{ThmD}-\ref{ThmC} are proven using the following main theorem about \eqref{Eqn:ThmA2}. Again, fix an algebraically closed field $F = \overline{F}$ of characteristic 0. 

\begin{thmx}[Main Theorem, Theorem \ref{Thm:kqLines}]\label{ThmF}
The completed $kq$-resolution \eqref{Eqn:ThmA2} satisfies the following:
\begin{enumerate}
\item The $0$-line is given by 
$$E^{0,*,*}_\infty \cong \m_2[h_0,h_1,v_1^4]/(h_0h_1, h_0v^4_1, \tau h_1^3)$$
where $|h_0| = (0,0)$, $|h_1| = (1,1)$, and $|v^4_1| = (8,4)$. 
\item The $1$-line is given by 
$$E^{1,*,*}_\infty \cong \bigoplus_{k \geq 0} \Sigma^{4k} \z/2^{\rho(k)}[\tau] \oplus \m_2[h_1,v^4_1]/(h_1^3 \tau)\{y\},$$
where $|y| = (9,5)$ and $\rho(k)$ is the $2$-adic valuation of $8k$. All of these classes are $v_1$-periodic. 
\item $E^{n,t,u}_\infty = 0$ whenever $6n > t+7$.
\end{enumerate}
\end{thmx}

To prove Theorem \ref{ThmF}, we begin with the $E_1$-page of \eqref{Eqn:ThmA2}. The $0$-line over $\Spec(\c)$, $\pi_{**}^\c(kq)$,  was calculated by Isaksen-Shkembi \cite{IS11}. To obtain analogous results over algebraically closed fields, we extend results of Wilson-{\O}stv{\ae}r \cite{WO17} to show $\pi_{**}^F(kq) \cong \pi_{**}^\c(kq)$. The $1$-line $E_1^{1,*,*} = \pi_{**}^F(kq \wedge \overline{kq})$ is (the augmentation ideal of) the ring of $kq$-cooperations. These were studied after rationalization over fields of characteristic not two by Ananyevskiy \cite{Ana17} and $\eta$-locally by R{\"o}ndigs \cite{Ron16}. Our analysis yields the following theorem over algebraically closed fields. 

\begin{thmx}[$1$-line, Theorem \ref{Cor:AmodmodA(1)Decomp} and Theorem \ref{Lem:ExtOfHZn}]\label{ThmG}
There is an isomorphism of bigraded groups
$$\pi_{**}(kq \wedge kq) \cong \bigoplus_{i \geq 0} \Sigma^{4i,2i} \Ext^{***}_{A(1)^\vee}(\m_2,H\underline{\z}_i)$$
where $H\underline{\z}_i$ is the $i$-th motivic integral Brown-Gitler module. The right-hand side may be described explicitly, modulo $v_1$-torsion, in terms of suspensions of $kq$ and Adams covers of $bo$. 
\end{thmx}

We use this to compute the higher lines of \eqref{Eqn:ThmA2} in Section \ref{Sec:E2}. Theorem \ref{ThmG} is a motivic analog of classical results of Mahowald \cite{Mah81} and Behrens-Ormsby-Stapleton-Stojanoska \cite{BOSS19}, but it is far from a corollary. In fact, we run into several surprising technical hurdles while proving Theorem \ref{ThmG} which do not appear in \cite{BOSS19}\cite{Mah81}. 

In particular, the reader familiar with Mahowald's work \cite{Mah81} might suspect that Theorem \ref{ThmG} can be proven using similar techniques. However, Mahowald's results rely heavily on the existence of integral Brown-Gitler spectra, and these currently have no motivic analog.\footnote{The construction of integral motivic Brown-Gitler spectra is the subject of work in progress by the authors.} Thus we turn to the algebraic approach of Behrens-Ormsby-Stapleton-Stojanoska \cite{BOSS19}, which we recall proceeds in three steps:
\begin{enumerate}
\item Define integral Brown-Gitler modules $H\underline{\z}_i^{cl}$ and prove an isomorphism of $A(1)^\vee$-comodules $(A/\kern-0.25em/A(1))^\vee \cong \bigoplus_{i \geq 0} \Sigma^{4i} H\underline{\z}_i^{cl}$. This proves the first part of Theorem \ref{ThmG} via a change of rings isomorphism. 
\item Produce short exact sequences of $A(1)^\vee$-comodules relating $H\underline{\z}_{2i}^{cl}$ and $H\underline{\z}_{2i+1}^{cl}$ to $H\underline{\z}_i^{cl}$. This yields an inductive procedure for studying the $\Ext$-groups of the right-hand side. 
\item Make the base-case $\Ext$ calculations and apply the inductive procedure to complete the proof of the theorem.
\end{enumerate}
Our proof of Theorem \ref{ThmG} generally follows this outline, but each step is more complicated in the motivic setting.

We define integral motivic Brown-Gitler modules $H\underline{\z}_i$ by assigning a weight to each element in the $2$-primary dual motivic Steenrod algebra \cite{HKO17}\cite{Voe03}. Despite the more complicated presentation of the motivic Steenrod algebra than the classical Steenrod algebra \cite{Mil58}, we obtain an isomorphism of $A(1)^\vee$-comodules in Theorem \ref{Cor:AmodmodA(1)Decomp} analogous to (1). Navigating similar technical difficulties, we also produce short exact sequences in Lemma \ref{Lem:SES} relating various Brown-Gitler modules (2).\footnote{Although we work over an algebraically closed field in the proofs of Theorem \ref{Cor:AmodmodA(1)Decomp} and Lemma \ref{Lem:SES}, the results can be obtained over arbitrary base fields where the motivic Steenrod algebra is known.}

The most difficult aspect of our algebraic analysis is the motivic analog of step (3). In the classical setting, $\Ext_{A(1)^\vee}^{**}(\f_2,H\underline{\z}_1^{cl}) \cong \Ext_{A(1)^\vee}^{**}(\f_2, H_*bsp)$, and more generally, the $\Ext$ groups of $H\underline{\z}_i^{cl}$ may be expressed in terms of the $\Ext$ groups of the homology of Adams covers of $bo$ and $bsp$. This drastically simplifies the inductive procedure since smash products of Adams covers are again Adams covers. 

Motivically, we find $\Ext_{A(1)^\vee}^{***}(\m_2,H\underline{\z}_1) \cong \Ext_{A(1)^\vee}^{***}(\m_2,ksp)$, but we cannot generally express the $\Ext$ groups of $H\underline{\z}_i$ using Adams covers of $kq$ and $ksp$. To see why, recall that $\eta \in \pi_{1,1}(kq)$ is detected by a non-nilpotent element $h_1 \in \Ext_{A(1)^\vee}^{1,2,1}(\m_2,\m_2)$. The infinite $h_1$-towers in $\Ext_{A(1)^\vee}^{***}(\m_2,\m_2)$ leave residual infinite $h_1$-towers in the $\Ext$ groups of Adams covers of $kq$ (and $ksp$), but our explicit calculations show that no such towers exist in $\Ext_{A(1)^\vee}^{***}(\m_2,H\underline{\z}_i)$. This leads to a more computationally intensive inductive procedure, but nevertheless we obtain explicit descriptions of the relevant $\Ext$-groups. 

Starting with the description of the $E_1$-page given by Theorem \ref{ThmG}, we obtain Theorem \ref{ThmF} using the following two results. 

\begin{thmx}[Differentials, Theorems \ref{Prop:kqDifferentials} and \ref{Prop:kqDiffl2}]\label{ThmH}
The completed $kq$-resolution \eqref{Eqn:ThmA2} satisfies the following:
\begin{enumerate}
\item Suppose that a $\tau$-torsion free class $x \in E_1^{n,t,u}$ with $n \geq 2$ is a cycle under $d_1$ and is represented in $\pi_{t,u}(kq \wedge \overline{kq}^{\wedge n})$ by an element of $H\f_2$-Adams filtration $f \geq 2$. Then $x$ is a boundary under $d_1$. 

\item Suppose that a $\tau$-torsion class $x \in E_1^{n,t,u}$ with $n \geq 2$ is a cycle under $d_1$ and is represented in $\pi_{t,u}(kq \wedge \overline{kq}^{\wedge n})$ by an element of $H\f_2$-Adams filtration $f \geq 2$. Then $x$ is a boundary under $d_1$.
\item Let $k \geq 1$. The $d_1$-differential $d_1 : \kqE^{0,4k,\ell}_1 \cong \z \to \z \cong \kqE^{1,4k,\ell}_1$ is given by multiplication by $2^{\rho(k)}$ for any $\ell \leq 2k$, where $\rho(k)$ is the $2$-adic valuation of $8k$. 
\end{enumerate}
\end{thmx}

\begin{thmx}[Vanishing, Theorem \ref{Thm:kqVanishing}]\label{ThmI}
The $E_2$-page of the completed $kq$-resolution \eqref{Eqn:ThmA2} has a vanishing line of slope $1/5$. More precisely, one has $E_2^{n,t,*}(S^0) = 0$ if $t+7<6n$. 
\end{thmx}

These results are motivic analogs of results of Mahowald \cite{Mah81} and Beaudry-Behrens-Bhattacharya-Culver-Xu \cite{BBBCX17}, but as in the proof of Theorem \ref{ThmG}, our approach is different.

We prove Theorem \ref{ThmH} using Betti realization and Mahowald's classical calculation \cite{Mah81} of differentials in the $bo$-resolution. There are two subtleties which arise. First, Betti realization only establishes the existence of differentials between $\tau$-torsion free classes, so the motivic lifts of Mahowald's differentials could produce $\tau$-torsion classes in subsequent pages; we rule out this possibility through a delicate argument using $\tau$-periodicity and motivic weight. Second, Mahowald's differentials do not lift to differentials between $\tau$-torsion classes, such as the differentials between infinite $h_1$-towers; we employ multiplicative structure and the lifted differentials between $\tau$-torsion free classes to obtain these. 

The proof of Theorem \ref{ThmI} is the most computationally difficult part of the paper. Our proof generally follows Mahowald's work \cite{Mah81}\cite{Mah84} where the vanishing line in the $bo$-resolution is established by showing that vanishing for the finite complex $A_1 = S^0/(2,\eta,v_1)$ persists through cofiber sequences relating $A_1$ to $S^0/2$ and then $S^0/2$ to $S^0$. In the motivic setting, we must also control the role $\eta$-periodicity and torsion as we pass between finite complexes, so we first pass from $A_1$ to $S^0/(2,\eta)$ and then from $S^0/(2,\eta)$ to $S^0/2$. This extra step is not only necessary motivically, but also greatly reduces the difficulty of Mahowald's arguments. This more straightforward approach proves vanishing above the line $6n > t+7$, which actually sharpens Mahowald's classical vanishing line $6n > t+14$. 

\subsection{Analysis of the $kq$-resolution over other base fields} \label{subsec: other base fields}
Throughout the paper, we have indicated when we know a result on the $2$-complete $kq$-resolution \eqref{Eqn:ThmA2} holds over more general base fields. If a general result is work in progress, we indicate our expectations, and if we expect a general result to be substantially harder to prove than in the algebraically closed case, we explain why. We summarize these remarks here for the reader's convenience:

\begin{enumerate}
\item Much of the analysis of the $1$-line of the $kq$-resolution in Section \ref{Section:Analysis} holds over general base fields, and preliminary computations show that it holds completely over $\Spec(\r)$. See Remarks \ref{Rem:3.14}, \ref{Rem:3.26}, and \ref{Rem:3.39}. 
\item We expect that determining differentials over general base fields will be substantially harder. We explain why and suggest an approach in Remark \ref{Rem:FDiff}. 
\item We do not know how to formulate a motivic Telescope Conjecture at all heights over general base fields because our formulation relies on work of Krause \cite{Kra18} which only holds over algebraically closed fields. We can, however, formulate the conjecture at low heights in greater generality; see Remark \ref{Rem:TelGen}. 
\item We know that the calculations in this paper hold for any algebraically closed field in characteristic 0. Indeed, if $f: F\to L$ is a morphism of algebraically closed fields, then $f^*: \SH(F)\to \SH(L)$ is a strong symmetric monoidal functor, and it follows from \cite[Lemma B.1]{BH17} that this base change functor commutes with taking very effective covers. This implies that there is a morphism of $kq$-resolutions
 	\[
 	\begin{tikzcd}
 		\pi^F_{t,u}(kq\wedge \overline{kq}^{\wedge s})\arrow[d]\arrow[r,Rightarrow] &\pi^F_{t-s,u}(S^{0,0})\arrow[d]\\
 		\pi^L_{t,u}(kq\wedge \overline{kq}^{\wedge s})\arrow[r,Rightarrow] & \pi^L_{t-s,u}(S^{0,0}).
 	\end{tikzcd}
 	\]
 	Moreover, it is an isomorphism on the $E_1$-page since the map 
 	\[
 	H_F^{**}(kq)\to H_L^{**}(kq)
 	\]
 	is an isomorphism of modules over the motivic Steenrod algebra, and hence an isomorphism on their motivic $H\f_2$-Adams spectral sequences. This shows that the map on $kq$-resolution $E_1$-terms is an isomorphism, and hence the $kq$-resolutions are isomorphic.
 	
 	To pass to algebraically closed fields $F$ in positive characteristic, one would consider the zig-zag of morphisms
 	\[
 	F\xleftarrow{\pi} W(F)\xrightarrow{\iota} \overline{K}
 	\]
 	where $W(F)$ is the Witt vectors on $F$ and $K$ is the fraction field on $W(F)$ and the corresponding base change functors in the motivic stable homotopy categories. Unfortunately, the authors do not know whether or not these base change functors commute with very effective covers. In particular, if $KQ_R$ denotes Hermitian $K$-theory in the $R$-motivic stable homotopy category, we do not whether or not $\pi^*\widetilde{f}_0KQ_{W(F)} = \widetilde{f}_0 KQ_F$, and similarly for $\iota^*$. If we knew this, then we would have a zig-zag of morphisms of spectral sequences relating the $kq$-based Adams spectral sequence over $F$ to the one over $\overline{K}$. The results of \cite[Section 6]{WO17} would show that the maps on $E_1$-terms are isomorphism, and hence we could import differentials from the characteristic 0 case. 
\end{enumerate}

\subsection{Outline}
In Section \ref{Sec:kqASS}, we define the $kq$-resolution and study its convergence (Theorem \ref{ThmA}). We apply the integral $kq$-resolution towards the calculation of low-dimensional Milnor-Witt stems (Theorem \ref{ThmB}) and explain the connection between the $p$-complete $kq$-resolution over algebraically closed fields and the classical $bo$-resolution of \cite{Mah81}. 

After Sectoin \ref{Sec:kqASS}, we always work over an algebraically closed field of characteristic zero. In Section \ref{Section:Analysis}, we analyze the $E_1$-page of the $2$-complete $kq$-resolution. We define motivic integral Brown-Gitler modules using a filtration of the dual motivic Steenrod algebra and verify that $A/\kern-0.25em/A(1)$ decomposes as a sum of motivic integral Brown-Gitler modules. We then identify the cohomology of these motivic integral Brown-Gitler modules with combinations of Adams covers of $bo$, $bsp$, $kq$, and $ksp$. We use these computations to prove that the motivic Adams spectral sequence \cite{DI10} converging to $\pi_{**}(kq \wedge kq)$ collapses at $E_2$. 

In Section \ref{Sec:E2}, we study the differentials and vanishing regions of the $kq$-resolution. We lift Mahowald's differentials in the $bo$-resolution to differentials between $\tau$-torsion free classes in the $kq$-resolution using Betti realization and propogate them to $\tau$-torsion classes using multiplicative structure. We then use our algebraic analysis from Section \ref{Section:Analysis} along with the calculated differentials to prove a vanishing region in the $E_2$-page of the $kq$-resolution (Theorem \ref{ThmI}). 

In Section \ref{Sec:MR}, we state our main computational result, Theorem \ref{ThmF}, which follows from the calculations in Sections \ref{Section:Analysis} and \ref{Sec:E2}. We use it to compute the $v_1$-periodic motivic stable stems (Theorem \ref{ThmD}) and to recalculate the $\eta$-periodic motivic stable stems (Theorem \ref{ThmC}) over algebraically closed fields. Our $v_1$-periodic calculations are new, and our $\eta$-periodic calculations recover work of Andrew-Miller \cite{AM17} which confirmed a conjecture of Guillou-Isaksen \cite[Conj. 1.3(b)]{GI15}.

In Section \ref{Sec:Tel}, we use the results of the previous sections to formulate and provide evidence for motivic Telescope and Smashing Conjectures (Conjectures \ref{Conj:MotTCLF} and \ref{Conj:MotSC}). We reinterpret our computations in terms of certain finite localizations and describe the remaining steps needed to prove some cases of the motivic Telescope Conjecture. 

In Section \ref{Sec:JHom}, we present a model for the connective motivic $j$ spectrum and identify its homotopy with the $v_1$-periodic motivic stable stems (Theorem \ref{Thm:Jo}). Our definition relies on further theoretical development of the work of Gheorghe-Isaksen-Krause-Ricka \cite{GIKR18} as well as an extension to algebraically closed fields. 

\subsection*{Conventions} Except in Section \ref{Sec:kqASS}, we will always work in the category of $2$-complete cellular motivic spectra over an algebraically closed field of characteristic zero. We will often abbreviate $\Ext_{A^\vee}(\m_2, C)$ by $\Ext_{A^\vee}(C)$. In case there is any ambiguity of which spectral sequence is being used, we use a left-hand superscript to indicate the spectral sequence. For example, we will write $\kqE_r^{n,t,u}(X)$ for the $r$-th page of the $kq$-based Adams spectral sequence for $X$. 

\subsection{Acknowledgements}
The authors thank Tom Bachmann, Mark Behrens, Prasit Bhattacharya, Jeremy Hahn, Dan Isaksen, Achim Krause, Oliver R{\"o}ndigs, Markus Spitzweck, and Mura Yakerson for helpful discussions. We further thank Dan Isaksen and Achim Krause for their comments on an earlier version. Additionally, we are indebted to two anonymous referees for many helpful comments, corrections, and suggestions. This project began at the Newton Institute Workshop on Derived Algebraic Geometry and Chromatic Homotopy Theory in 2018. Section \ref{Sec:Tel} benefited tremendously from the authors' visit to the Mathematisches Forschungsinstitut Oberwolfach for the Arbeitsgemeinschaft on Elliptic Cohomology according to Lurie and the second author's visit to the Institut f{\"u}r Mathematik der Universit{\"a}t Osnabr{\"u}ck. We thank all institutes, organizers, and hosts for their hospitality. The second author was partially supported by NSF grant DMS-1547292.


\section{The $kq$-based Adams spectral sequence}\label{Sec:kqASS}

Let $F$ be a field of characteristic not two and let $kq$ denote the very effective cover of Hermitian K-theory $KQ$ defined in \cite{ARO17}. In this section, we define the $kq$-resolution as the $kq$-based motivic Adams spectral sequence and discuss convergence of both the integral and $p$-complete versions. The results are recorded in Theorem \ref{Thm:kqASS}. We then apply the integral $kq$-resolution towards the computation of low-dimensional Milnor-Witt stems of $F$. 

\subsection{Construction and convergence of the $kq$-resolution}

The motivic Adams spectral sequence \cite{DI10}\cite{Mor99} and motivic Adams-Novikov spectral sequence \cite{HKO11} have been studied extensively in computational motivic stable homotopy theory. More generally, the $E$-based motivic Adams spectral sequence may be defined for any motivic ring spectrum $E$; see for example \cite[Sec. 6]{Man18}. 

The canonical $kq$-Adams resolution of the sphere is given by
\begin{equation}\label{Eqn:AdamsRes}
\begin{tikzcd}
	S^{0,0}\arrow[d] & \Sigma^{-1,0}\overline{kq}\arrow[d] \arrow[l]& \arrow[l] \Sigma^{-2,0}\overline{kq}^{\wedge 2}\arrow[d] & \cdots \arrow[l] \\
	kq & \Sigma^{-1,0}kq\wedge \overline{kq} & \Sigma^{-2,0}kq\wedge \overline{kq}^{\wedge 2} & 
\end{tikzcd}
\end{equation}
where $\overline{kq}$ denotes the cofiber of the unit map $S^{0,0} \to kq$. The canonical $p$-complete $kq$-Adams resolution is defined by replacing $kq$ by $kq_p^\wedge$. 

\begin{thm}\label{Thm:kqASS}
There is a strongly convergent spectral sequence, the \emph{$kq$-resolution}, of the form
\begin{equation}\label{Eqn:ZkqRes}
E_1^{n,t,u} = \pi_{t,u}^F(kq \wedge \overline{kq}^{\wedge n}) \Rightarrow \pi_{t-n,u}^F(S^{0,0}).
\end{equation}
The $d_r$-differentials have the form
$$d_r : E_r^{n,t,u} \to E_r^{n+r, t+r-1,u}.$$
There is an analogous strongly convergent spectral sequence of the form
\begin{equation}\label{Eqn:pkqRes}
E_1^{n,t,u} = \pi_{t,u}^F(kq_p^\wedge \wedge \overline{kq_p^\wedge}^{\wedge n}) \Rightarrow \pi_{t-n,u}^F((S^{0,0})_p^\wedge).
\end{equation}
\end{thm}

\begin{proof}
Applying $F$-motivic stable homotopy groups to \eqref{Eqn:AdamsRes} produces a spectral sequence
\begin{equation}
E_1^{n,t,u} = \pi_{t,u}(kq\wedge \overline{kq}^{\wedge n}) \Rightarrow \pi_{t-n,u}^F((S^{0,0})^{\wedge}_{kq})
\end{equation}
where $(S^{0,0})^{\wedge}_{kq}$ denotes the $kq$-nilpotent completion. Analogously, we obtain from a spectral sequence
\begin{equation}
E_1^{n,t,u} = \pi^F_{t,u}(kq_p^\wedge \wedge \overline{\kq_p^\wedge}^{\wedge n}) \Rightarrow \pi_{t-n,u}^F((S^{0,0})^\wedge_{kq_p^\wedge}).
\end{equation}
The differentials have the specified form by construction, so we only need to identify the abutments with those appearing in \eqref{Eqn:ZkqRes} and \eqref{Eqn:pkqRes} to prove the theorem.

To identify the $kq$-nilpotent completion of $S^{0,0}$, recall from \cite{Bac17} that the unit map $S^{0,0} \to kq$ induces an isomorphism on $\Pi_0$, so $\overline{kq}$ is $1$-connected in the homotopy t-structure on $SH(F)$ and $kq \wedge \overline{kq}^{\wedge n}$ is $n$-connected in the homotopy t-structure. Therefore we have $(S^{0,0})_{kq}^\wedge \simeq S^{0,0}$ and the abutment of \eqref{Eqn:ZkqRes} may be identified with $\pi_{t-n,u}^F(S^{0,0})$.\footnote{We refer the reader to \cite[Sec. 5]{SO12} for a more detailed convergence argument along these lines.} For the $p$-complete $kq$-resolution, convergence follows from applying \cite[Thms. 1.0.1 and 7.3.4]{Man18} to see that $(S^{0,0})^{\wedge_{kq^\wedge_p}} \simeq (S^{0,0})_p^\wedge$. 
\end{proof}

\subsection{Applications of the integral $kq$-resolution}

We now record some applications of the integral $kq$-resolution \eqref{Eqn:ZkqRes} to the calculation of low-dimensional Milnor-Witt stems. 

In the proof of Theorem \ref{Thm:kqASS}, we used the fact that $\Pi_0(\overline{kq}) = 0$ to show that $kq \wedge \overline{kq}^{\wedge n}$ is $n$-connected in the homotopy $t$-structure on $SH(F)$. In fact, a stronger statement can be made about the connectivity of $\overline{kq}$. 

\begin{thm}\cite{RSO19}
The map $\Pi_1^F(S^{0,0}) \to \Pi_1^F(kq)$ is surjective. In particular, $\Pi_1^F(\overline{kq}) = 0$. 
\end{thm}

\begin{proof}
The first claim follows from forthcoming work of Bachmann-Hopkins building on \cite{RSO19} and the second claim follows from the long exact sequence
$$\cdots \to \Pi_1^F(S^{0,0}) \to \Pi_1^F(kq) \to \Pi_1^F(\overline{kq}) \to \Pi_0^F(S^{0,0}) \overset{\cong}{\to} \Pi_0^F(kq) \to 0.$$
\end{proof}

\begin{cor}\label{Cor:Conn}
The motivic spectrum $kq \wedge \overline{kq}^{\wedge n}$ is $2n$-connective in the homotopy $t$-structure. In other words, we have $\Pi_i^F(kq \wedge \overline{kq}^{\wedge n}) = 0$ for $i < 2n$. 
\end{cor}

\begin{thm}\label{Thm:MWStems}
There are isomorphisms of graded abelian groups
\begin{equation}\label{Eqn:Pi0}
\Pi_0^F(S^{0,0}) \cong \Pi_0^F(kq),
\end{equation}
and a short exact sequence of graded abelian groups
\begin{equation}\label{Eqn:Pi1}
0 \to \Pi_2((kq \wedge \overline{kq})/d_1) \to \Pi_1^F(S^{0,0}) \to \Pi_1^F(kq)  \to 0.
\end{equation}
More generally, $\Pi^F_k(S^{0,0})$ is detected in $E_\infty^{n,t,u}$ of \eqref{Eqn:ZkqRes} in the range of tridegrees $(n,t,u)$ where $t-n-u=k$ and $n \leq k$. 
\end{thm}

\begin{proof}
We begin by proving the general statement about $\Pi_k^F(S^{0,0})$. First, observe that the tridegrees $(n,t,u)$ for which $E_\infty^{n,t,u}$ can contribute to $\Pi_k^F(S^{0,0})$ must satisfy $t-n-u = k$, or equivalently, $t-u = n+k$. On the other hand, Corollary \ref{Cor:Conn} implies that
$$E_1^{n,t,u} = \pi_{t,u}^F(kq \wedge \overline{kq}^{\wedge n}) \subseteq \Pi_{t-u}^F(kq \wedge \overline{kq}^{\wedge n}) = 0$$
if $t-u < 2n$. Therefore the tridegrees with possible nonzero contributions to $\Pi_k^F(S^{0,0})$ are precisely those satisfying $t-n-u = k$ and $t-u = n+k \geq 2n$; the last inequality may be rewritten as $n \leq k$. 

Applying the general statement with $k=0$, we find that $\Pi_0(S^{0,0})$ is detected in the groups $\bigoplus_{t}E^{0,t,t}_\infty$. Classes in these groups cannot be the targets of differentials because they are concentrated in the $0$-line. Moreover, classes in these tridegrees are permanent cycles: the $d_r$-differentials decrease Milnor-Witt stem by $1$, and since $\Pi_{-1}^F(kq \wedge \overline{kq}^{\wedge n}) = 0$ for all $n$, there are no possible nonzero targets for differentials from $E_r^{0,t,t}$. We conclude that 
$$\Pi_0^F(S^{0,0}) \cong \bigoplus_{t \in \z} E_\infty^{0,t,t} \cong \bigoplus_{t \in \z} E_1^{0,t,t} \cong \bigoplus_{t \geq 0} \pi_{t,t}(kq) = \Pi_0(kq)$$
which proves \eqref{Eqn:Pi0}.

To prove \eqref{Eqn:Pi1}, we apply the general statement when $k=1$ to see that the contributions to $\Pi_1^F$ are $\bigoplus_{t \in \z} E_\infty^{0,t+1,t}$ and $\bigoplus_{t \in \z} E_\infty^{1,t+2,t}$. As in the $k=0$ case, the classes in $E_1^{0,t+1,t}$ are not the targets of differentials. Moreover, they are permanent cycles since their potential targets lie in the $0$-th Milnor-Witt stems of the higher lines, all of which are zero. Therefore $E_\infty^{0,t+1,t} \cong E_1^{0,t+1,t}$ and $\bigoplus_{t \in \z} E_\infty^{0,t+1,t} \cong \bigoplus_{t \in \z} E_1^{0,t+1,t} = \Pi_1(kq)$.

Similar considerations show that any $x \in E_1^{1,k+2,k}$ is a permanent cycle, so we have
$$\bigoplus_{k \in \z} E_\infty^{1,k+2,k} = \bigoplus_{k \in \z} \coker(d_1 : E_1^{0,k+2,k} \to E_1^{1,k+2,k}).$$
This completes the proof. 
\end{proof}

\begin{rem2}
Combining Theorem \ref{Thm:MWStems} with the main result of \cite{RSO19}, we find that 
$$\bigoplus_{k \in \z} \coker(d_1 : E_1^{0,k+2,k} \to E_1^{1,k+2,k}) \cong K^M_*(F)/24.$$ 
Work in progress suggests that the left-hand side can be identified with the cokernel of a map $\Pi_2(kq) \to \Pi_0(ksp)$ after $2$-completion. 
\end{rem2}

\subsection{Comparison with the $bo$-resolution}

We now implicitly $2$-complete everything. The $kq$-resolution is the motivic analog of the $bo$-resolution \cite{BBBCX17}\cite{LM87}\cite{Mah81}\cite{Mah84}. Let $\overline{bo}$ denote the cofiber of the unit map $S^0\to bo$. The canonical $bo$-based Adams resolution of the sphere is given by
\[
\begin{tikzcd}
	S^0\arrow[d] & \Sigma^{-1}\overline{bo}\arrow[d] \arrow[l]& \arrow[l] \Sigma^{-2}\overline{bo}^{\wedge 2}\arrow[d] & \cdots \arrow[l] \\
	bo & \Sigma^{-1}bo\wedge \overline{bo} & \Sigma^{-2}bo\wedge \overline{bo}^{\wedge 2} & 
\end{tikzcd}
\]
where $\overline{bo}$ is the cofiber of the unit map $S^0 \to bo$. Applying classical stable homotopy groups gives rise to a spectral sequence 
\[
E_1^{n,t} = \pi_t (bo\wedge \overline{bo}^{\wedge n})\implies \pi_{t-n}(S^0)^{\wedge}_{bo}. 
\]
It follows from a theorem of Bousfield \cite{Bou79} that $(S^0)^\wedge_{bo}\simeq S^0$. 

We can compare the $kq$- and $bo$-resolutions using the following theorem of Ananyevskiy-R{\o}ndigs-{\O}stv{\ae}r.

\begin{lem}\cite[Lem. 2.13]{ARO17} 
The Betti realization of $kq$ is $bo$. 
\end{lem}

Since Betti realization preserves cofiber sequences and is strong symmetric monoidal \cite{HO16}, we obtain the following corollary which will be used frequently throughout the sequel. 

\begin{cor}
	Betti realization takes the $kq$-resolution to the $bo$-resolution. In particular, Betti realization induces a multiplicative map of spectral sequences
	\[
	E_r^{s,t,*}\to E_r^{s,t}.
	\]
\end{cor}

\section{Analysis of the $E_1$-page}\label{Section:Analysis}

Just as in the classical situation, the spectrum $kq$ does not satisfy Adams' flatness condition. Consequently, the $E_2$-page of the $kq$-resolution does not have the usual algebraic description as an $\Ext$-group and we need to analyze its $E_1$-page. Since the homotopy groups $\pi_{**}(kq\wedge \overline{kq}^{\wedge s})$ make up the $E_1$-term of the $kq$-Adams spectral sequence, we begin by studying the $1$-line $\pi_{**}(kq\wedge \overline{kq})$. These constitute the $1$-line, which together with the known $0$-line, will be shown to detect the $v_1$-periodic and $\eta$-periodic stable stems. 

In this section and subsequent sections, we we will work over the field $\c$ of complex numbers for concreteness. By the remarks in Subsection \ref{subsec: other base fields}, what is shown below also holds over any algebraically closed field of characteristic 0.

\subsection{The $E_1$-page of the $bo$-resolution}

To start, we recall the analysis of the $E_1$-page of the $bo$-resolution. By \cite[Thm. 2.4]{Mah81}, there is a homotopy equivalence 
$$bo \wedge bo \simeq \bigvee_{i = 0}^\infty \Sigma^{4i} bo\wedge H\z_i^{cl}$$
where $H\z_i^{cl}$ is the $i$-th classical Brown-Gitler spectrum. Replacing the right-hand copy of $bo$ by $\overline{bo}$ gives the homotopy equivalence
$$bo \wedge \overline{bo} \simeq \bigvee_{i = 1}^\infty \Sigma^{4i}bo\wedge  H\z_i^{cl}.$$
Mahowald extends this description of the $1$-line of the $E_1$-page of the $bo$-resolution to the entire $E_1$-page via the K\"unneth isomorphism in the proof of \cite[Thm. 5.11]{Mah81}. In particular, one has
$$bo \wedge \overline{bo}^{\wedge n} \simeq \bigvee_{i_1,\ldots,i_n} \Sigma^{4(i_1+\cdots + i_n)}bo \wedge H\z_{i_1}^{cl} \wedge \cdots \wedge H\z_{i_n}^{cl}$$
where $i_j >0$ for all $1 \leq j \leq n$. Our goal in the remainder of the section is to obtain an analogous decomposition in the $\c$-motivic setting. 

\subsection{Motivic Brown-Gitler modules}\label{Section:BrownGitler}
We begin by recalling some facts about the motivic dual Steenrod algebra $A^\vee$ and the very effective cover of Hermitian K-theory $kq$ from ~\cite{Voe03} and ~\cite{ARO17}, respectively. We then define motivic integral Brown-Gitler modules following ~\cite{BOSS19}.

Let $\m_2 := \f_2[\tau]$, $|\tau| = (0,-1)$, denote the mod two motivic homology of a point. 

\begin{thm}\cite[Sec. 12]{Voe03}\label{thm: dual steenrod algebra}
The dual motivic Steenrod algebra $A^\vee$ is given by
$$\m_2[\bar{\xi}_1,\bar{\xi}_2,\ldots,\bar{\tau}_0,\bar{\tau}_1,\ldots]/(\bar{\tau}^2_i = \tau \bar{\xi}_{i+1}),$$
where
\begin{align*}
	|\oxi_i| = (2^{i+1}-2, 2^i-1), & &|\otau_i| = (2^{i+1}-1, 2^i-1).
\end{align*}
The coproduct is determined by 
\begin{align*}
	\psi(\oxi_k)= \sum_{i+j=k} \oxi_i\otimes \oxi_j^{2^i}, & & \psi(\otau_k)= \sum_{i+j=k}\otau_i\otimes \oxi_j^{2^i}+ 1\otimes \otau_k.
\end{align*}
\end{thm}

\begin{defin}(compare with \cite[Def. 3.2.1]{Gre12} and \cite[Def. 2.4]{IS11})
For any $n\geq 0$, define an ideal $I(n) \subseteq A^\vee$ by
$$I(n) := (\bar{\tau}_{n+1},\bar{\tau}_{n+2},\ldots,\bar{\xi}_1^{2^n},\bar{\xi}_2^{2^{n-1}},\ldots,\bar{\xi}_n^2,\bar{\xi}_{n+1},\ldots).$$
We define the quotient $A^\vee(n) \subseteq A^\vee$ by
$$A(n)^\vee := A^\vee / I(n).$$
For $n \geq 0$, define $A(n)$ to be the $\m_2$-subalgebra of $A$ defined by
$$A(n) := \langle Sq^1, Sq^2, \ldots, Sq^{2^n} \rangle,$$
so
$$(A /\kern-0.25em/ A(n))^\vee \cong \m_2[\bar{\xi}_1^{2^{n}},\bar{\xi}_2^{2^{n-1}},\ldots,\bar{\tau}_{n+1},\bar{\tau}_{n+2},\ldots]/(\bar{\tau}^2_i = \tau \bar{\xi}_{i+1}).$$
\end{defin}

Observe that $H^{**}(H\z) \cong A/\kern-0.25em/A(0)$, and dually, $H_{**}(H\z) \cong (A/\kern-0.25em/A(0))^\vee$. We have the following:

\begin{thm}\cite{ARO17}
The mod two motivic cohomology of $kq$ is given by 
$$H^{**}(kq) \cong A/\kern-0.25em/ A(1).$$
\end{thm}

We obtain the following by dualizing. 

\begin{cor}\label{Cor:H_*kq}
The mod two motivic homology of $kq$ is given by
$$H_{**}(kq) \cong (A /\kern-0.25em/ A(1))^\vee \cong \m_2[\bar{\xi}_1^2,\bar{\xi}_2,\ldots,\bar{\tau}_2,\bar{\tau}_3,\ldots]/(\bar{\tau}_i^2 = \tau \bar{\xi}_{i+1}).$$
\end{cor}

We will also need to know the structure of $H_{**}(kq)$ as a comodule over the dual Steenrod algebra. Note that $(A/\kern-0.25em/ A(1))^\vee$ is a subalgebra of the dual Steenrod algebra. The coaction is then obtained by restriction, yielding 
\[
\psi: (A/\kern-0.25em/A(1))^\vee \to A^\vee \otimes (A/\kern-0.25em/ A(1))^\vee.
\] 
In the next subsection, we will calculate the homotopy groups of $kq\wedge kq$ via the motivic Adams spectral sequence. The $E_2$-term of this spectral sequence is given by $\Ext_{A^\vee}((A/\kern-0.25em/A(1))^\vee\otimes (A/\kern-0.25em/ A(1))^\vee)$, where our convention is that 
$$\Ext_{A^\vee}(M) := \Ext_{A^\vee}(\m_2,M).$$
A change-of-rings isomorphism yields the $E_2$-term
\[
E_2 \cong \Ext_{A(1)^{\vee}}((A/\kern-0.25em/A(1))^\vee) \Rightarrow \pi_{**}(kq \wedge kq).
\]
Thus we want to know $(A/\kern-0.25em/A(1))^\vee$ as a comodule over $A(1)^\vee$. It follows from the definitions that there is a projection $\pi: A^\vee\to A(1)^\vee$. The $A(1)^\vee$-coaction on $A/\kern-0.25em/A(1)^\vee$ is given by the composite
\[
\begin{tikzcd}
	(A/\kern-0.25em/A(1))^\vee\arrow[r, "\psi"] & (A^\vee\otimes A/\kern-0.25em/A(1))^\vee \arrow[r, "\pi\otimes 1"] & A(1)^\vee\otimes (A/\kern-0.25em/A(1))^\vee.
\end{tikzcd}
\]
In order to have a practical means of calculating this coaction, we need the following. 

\begin{prop}
	As a Hopf algebra, we have 
	\[
	A(1)^\vee\cong \m_2[\oxi_1, \otau_0, \otau_1]/(\otau_0^2-\tau\oxi_1, \otau_1^2, \oxi_1^2)
	\]
\end{prop}

Observe also that $A/\kern-0.25em/A(1)^\vee$ is a comodule algebra (since $kq$ is an commutative ring spectrum). Thus the coaction is completely determined by its values on $\oxi_i$ and $\otau_j$. 

\begin{cor}\label{cor: A(1) coaction}
	The $A(1)^\vee$-coaction on $A/\kern-0.25em/A(1)^\vee$ is completely determined by the following formulas
	\begin{align*}
		\psi(\oxi_1^2)&= 1\otimes \oxi_1^2, & \\
		\psi(\oxi_k) &= 1\otimes \oxi_k  + \oxi_1\otimes \oxi_{k-1}^2& \text{ for } k>1,\\
		\psi(\otau_k)&= 1\otimes \otau_k+ \otau_0\otimes \oxi_k+ \otau_1\otimes \oxi_{k-1}^2 & \text{ for } k>1.
	\end{align*}
	
\end{cor}

We will also need to use the $A(1)^\vee$-comodule structure of $A/\kern-0.25em/A(0)^\vee$. Computing this coaction is very similar to the analysis above; we record the result here for the reader's convenience:

\begin{prop}\label{prop: A(1)-comodule structure AmodmodA(0)}
	The subalgebra $A/\kern-0.25em/A(0)^\vee$ is an $A(1)^\vee$-comodule algebra and its coaction is completely determined on its algebra generators. These are given by the following formulas:
	\begin{align*}
		\psi(\oxi_1) &= \oxi_1\otimes 1+1\otimes \oxi_1, & \\
		\psi(\oxi_k) &= 1\otimes \oxi_k +\oxi_1\otimes \oxi_{k-1}^2 & k>2, \\
		\psi(\otau_1) &= \otau_1\otimes 1+ \otau_0\otimes \oxi_1+1\otimes \otau_1, & \\
		\psi(\otau_k) &= \otau_0\otimes \oxi_k+ \otau_1\otimes\oxi_{k-1}^2+ 1\otimes \otau_k & k>1 .
	\end{align*}
\end{prop}

We can now define the \emph{Mahowald filtration} of $(A/\kern-0.25em/A(1))^\vee$. 

\begin{defin}
We define the \emph{Mahowald weight} of the multiplicative generators of $A^\vee$ by setting
$$wt(\bar{\tau}_i)= wt( \bar{\xi}_i ) = 2^i, \quad wt( \tau ) = 0.$$
We extend this definition to $A^\vee$ by defining $wt(x \cdot y ) = wt(x) + wt(y)$. We define the Mahowald weight on the subalgebras $(A/\kern-0.25em/A(n))^\vee$ in the obvious way. 
\end{defin}
Observe that for any monomial $m\in (A/\kern-0.25em/A(n))^\vee$, the weight of $m$ is divisible by $2^{n+1}$. So, all monomials of $(A/\kern-0.25em/A(0))^\vee$ have weight divisible by 2, and all monomials in $(A/\kern-0.25em/A(1))^\vee$ have weight divisible by 4. 

This definition is partially motivated by the following observation.

\begin{lem}[compare with \cite{Cul17}]\label{lem: coaction preserves mahowald weight}
The $A(1)^\vee$-coaction on $(A/\kern-0.25em/A(1))^\vee$ preserves Mahowald weight.
\end{lem}

\begin{proof}
This is a trivial consequence of Corollary \ref{cor: A(1) coaction}.
\end{proof}

We conclude this subsection by defining motivic analogs of Brown-Gitler modules. 

\begin{defin}
The \emph{$i$-th integral Brown-Gitler module} $H\underline{\z}_i \subseteq (A/\kern-0.25em/A(0))^\vee$  is defined by setting
$$H \underline{\z}_i := \{x \in (A/\kern-0.25em/A(0))^\vee : wt(x) \leq 2i\}.$$
The \emph{$i$-th $kq$-Brown-Gitler module} $\underline{kq}_i$ is defined by setting 
$$\underline{kq}_i := \{ x \in (A/\kern-0.25em/A(1))^\vee : wt(x) \leq 4i\}.$$
\end{defin}

\begin{lem}\label{lem: HZisubcomodules}
	The integral Brown-Gitler modules $H\ul{\z}_i$ are $A(1)^\vee$ subcomodules of $(A/\kern-0.25em/ A(0))^\vee$.
\end{lem}

\begin{exm}\label{exm: motivic HZ1}
We have $H\uz_1 \cong \m_2\{1,\bar{\xi}_1, \bar{\tau}_1 \}$ with nontrivial $A(1)$-action given by $Sq^2(1) = \bar{\xi}_1$ and $Sq^1(\bar{\xi}_1) = \bar{\tau}_1$. This module is realized as the motivic cohomology of a spectrum $H\z_1$ which can be constructed as follows. Let $L^2$ be the (simplicial) $4$-skeleton of the geometric classifying space of the group scheme $\mu_2$ of square roots of unity $B_{gm}\mu_2$ defined in \cite{MV99}. Let $X$ be the cofiber of the inclusion $S^{3,2} \hookrightarrow L^2$. Setting $H\z_1 := \Sigma^{4,2} F(X,S^{0,0})$ gives a spectrum whose motivic cohomology has the desired $A(1)$-module structure. 
\end{exm}

\begin{rem2}\label{Rem:3.14}
	In classical topology, one can construct the $i$th integral Brown-Gitler spectrum $H\ul{\z}^{cl}_i$. These have the property that $H_*(H\ul{\z}^{cl}_i)$ is the $i-th$ integral Brown-Gitler module. In the motivic case, these have been defined at odd primes by R{\"o}ndigs in \cite[Sec. 7]{Ron16} using real Betti realization. 
\end{rem2}

\begin{rem2}
The definitions and results of this section carry over unchanged to general base fields of characteristic not two. Indeed, the motivic Steenrod algebra and its dual are known over arbitrary fields \cite{HKO17}\cite{Voe03}, and we may assign Mahowald weights without change. Further, the isomorphism $H^{**}(kq) \cong A/\kern-0.25em/A(1)$ holds over fields of characteristic not two \cite{ARO17}. Finally, the cell structure of $B_{gm}\mu_2$ used in Example \ref{exm: motivic HZ1} is valid over arbitrary base fields \cite{GHKRO20} so we may construct $H\uz_1$ the same way. 
\end{rem2}

\subsection{The inductive procedure for calculating $kq$-cooperations}\label{sec: kq-cooperations}
Our goal now is to compute $\pi_{**}(kq \wedge kq)$, which will allow us to determine the $1$-line of the $kq$-based Adams spectral sequence. We start by proving the existence of several short exact sequences which give a recursive method for calculating $kq$-cooperations. 

Consider the motivic Adams spectral sequence
$$\Ext^{***}_{A^\vee}(H_{**}(kq \wedge kq)) \Rightarrow \pi_{**}(kq \wedge kq).$$
By the K\"unneth isomorphism for motivic homology, a change-of-rings isomorphism, and Corollary \ref{Cor:H_*kq}, we have
$$\Ext^{***}_{A^\vee}(H_{**}(kq \wedge kq)) \cong \Ext^{***}_{A(1)^\vee}((A/\kern-0.25em/A(1))^\vee).$$
We will now argue that, as in the classical setting, one can decompose $A/\kern-0.25em/A(1)^\vee$ as an $A(1)^\vee$-comodule into an infinite direct sum of suspensions the integral Brown-Gitler modules. Towards this end, we define the following. 

\begin{defin}
	Define $M_1(k)$ to be the subspace of $A/\kern-0.25em/A(1)^\vee$ spanned by the monomials of degree equal to $4k$. Analogously, define $M_0(k)$ to be the subspace of $A/\kern-0.25em/A(0)^\vee$ spanned by monomials of weight $2k$.
\end{defin}

\begin{lem}
	The subspaces $M_1(k)$ are sub-comodules of $A/\kern-0.25em/A(1)^\vee$. Furthermore, there is an isomorphism of $A(1)^\vee$-comodules
	\[
	A/\kern-0.25em/A(1)^\vee\cong \bigoplus_k M_1(k).
	\]
\end{lem}
\begin{proof}
	This is an immediate consequence of Lemma \ref{lem: coaction preserves mahowald weight}.
\end{proof}

Next, we show that the $M_1(k)$ are isomorphic as $A(1)^\vee$-comodules to the integral Brown-Gitler modules (up to a suspension). This argument is an adaptation of the classical ones (c.f. \cite{BHHM08}\cite{BOSS19}\cite{Cul17}).

First, there is an algebra map
\[
\varphi: (A/\kern-0.25em/A(1))^\vee\to (A/\kern-0.25em/A(0))^\vee
\]
defined on generators by 
\[
\varphi(\oxi_k^{2^\ell}):=
\begin{cases}
	\oxi_{k-1}^{2^\ell} & k>1,\\
	1 & k =1,
\end{cases}
\]
and
\[
\varphi(\otau_k)  := \otau_{k-1}.
\]

\begin{lem}
The map $\varphi$ is a map of ungraded $A(1)^\vee$-comodules. 	
\end{lem}
\begin{proof}
	Since both $A/\kern-0.25em/A(1)^\vee$ and $A/\kern-0.25em/A(0)^\vee$ are comodule algebras, it is enough to check that the map $\varphi$ commutes with the coaction on the generators. That this is true follows from direct computation using Corollary \ref{cor: A(1) coaction} and Proposition \ref{prop: A(1)-comodule structure AmodmodA(0)}. 
\end{proof}

\begin{lem}[c.f. \cite{BHHM08}]
	The map $\varphi$ maps the subspace $M_1(k)$ isomorphically onto the $A(1)^\vee$-subcomodules $H\ul{\z}_k$.
\end{lem}
\begin{proof}
	Let $m$ be a generic monomial in $M_1(k)$. Then $m$ is uniquely expressible as $\oxi_1^{2s}x$ for some natural number $s$ and $x$ a monomial in $\mathbb{M}_2[\oxi_2,\oxi_3, \ldots, \otau_2, \otau_3, \ldots]$. Since $wt(m)=4k$, we have that 
	\[
	4k = wt(m) = wt(\oxi_1^{2s})+wt(x) = 4s+wt(x).
	\] 
	Thus $wt(x) = 4(k-s)$. Hence $\varphi$ maps the subspace spanned by monomials of the form $\oxi_1^{2s}x$ isomorphically onto $M_0(k-s)$. Since $H\ul{\z}_k$ is the direct sum of $M_0(j)$ for $0\leq j\leq k$, we have the result.
\end{proof}

\begin{cor}
	The map $\varphi$ induces a graded isomorphism of $A(1)^\vee$-comodules $M_1(k)\cong \Sigma^{4k,2k}H\ul{\z}_k$.
\end{cor}

Combining these lemmata now gives the following decomposition of $A/\kern-0.25em/A(1)^\vee$. 

\begin{thm}\label{Cor:AmodmodA(1)Decomp}
There is an isomorphism of $A(1)^\vee$-comodules
$$(A/\kern-0.25em/A(1))^\vee \cong \bigoplus_{i \geq 0} \Sigma^{4i,2i} H\underline{\z}_i.$$
\end{thm}

Therefore the motivic Adams spectral sequence computing $\pi_{**}(kq \wedge kq)$ has the form
$$E_2 \cong \bigoplus_{i \geq 0} \Ext^{***}_{A(1)^\vee}(\Sigma^{4i,2i} H\underline{\z}_i) \Rightarrow \pi_{**}(kq \wedge kq).$$
Thus, we need to be able to compute the $\Ext_{A(1)^\vee}$-groups of the integral Brown-Gitler modules. Following \cite{BHHM08}\cite{BOSS19}\cite{Cul17}, we will inductively compute these groups by using the following short exact sequences.

\begin{lem}\label{Lem:SES}(compare with \cite[Lem. 2.5]{BOSS19})
There are short exact sequences of $A(1)^\vee$-comodules
$$0 \to \Sigma^{4j,2j} H\underline{\z}_j \to H \underline{\z}_{2j} \to \underline{kq}_{j-1} \otimes (A(1)/\kern-0.25em/A(0))^\vee \to 0,$$
$$0 \to \Sigma^{4j,2j} H\underline{\z}_j \otimes H\underline{\z}_1 \to H\underline{\z}_{2j+1} \to \underline{kq}_{j-1} \otimes (A(1)/\kern-0.25em/A(0))^\vee \to 0.$$
\end{lem}

Before delving into the proof of this theorem, we need some preliminaries. The following is inspired by the technical work in \cite{BHHM08} and \cite{Cul17}. Recall that
\[
A(1)/\kern-0.25em/A(0)^\vee = E(\oxi_1, \otau_1).
\]
Note that there is an isomorphism of $\m_2$-modules
\[
\kappa: A/\kern-0.25em/A(0)^\vee\to A/\kern-0.25em/A(1)^\vee \otimes A(1)//A(0)^\vee \]
\[ \oxi_1^\epsilon\otau_1^\eta m\mapsto m\otimes \oxi_1^\epsilon\otau_1^\eta
\]
where $m$ is a monomial in $A/\kern-0.25em/A(1)^\vee$. While this map is an isomorphism of $\m_2$-modules, it is not an isomorphism of $A(1)^\vee$-comodules (where we endow the right hand side with the diagonal comodule structure), as the following example illustrates. 

\begin{exm}
	On the left-hand side, we have the element $\otau_1$, whose coaction is given by 
	\[
	\alpha(\otau_1) = 1\otimes \otau_1+\otau_0\otimes \oxi_1 + \otau_1\otimes 1. 
	\]
	Under $\kappa$ this is sent to 
	\[
	\kappa\alpha(\otau_1) = 1\otimes 1\otimes \otau_1+ \otau_0\otimes 1\otimes \oxi_1+ \otau_1\otimes 1\otimes 1.
	\]
	On the other hand, $\kappa(\otau_1) = 1\otimes \otau_1$, the coaction of which is 
	\[
	\alpha(1\otimes \otau_1) = 1\otimes 1\otimes \otau_1 + \otau_1\otimes 1\otimes 1
	\]
	since $\otau_1$ is primitive in $A(1)/\kern-0.25em/A(0)^\vee$. 
\end{exm}

While $\kappa$ is not a comodule map, there is a filtration on $A/\kern-0.25em/A(0)^\vee$ which induces a comodule map on the associated graded. Define a filtration $F^jA/\kern-0.25em/A(0)^\vee$ on $A/\kern-0.25em/A(0)^\vee$ by 
\[
F^jA/\kern-0.25em/A(0)^\vee :=\kappa^{-1}\left(\bigoplus_{k\geq j} M_1(k)\otimes A(1)/\kern-0.25em/A(0)^\vee\right).
\]
Clearly, this is a multiplicative filtration on $A/\kern-0.25em/A(0)^\vee$, and furthermore induces a map on the associated graded
\[
E_0\kappa : E_0A/\kern-0.25em/A(0)^\vee\to A/\kern-0.25em/A(1)^\vee\otimes A(1)/\kern-0.25em/A(0)^\vee. 
\]
Since $\kappa$ is clearly an isomorphism of $\f_2$-vector spaces, the map $E_0\kappa$ is also an isomorphism of $\f_2$-vector spaces. 

\begin{prop}
	The map $E_0\kappa$ is an algebra map.
\end{prop}
\begin{proof}
	 Consider two monomials $x, y\in A/\kern-0.25em/A(0)^\vee$ and write them as $x = m\oxi_1^{\epsilon}\otau_1^{\eta}$ and $y = m'\oxi_1^{\varepsilon'}\otau_1^{\eta'}$. Then under $E_0\kappa$, we have
	\[
	x\mapsto m\otimes \oxi_1^\epsilon\otau_1^\eta
	\]
	and
	\[
	y\mapsto m'\otimes \oxi_1^{\epsilon'}\otau_1^{\eta'}.
	\]
	If, say, $\epsilon = \epsilon' = 1$, then $xy$ lives in higher filtration than expected, and so $xy=0$ in the associated graded. This shows that the map $E_0\kappa$ is an algebra map. 
\end{proof}

\begin{prop}
	The map $E_0\kappa$ is a comodule map.
\end{prop}
\begin{proof}
	Since we know that $E_0\kappa$ is an algebra map, and since $E_0A/\kern-0.25em/A(0)^\vee$ and $A/\kern-0.25em/A(1)^\vee\otimes A(1)/\kern-0.25em/A(0)^\vee$ are comodule algebras, it is enough to show that the coaction commutes with $E_0\kappa$ on the set of algebra generators 
	\[
	\{\oxi_1, \oxi_2, \ldots ; \otau_1, \otau_2, \ldots \}.
	\]
	An easy inspection of the formulas for the $A(1)$-coaction gives the result. 
\end{proof}

Following \cite{BHHM08}\cite{BOSS19}\cite{Cul17}, we define
\[
R^j A/\kern-0.25em/ A(0): = A/\kern-0.25em/A(0)^\vee/ F^{j+1}A/\kern-0.25em/A(0)^\vee.
\]
This clearly has an induced filtration and so we get induced maps
\[
E_0\kappa: E^0R^jA/\kern-0.25em/A(0)^\vee\to \ul{kq}_j\otimes A(1)/\kern-0.25em/A(0)^\vee. 
\]

\begin{rem2}\label{Rem:kqMargolis}
	There is a finite filtration on $R^{j}A/\kern-0.25em/A(0)^\vee$. Thus we get an associated spectral sequence 
	\[
	\Ext_{A(1)^\vee}(E_0R^{j}A/\kern-0.25em/A(0)^\vee)\implies \Ext_{A(1)^\vee}(R^jA/\kern-0.25em/A(0)^\vee). 
	\]
	The $E_1$-term is isomorphic to 
	\[
	\Ext_{A(1)^\vee}(E_0R^{j}A/\kern-0.25em/A(0)^\vee)\cong \Ext_{A(0)^\vee}(\ul{kq}_j).
	\]
	We have 
	\[
	\Ext_{A(0)^\vee}(\ul{kq}_j)\cong \m_2[v_0]\otimes M_*(\ul{kq}_{j}; Q_0)\oplus \left(v_0\text{-torsion concentrated in Adams filtration 0}\right).
	\]
	The $Q_0$-Margolis homology of $H_*kq$ is 
	\[
	M_*(A/\kern-0.25em/A(1)^\vee;Q_0) = \m_2[\oxi_1^2].
	\]
	The Mahowald weight filtration induces a filtration on the $Q_0$-Margolis homology and so we have 
	\[
	M_*(\ul{kq}_j;Q_0)\cong \m_2\{\oxi_1^2, \ldots, \oxi_1^{2j} \}.
	\]
	Since the spectral sequence is linear over $\Ext_{A(1)^\vee}(\m_2)$, and since the $v_0$-towers are concentrated in even stems, the spectral sequence collapses. 
\end{rem2}

\begin{proof}[Proof of \ref{Lem:SES}]
	Consider the composite 
	\[
	\begin{tikzcd}
		H\ul{\z}_{2j}\arrow[r] & A/\kern-0.25em/A(0)^\vee\arrow[r,"\pi"] & R^{j-1}A/\kern-0.25em/A(0)^\vee. 
	\end{tikzcd}
	\]
	We claim this is a surjective morphism. Let $x = m\oxi_1^\epsilon\otau_1\eta$ where $m$ is a monomial in $A/\kern-0.25em/A(1)^\vee$. Suppose that this determines a nonzero class in $R^{j-1}A/\kern-0.25em/A(0)^\vee$. Recall that 
	\[
	R^{j-1}A/\kern-0.25em/A(0) = A/\kern-0.25em/A(0)/F^jA/\kern-0.25em/A(0)^\vee.
	\]
	For this to be a nonzero class, the Mahowald weight of $m$ must be no more than $4(j-1) = 4j-4$. Since we have $wt(\oxi_1) = wt(\otau_1) = 2$, the monomial $m\oxi_1^\epsilon\otau_1^\eta$ has Mahowald weight bounded by $4j$. Thus $x\in H\ul{\z}_{2j}$, and hence the map $H\ul{\z}_{2j}\to R^{j-1}A/\kern-0.25em/A(0)^\vee$ is surjective. 
	
	We need to calculate the kernel. Consider an element $x = m\oxi_1^\epsilon\otau_1^\eta$ with $m\in A/\kern-0.25em/A(1)^\vee$. Suppose this element projects to 0 in $R^{j-1}A/\kern-0.25em/A(0)^\vee$. Then the Mahowald weight of $m$ must be $4j$. So the kernel of the map is 
	\[
	M_1(2j)\cong \Sigma^{4j,2j}H\ul{\z}_j.
	\]
	This completes the proof for $H\ul{\z}_{2j}$. The case for odd indices follows similarly.
\end{proof}

\begin{rem2}\label{Rem:3.26}
The computations in this section generalize to arbitrary base fields. The key point is that the generators of the dual motivic Steenrod algebra are the same for all base fields. 
\end{rem2}

\subsection{$\Ext_{A(1)^\vee}$ of Brown-Gitler modules}
The goal of this section is to prove Theorem \ref{Lem:ExtOfHZn} which computes $\Ext^{***}_{A(1)^\vee}(H\underline{\z}_j)$ modulo $v_1$-torsion for all $n$. 
Lemma \ref{Lem:SES} allows us to compute $\Ext^{***}_{A(1)^\vee}(H\underline{\z}_j)$ inductively from $\Ext^{***}_{A(1)^\vee}(H\underline{\z}_1^{\otimes i})$. We can compute $\Ext^{***}_{A(1)^\vee}(H\underline{\z}_1)$ using the \emph{algebraic Atiyah-Hirzebruch spectral sequence} (\emph{aAHSS}),
$$E^{***}_1 = \bigoplus_{x \in B} \Sigma^{|x|} \Ext^{*,*-stem(x),*-wt(x)}_{A(1)^\vee}(\m_2) \Rightarrow \Ext^{***}_{A(1)^\vee}(H\uz_1)$$
where $B = \{1, \bar{\xi}_1, \bar{\tau}_1\}$. By \cite[Thm. 6.6]{IS11}, we have
$$\Ext^{***}_{A(1)^\vee}(\m_2) \cong \dfrac{\m_2[h_0,h_1,\alpha,\beta]}{(h_0h_1,\tau h_1^3, h_1 \alpha, \alpha^2 = h_0^2 \beta)}$$
with $|h_0| = (1,0,0)$, $|h_1| = (1,1,1)$, $|\alpha| = (3,4,2)$, and $|\beta| = (4,8,4)$. 

The $E_1$-page of the aAHSS is depicted in Figure \ref{Fig:HZ1E1}, the $E_2$-page of the aAHSS is depicted in Figure \ref{Fig:HZ1E2}, and the $E_3 = E_\infty$-page of the aAHSS is depicted in Figure \ref{Fig:HZ1E3}. Differentials can be read off from the $A^\vee$-comodule structure of $H\underline{\z}_1$ as in the classical setting. The hidden extensions shown in Figure \ref{Fig:HZ1E3} may be deduced by applying the Betti realization functor $Re_\c : SH_\c \to SH$ and comparing to the classical computation of $\Ext^{**}_{A(1)^{cl}_*}(H\uz_1^{cl})$. 

\begin{fig}\label{Fig:HZ1E1}
The $E_1$-page of the aAHSS converging to $\Ext^{***}_{A(1)}(H\uz_1)$ with $d_1$-differentials. A black $\bullet$ represents $\m_2$, a red $\bullet$ represents $\f_2$, and a black $\square$ represents $\m_2[h_0]$. Vertical black lines also represent multiplication by $h_0$. Differentials are blue, $\tau$-linear, and they preserve motivic weight. 
\begin{center}
\includegraphics{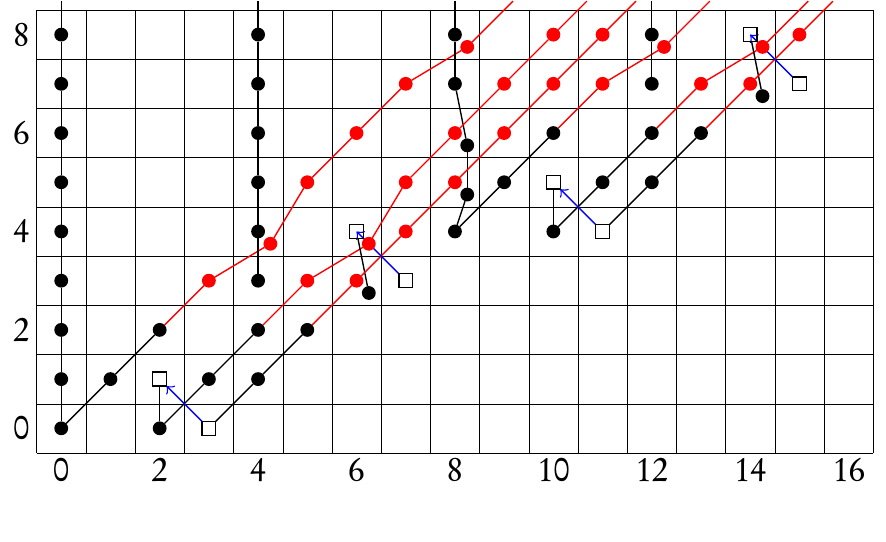}
\end{center}

\end{fig}

\begin{fig}\label{Fig:HZ1E2}
The $E_2$-page of the aAHSS converging to $\Ext^{***}_{A(1)}(H\uz_1)$ with $d_2$-differentials. Notation is as in Figure \ref{Fig:HZ1E1}, with dashed blue arrows representing differentials where either the source, target, or both are $\tau$-torsion. 

\begin{center}
	\includegraphics{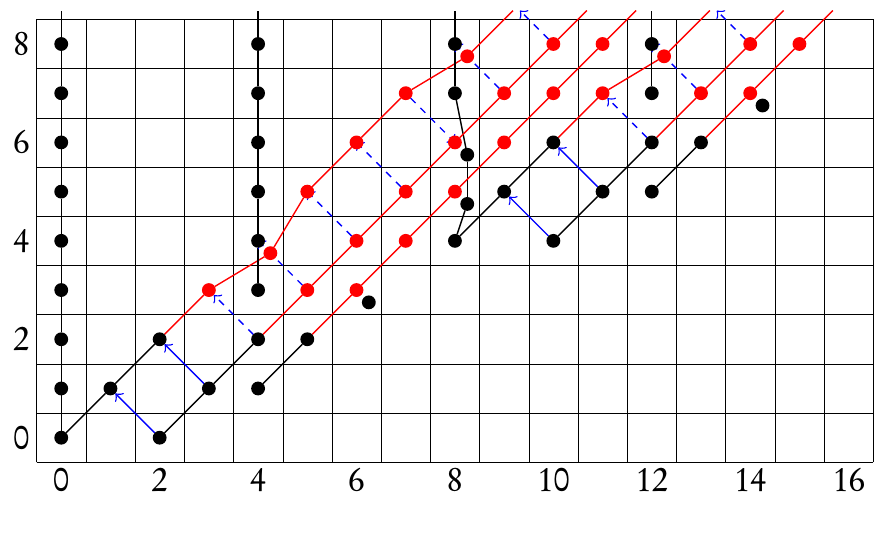}
\end{center}

\end{fig}

\begin{fig}\label{Fig:HZ1E3}
The $E_3 = E_\infty$-page of the aAHSS converging to $\Ext^{***}_{A(1)}(H\uz_1)$. Notation is as in Figure \ref{Fig:HZ1E1}. Vertical green lines indicate hidden $h_0$-extensions while a green $\bullet$ indicates that two classes are connected by a hidden $h_0$- or $\tau$-extension. 
\begin{center}
	\includegraphics{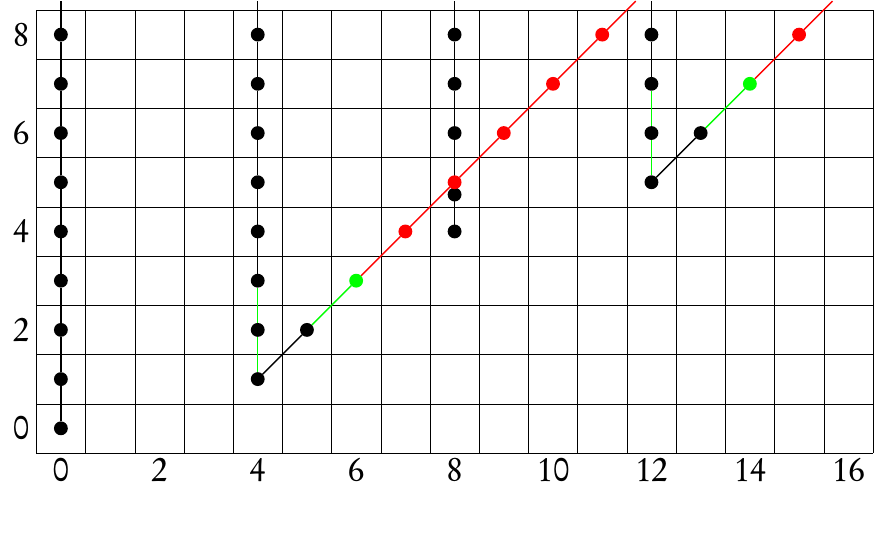}
\end{center}
\end{fig}

\begin{defin}
Let $ksp$ be the very effective cover of $\Sigma^{4,2}KQ$.
\end{defin}

\begin{cor}\label{Cor:HZ1}
There is an isomorphism
$$\Ext^{***}_{A(1)^\vee}(H\underline{\z}_1) \cong \pi_{**}(ksp).$$
\end{cor}

\begin{proof}
The right-hand side can be calculated using \cite[Thm. 16]{Bac17}, and the result is clearly isomorphic to the left-hand side. 
\end{proof}

In \cite{Mah81} (see also \cite{BOSS19}), Mahowald expresses $\Ext^{***}_{A(1)^{\vee}}(H\underline{\z}_1^{\otimes i})$ in terms of the Adams covers of $bo$ and $bsp$. The motivic analog of this result is slightly more complicated; as an example, we begin by computing $\Ext^{***}_{A(1)^\vee}(H\underline{\z}_1 \otimes H\underline{\z}_1)$. One way to approach this calculation is via the aAHSS for the functor $\Ext_{A(1)^\vee}^{***}(H\ul{\z}_1\otimes -)$. This is a spectral sequence taking the form 
\[
\Ext_{A(1)^\vee}^{***}(H\ul{\z}_1)\otimes_{\m_2}\m_2\{1, \oxi_1, \otau_1\}\implies \Ext_{A(1)^\vee}^{***}(H\ul{\z}_1^{\otimes 2}).
\]
The differentials and extensions can be determined as in the previous example. We present the calculation in the following figures. 

\begin{fig}\label{fig: HZ1HZ1E1}
$E_1$-page of the aAHSS for $\Ext^{***}_{A(1)^\vee}(H\underline{\z}_1 \otimes H\underline{\z}_1)$ with $d_1$-differentials. A black $\bullet$ represents $\m_2$, a red $\bullet$ represents $\f_2$. Differentials are blue and $\tau$-linear.
\begin{center}
	\includegraphics{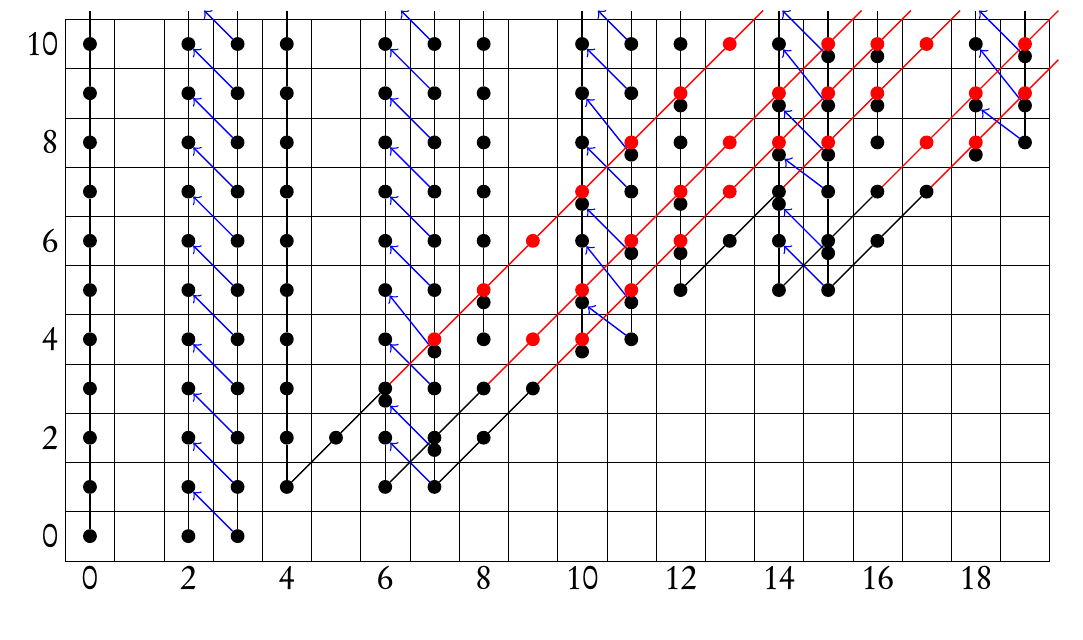}
\end{center}
\end{fig}

\begin{fig}\label{fig: HZ1HZ1E2}
The $E_2$-page of the aAHSS for $\Ext^{***}_{A(1)^\vee}(H\underline{\z}_1 \otimes H\underline{\z}_1)$ with $d_2$-differentials. A black $\bullet$ represents $\m_2$, a red $\bullet$ represents $\f_2$. Differentials are blue and $\tau$-linear, with a dashed differential indicating that the target, source, or both are $\tau$-torsion. 
\begin{center}
	\includegraphics{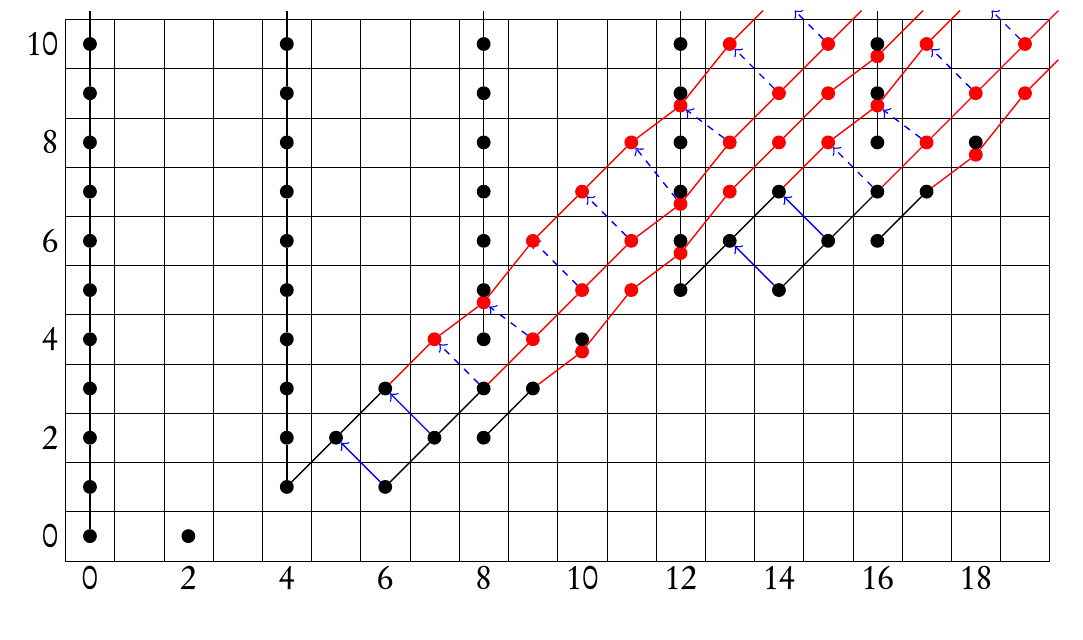}
\end{center}
\end{fig}

\begin{fig}\label{Fig:aAHSSi2}
The $E_3 = E_\infty$-page of the aAHSS for $\Ext^{***}_{A(1)^\vee}(H\underline{\z}_1 \otimes H\underline{\z}_1)$. A black $\bullet$ represents $\m_2$, a red $\bullet$ represents $\f_2$. Vertical green lines indicate hidden $h_0$-extensions while a green $\bullet$ indicates that two classes are connected by a hidden $h_0$- or $\tau$-extension. 

\begin{center}
	\includegraphics{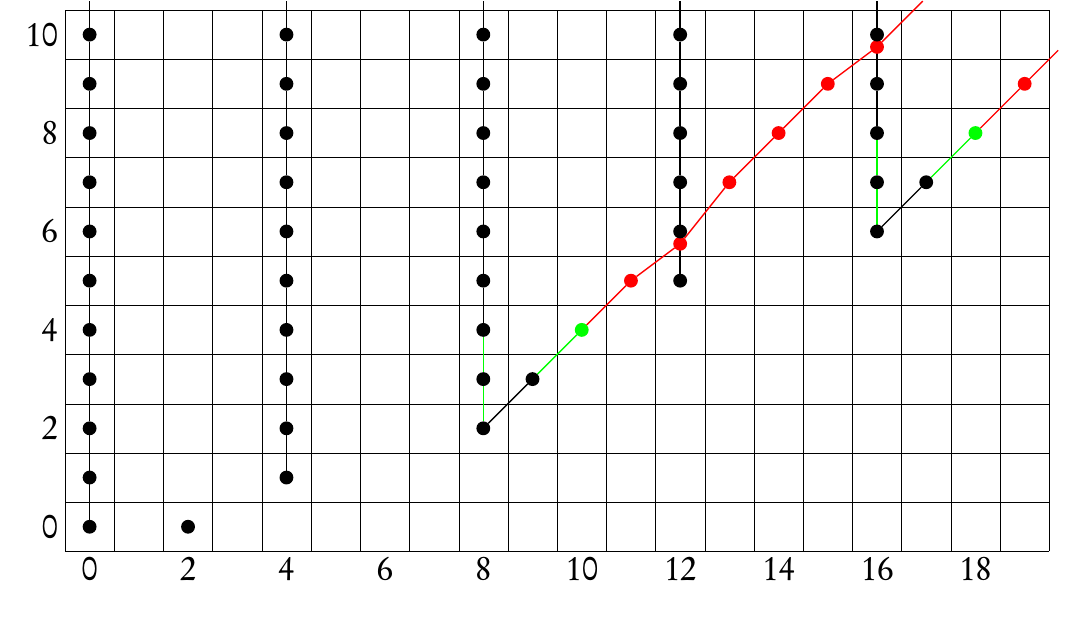}
\end{center}
\end{fig}

Already, we see that the classical description of $\Ext^{***}_{A(1)^{\vee}}(H \underline{\z}_1^{\otimes i})$ in terms of the Adams covers of $kq$ and $ksp$ fails when $i=2$. In particular, there would be an infinite $\tau$-torsion $h_1$-tower beginning in bidegree $(2,0)$ of Figure \ref{Fig:aAHSSi2} if this were the case. However, one may express the above computation as
$$\Ext_{A(1)^\vee}^{***}(H\underline{\z}_1^{\otimes 2}) \cong \left( \m_2 \otimes \tau_{<4} \Ext^{**}_{A(1)^{cl}}(bo^{\langle 2 \rangle}) \right)  \oplus  \tau_{\geq 4} \Ext_{A(1)^\vee}^{***}( kq^{\langle 2 \rangle})$$
where $\Ext^{**}_{A(1)^{cl}}(bo^{\langle 2 \rangle})$ is made into a trigraded object in an appropriate way. 

More generally, we see that $\Ext^{***}_{A(1)^\vee}(H\underline{\z}_1^{\otimes i}) / (v_1\mathrm{-tor})$ may be expressed as a combination of a classical Adams cover of $bo$ or $bsp$ and a suspension of $kq$ or $ksp$. Recall from \cite[Lem. 3.2]{BOSS19} that there are isomorphisms
\[
\dfrac{\Ext_{A(1)^{cl}_*}^{**}(H\underline{\z}_1^{\otimes i})}{v_1\mathrm{-tor}} \cong \begin{cases} \Ext^{**}_{A(1)^{cl}_*}(bo^{\langle i \rangle}), \quad & \text{ if } i \equiv 0 \mod 2, \\
\Ext^{**}_{A(1)^{cl}_*}(bsp^{\langle i - 1 \rangle}), \quad & \text{ if } i \equiv 1 \mod 2.
\end{cases}
\]
If $i \equiv 0 \mod 4$, then $i = 4k$ and the $\Ext$-groups on the right-hand side above can be expressed as 
$$\Ext^{**}_{A(1)^{cl}_*}(bo^{\langle i \rangle}) \cong \bigoplus_{j = 0}^{2k-1} \Sigma^{4j} \f_2[h_0] \oplus \Sigma^{4k} \Ext^{**}_{A(1)^{cl}_*}(bo).$$
Similar decompositions hold for other congruence classes of $i$; these decompositions lead us to the following definition.

\begin{defin}\label{Def:Zi}
For $i \geq 0$, let $Z_i$ be the trigraded group defined as follows. When $i \equiv 0 \mod 4$, let
$$Z_i := \bigoplus_{j = 0}^{i/2-1} \Sigma^{4j,2j} \m_2[h_0] \oplus \Sigma^{2i,i} \Ext^{***}_{A(1)^\vee}(\m_2).$$
When $i \equiv 1 \mod 4$, let
$$Z_i := \bigoplus_{j=0}^{(i-1)/2 -1} \Sigma^{4j,2j} \m_2[h_0] \oplus \Sigma^{2i-2,i-1} \Ext^{***}_{A(1)^\vee}(H\uz_1).$$
When $i \equiv 2 \mod 4$, let
$$Z_i := \bigoplus_{j=0}^{i/2-1} \Sigma^{4j,2j} \m_2[h_0] \oplus \Sigma^{2i-2,i} \m_2 \oplus \Sigma^{2i,i}  \Ext^{***}_{A(1)^\vee}(H\uz_1)[1].$$
When $i \equiv 3 \mod 4$, let
$$Z_i := \bigoplus_{j=0}^{(i-1)/2} \Sigma^{4j,2j} \m_2[h_0] \oplus \Sigma^{2i-1,i} \m_2[h_1]/(h_1^2) \oplus \Sigma^{2i+2,i+1} \Ext^{***}_{A(1)^\vee}(H\uz_1)[2].$$
In all cases, the tridegrees of the generators are $|\tau| = (0,0,-1)$, $|h_0| = (1,1,0)$, $|h_1| = (1,2,1)$, and $|x| = (s,t,u)$ for $x \in \Ext^{s,t,u}$. The functor $\Sigma^{k,\ell}$ shifts an element in tridegree $(s,t,u)$ to tridegree $(s,t+k,u+\ell)$ and the functor $(-)[m]$ shifts an element in tridegree $(s,t,u)$ to tridegree $(s+m,t+m,u)$. 
\end{defin}

\begin{lem}\label{Lem:ExtOfOtimes}
There is an isomorphism of $A(1)^\vee$-comodules $\Ext_{A(1)^\vee}(\m_2)$-modules
$$\dfrac{\Ext^{***}_{A(1)^\vee}(H\underline{\z}_1^{\otimes i })}{v_1-tor} \cong Z_i.$$
\end{lem}


\begin{proof}
Regarding $\Ext^{***}_{A(1)^\vee}(H\underline{\z}_1^{\otimes {i-1}} \otimes -)$ as a homology theory in the category of $A(1)^\vee$-comodules, we may inductively compute $\Ext^{***}_{A(1)^\vee}(H\underline{\z}_1^{\otimes i})$ via the aAHSS associated to the filtration of $H\underline{\z}_1$ by topological dimension. Thus we obtain a convergent spectral sequence
\[
\Ext_{A(1)^\vee}^{***}(H\ul{\z}_1^{\otimes i})\otimes_{\m_2}\m_2\{x[0], x[2],x[3]\}\implies \Ext_{A(1)^\vee}^{***}(H\ul{\z}_1^{\otimes i+1}).
\]
Here, an expression $\alpha x[k]$ where $\alpha\in \Ext_{A(1)^\vee}(H\ul{\z}_1^{\otimes i})$ denotes the element of $\Ext(H\ul{\z}_1^{\otimes i})$ on the cell of dimension $k$. In particular, this spectral sequence arises from applying the $\Ext$-functor to the following filtration by comodules of $H\ul{\z}_1$. Let $F^kH\ul{\z}_1$ denote the subspace spanned by generators in degrees $\leq k$. This is clearly a filtration by comodules. Thus we get 
\[
\begin{tikzcd}
0 \arrow[r] & F^0 H\ul{\z}_1\arrow[r]\arrow[d] & F^1H\ul{\z}_1\arrow[r] \arrow[d]& F^2H\ul{\z}_1\arrow[r]\arrow[d] &  F^3H\ul{\z}_1 \arrow[d]= H\ul{\z}_1 \\
& F^0H\ul{\z}_1 & 0 & F^2/F^2 & F^3/F^2
\end{tikzcd}
\]
and our spectral sequence arises from the exact couple coming from applying $\Ext$. Here the differentials are obtained by taking a cobar representative for a class $\alpha x[k]$, lifting them to the cobar complex for $F^kH\ul{\z}_1$, applying the cobar differential, and then projecting on to the highest filtration forms. Consider an element $\alpha x[3]$ on $E_1$. This is represented in the (normalized) cobar complex for $F^3/F^2$ by $a\otimes \otau_1$, where $a$ is a cocycle of the cobar complex for $H\ul{\z}_1^{\otimes i}$ representing $\alpha$. Since $a$ is a cocycle, the cobar differential yields
\[
d(\alpha\otimes \otau_1) = \alpha\otimes \otau_0|\oxi_1.
\]
Thus we derive for each $\alpha\in \Ext(H\ul{\z}_1^{\otimes i})$ 
\[
d_1(\alpha x[3]) = \alpha h_0x[2].
\]
Similar considerations show that there is a $d_2$-differential 
\[
d_2(\alpha x[2]) = \alpha h_1 x[0].
\]
Degree considerations show that these are the only differentials which occur. 

In order to get the desired hidden extensions, we use the isomorphism $H\underline{\z}_1^{\otimes i}[\tau^{-1}] \cong (H\underline{\z}^{cl}_1)^{\otimes i}[\tau^{\pm 1}]$ which induces a ring homomorphism
\[
\Ext_{A(1)^\vee}(H\ul{\z}_1^{\otimes i})\to \Ext_{A(1)^{cl}}((H\ul{\z}^{cl}_1)^{\otimes i}).
\]
Comparison with the classical computation (cf. \cite{BOSS19}) yields the desired extensions. 
\end{proof}

\begin{rem2}
Working modulo $v_1$-torsion serves two purposes. First, it drastically simplifies calculations since we may ignore the numerous $v_1$-torsion classes in Adams filtration $0$. Second, we are only interested in $v_1$-periodic and $\eta$-periodic classes in the stable stems. All $\eta$-periodic classes are also $v_1$-periodic, so we do not miss anything by ignoring $v_1$-torsion. 
\end{rem2}

This gives us the input to calculate the $\Ext$-groups for the comodules $H\ul{\z}_n$ using long exact sequences in $\Ext$ arising from the short exact sequences of Lemma \ref{Lem:SES}. The following should be compared with \cite[Prop. 2.6]{Mah81} and \cite[Prop. 3.3]{BOSS19}. 

\begin{thm}\label{Lem:ExtOfHZn}
There is an isomorphism of $A(1)^\vee$-comodules $\Ext_{A(1)^\vee}(\m_2)$-modules
\[
\dfrac{\Ext^{***}_{A(1)^\vee}(H\underline{\z}_n)}{v_1\mathrm{-tor}} \cong Z_{2n-\alpha(n)},
\]
where $\alpha(n)$ is the number of $1$'s in the dyadic expansion of $n$. 
\end{thm}

\begin{proof}
For this proof only, all $\Ext$-groups are implicitly calculated modulo $v_1$-torsion. We proceed by induction on $n$. The lemma is clear for $n = 1$ by Corollary \ref{Cor:HZ1}. Assume now that the lemma holds for all $H\underline{\z}_i$ with $i < n$.

If $n$ is even, then $H\underline{\z}_n$ fits into the short exact sequence
$$0 \to \Sigma^{2n,n} H\underline{\z}_{n/2} \to H\underline{\z}_{n} \to \underline{kq}_{n/2-1} \otimes (A(1)/\kern-0.25em/A(0))^{\vee} \to 0$$
by Lemma \ref{Lem:SES}. Applying $\Ext_{A(1)^\vee}^{***}(-)$, we see that the boundary map in the resulting long exact sequences is zero (modulo $v_1$-torsion). Therefore $\Ext^{***}_{A(1)^\vee}(H\underline{\z}_n)$ decomposes into the $\Ext$ groups of the left-hand side and right-hand side. By Remark \ref{Rem:kqMargolis}, the $\Ext$-groups of the right-hand side are isomorphic to
$$\Ext^{***}_{A(1)^\vee}(\underline{kq}_{n/2-1} \otimes (A(1)/\kern-0.25em/A(0))^\vee) \cong \bigoplus_{j =0}^{n/2-1} \Sigma^{4j,2j} \m_2[h_0].$$
The $\Ext$-groups of the left-hand side are given by
$$\Ext^{***}_{A(1)^\vee}(\Sigma^{2n,n} H\underline{\z}_{n/2}) \cong \Sigma^{2n,n} Z_{n-\alpha(n/2)}$$
by the induction hypothesis, so we have
$$\Ext^{***}_{A(1)^\vee}(H\underline{\z}_n) \cong \bigoplus_{j  = 0}^{n/2-1} \Sigma^{4j,2j} \m_2[h_0] \oplus \Sigma^{2n,n} Z_{n-\alpha(n/2)} \cong Z_{2n-\alpha(n)}$$
(note that $\alpha(n) = \alpha(n/2)$). 

On the other hand, if $n$ is odd, then $H\underline{\z}_n$ fits into the short exact sequence
$$0 \to \Sigma^{2(n-1),n-1} H\underline{\z}_{(n-1)/2} \otimes H\underline{\z}_1 \to H\underline{\z}_n \to \underline{kq}_{(n-1)/2-1} \otimes (A(1)/\kern-0.25em/A(0))^\vee \to 0$$
by Lemma \ref{Lem:SES}. The $\Ext$-groups of the right-hand side are
$$\Ext^{***}_{A(1)^\vee}(\underline{kq}_{(n-1)/2-1} \otimes (A(1)/\kern-0.25em/A(0))^\vee) \cong \bigoplus_{j = 0}^{(n-1)/2-1} \Sigma^{4j,2j} \m_2[h_0]$$
by Remark \ref{Rem:kqMargolis}, where the $v_0$-torsion classes in Adams filtration are suppressed since they are $v_1$-torsion. It remains to calculate
$$\Ext^{***}_{A(1)^\vee}\left(\Sigma^{2(n-1),n-1}H\underline{\z}_{(n-1)/2} \otimes H\underline{\z}_1\right).$$
As in the proof of Lemma \ref{Lem:ExtOfOtimes}, we analyze the aAHSS associated to the homology theory $\Ext^{***}_{A(1)^\vee}(- \otimes H\underline{\z}_1)$ in the category of $A(1)^\vee$-comodules. This spectral sequence takes the form
$$\Ext^{***}_{A(1)^\vee}(H\underline{\z}_{(n-1)/2}) \otimes_{\m_2} \m_2\{x[0],x[2],x[3]\} \Rightarrow \Ext^{***}_{A(1)^\vee}(H\underline{\z}_{(n-1)/2} \otimes H\underline{\z}_1).$$
By the induction hypothesis, the left-hand side is isomorphic as $A(1)^\vee$-comodules to
$$Z_{n-1-\alpha((n-1)/2)} \otimes_{\m_2} \m_2\{x[0],x[2],x[3]\}.$$

Suppose that $n-1-\alpha((n-1)/2) \equiv 0 \mod 4$. Then $Z_{n-1-\alpha((n-1)/2)}$ decomposes as a direct sum
\begin{align*}
\bigoplus_{j = 0}^{(n-1-\alpha(\frac{n-1}{2}))/2-1} \Sigma^{4j,2j} \m_2[h_0] \oplus \Sigma^{2(n-1-\alpha(\frac{n-1}{2})),n-1-\alpha(\frac{n-1}{2})} \Ext^{***}_{A(1)^\vee}(kq)[0] \\
\cong \bigoplus_{j=0}^{(n-\alpha(n))/2-1} \Sigma^{4j,2j} \m_2[h_0] \oplus \Sigma^{2n-2\alpha(n),n-\alpha(n)} \Ext^{***}_{A(1)^\vee}(kq)[0]
\end{align*}
where we have used the relation $\alpha(\frac{n-1}{2}) = \alpha(n)-1$ to rewrite the upper bound of the direct sum and the bidegree of the suspension. This splitting gives rise to an analogous decomposition of the aAHSS. On the left-hand summand, the aAHSS collapses at $E_2$ and is identical to the $E_1$-page (recall that we are working modulo $v_1$-torsion). The aAHSS for the right-hand summand was calculated in Corollary \ref{Cor:HZ1} to be $Z_1$. Therefore we obtain an isomorphism of $A(1)^\vee$-comodules 
$$\Ext^{***}_{A(1)^\vee}\left(H\underline{\z}_{(n-1)/2} \otimes H\underline{\z}_1\right) \cong \bigoplus_{j=0}^{(n-\alpha(n))/2-1} \Sigma^{4j,2j}\m_2[h_0] \oplus \Sigma^{2n-2\alpha(n),n-\alpha(n)} Z_1[0].$$
Altogether, we see that $\Ext^{***}_{A(1)^\vee}(H\underline{\z}_n)$ is isomorphic to 
$$\bigoplus_{j=0}^{(n-1)/2-1} \Sigma^{4j,2j} \m_2[h_0] \oplus \Sigma^{2(n-1),n-1} \left( \bigoplus_{k=0}^{(n-\alpha(n))/2-1} \Sigma^{4k,2k} \m_2[h_0] \oplus \Sigma^{2n-2\alpha(n),n-\alpha(n)} Z_1[0] \right)$$
which can be rewritten as
$$\bigoplus_{j=0}^{(2n-\alpha(n))/2-1} \Sigma^{4j,2j} \m_2[h_0] \oplus \Sigma^{4n-2\alpha(n),2n-\alpha(n)} \Ext^{***}_{A(1)^\vee}(kq) \cong Z_{2n-\alpha(n)}.$$

The calculations for other congruence classes of $n-1-\alpha((n-1)/2)$ modulo $4$ are similar. The only subtlety arises in the case $n-1-\alpha((n-1)/2) \equiv 2 \mod 4$, in which case the copy of $\m_2$ suspended by $(2i-2,i)$ is related to the copy of $\m_2[h_0]$ beginning in cohomological filtration $1$ by a hidden $h_0$-extension. This can be seen by comparing to the classical case. 
\end{proof}

\begin{rem2}\label{Rem:3.39}
Analogous methods can be used to show that $\Ext_{A(1)^\vee}^{***}(H\uz_n^\r)$ (modulo $v_1$-torsion) admits a similar description over $\Spec(\r)$. Hill's computation \cite{Hil11} of $\Ext_{A(1)}^{***}(\m_2^\r)$ serves as the input for the algebraic Atiyah-Hirzebruch spectral sequence, and base-change can be used to resolve hidden extensions. The generalization to arbitrary base fields is the subject of ongoing investigation. 
\end{rem2}

We observe the following relation between the motivic weight and stem of the generators of infinite $\tau$-towers in these $Ext$-groups.

\begin{cor}\label{Cor:Half1}
Suppose that $x$ is a $\tau$-free class in $\Ext^{s,t,u}_{A(1)^\vee}(H\underline{\z}_n) / (v_1-tor)$ which is not $\tau$-divisible. Then we have $u =
\lceil (t-s)/2 \rceil$ if $t-s \equiv 0,1 \mod 4$ and $u = \lceil (t-s)/2 \rceil + 1$ if $t-s \equiv 2 \mod 4$. Moreover, there are no $\tau$-free classes with $t-s \equiv 3 \mod 4$. 
\end{cor}

We can now compute the ring of $kq$-cooperations $\pi_{**}(kq \wedge kq)$ using the motivic Adams spectral sequence. First, we summarize the $E_2$-page. 

\begin{prop}\label{Prop:MASSE2}
The $E_2$-page of the motivic Adams spectral sequence converging to $\pi_{**}(kq \wedge kq)$ is given, modulo $v_1$-torsion, by
$$E_2^{s,t,u} = \Ext^{s,t,u}_{A^\vee}(H_{**}(kq \wedge kq)) \cong \bigoplus_{i \geq 0} \Sigma^{4i,2i} Z_{2i-\alpha(i)},$$
where $Z_j$ is as in Definition \ref{Def:Zi}. 
\end{prop}

\begin{proof}
Combine Theorem \ref{Cor:AmodmodA(1)Decomp} with Theorem \ref{Lem:ExtOfHZn}. 
\end{proof}

Now, recall the following theorem of Mahowald:

\begin{thm}\cite[Thm. 2.9]{Mah81}
The Adams spectral sequence converging to $\pi_{*}(bo \wedge bo)$ collapses at $E_2$, i.e. $E_2 = E_\infty$. 
\end{thm}

\begin{cor}\label{Cor:kqCollapse}
The motivic Adams spectral sequence converging to $\pi_{**}(kq \wedge kq)$ collapses at $E_2$. 
\end{cor}

\begin{proof}
Since Betti realization is obtained by inverting $\tau$ and then setting $\tau = 1$, there can be no differentials in the motivic Adams spectral sequence where both the source and target are $\tau$-free. On the other hand, the only $\tau$-torsion classes are $h_1$-torsion free. Therefore the only possible differentials are between $\tau$-torsion $h_1$-towers, but these are not possible for tridegree reasons. More precisely, the beginning of each $\tau$-torsion $h_1$-tower is separated from another by stem $4i$ and motivic weight $2i$, $i \geq 1$. Since $stem(h_1) = 1$ and $wt(h_1) = 1$, the possible targets of a $kq$-Adams differential are in too high of a weight.
\end{proof}

\begin{fig}\label{fig: kqresE1}
The $E_2 = E_\infty$-page of the motivic Adams spectral sequence converging to $\pi_{**}(kq \wedge kq)$ with $v_1$-torsion classes suppressed. The horizontal axis indicates stem, and the vertical axis indicates which $i$ in the decomposition above is used. Adams filtration and motivic weight are suppressed. A $\square$ represents $\m_2[h_0]$, a black $\bullet$ represents $\m_2$, and a red $\textcolor{red}{\bullet}$ indicates $\f_2$. Horizontal lines indicate multiplication by $h_1$, with a red horizontal line indicating that the target is simple $\tau$-torsion. A red horizontal arrow represents an infinite $h_1$-tower. 
\hskip -.5in
\begin{center}
\includegraphics[width = \textwidth]{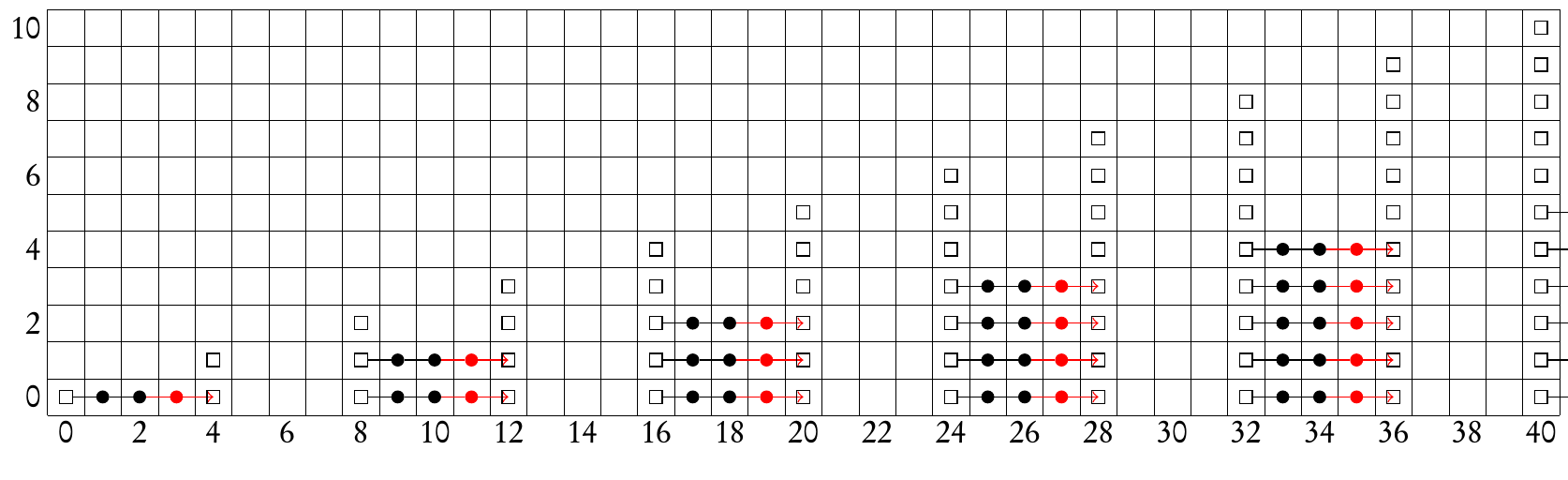}
\end{center}
\end{fig}

\subsection{A vanishing line of slope $1/3$ on $E_1$ of the $kq$-resolution}

We now show that the $E_1$-page of the $kq$-resolution has a vanishing line of slope $1/3$. Recall that we have the cofiber sequence
\[
S^{0,0}\to kq\to \overline{kq}.
\] 
The long exact sequence in homotopy groups shows that $\pi_{t,*}\overline{kq} = 0$ for $t\leq 3$. Thus, the smash powers $kq\wedge \overline{kq}^{\wedge s}$ are $4s-1$-connected. Thus we can conclude the following. 

\begin{lem}\label{Lem:OneThird}
The $E_1$-page of the $kq$-resolution satisfies $E_1^{s,t,*} = 0$ whenever $t<4s$. 
\end{lem}

In other words, on the $E_1$-page of the $kq$-resolution, the region above the line of slope $1/3$ passing through the origin consists only of trivial groups. This result will be used in the sequel to prove that there is a vanishing line of slope $1/5$ on the $E_2$-page. 

\section{Differentials and a vanishing line of slope $1/5$ on the $E_2$-page}\label{Sec:E2}

\subsection{Review of the $E_\infty$-page of the $bo$-resolution}
We begin by reviewing the analysis of the $bo$-resolution
$$E_1^{n,t} = \pi_t(bo \wedge \overline{bo}^{\wedge n}) \Rightarrow \pi_{t-n}(S^0)$$
from \cite{Mah81} and \cite{Mah84}. Mahowald's main theorem consists of two parts: the calculation of the $0$- and $1$-lines of the $E_\infty$-page of the $bo$-resolution (\cite[Thm. 1.1.(a-b)]{Mah81}) and the vanishing line of slope $1/5$ on the $E_2$-page of the $bo$-resolution \cite[Thm. 1.1.(c)]{Mah81}. 

Both parts depend on understanding differentials in the $bo$-resolution. The first set of differentials is discussed in the following theorem, which is also called the ``Bounded Torsion Theorem" in \cite[Cor. 3.7]{BBBCX17}. 

\begin{thm}\label{Thm:boDifferentials}\cite[Thm. 5.11]{Mah81}\cite[(c')]{Mah84}
Each homotopy class in $\pi_*(bo \wedge \overline{bo}^{\wedge n})$ is either of Adams filtration zero or one, or in the image of $d_n$, or is mapped essentially under $d_n$, or $n = 0$ or $n = 1$ and the homotopy can be identified with $\Ext^{**}_{A(1)^{cl}_*}(H\underline{\z}_1^{cl})$. 
\end{thm}

The second set of differentials produces the relevant $2$-torsion in the $1$-line of the $bo$-resolution. 

\begin{thm}\label{Thm:boDiffl2}\cite[Pg. 380]{Mah81}
Let $k \geq 1$. The $d_1$-differential in the $bo$-resolution $d_1 : E^{0,4k}_1 \cong \z \to \z \cong E^{1,4k}_1$ is given by multiplication by $2^{\rho(k)}$, where $\rho(k)$ is the $2$-adic valuation of $8k$. 
\end{thm}

\begin{rem2}\label{Rem:FDiff}
Both theorems follow from an explicit comparison with the $X$-resolution, where $X := Th(\Omega S^5 \to BO)$. In the next section, we will produce analogous differentials in the $kq$-resolution. Instead of defining a motivic analog of $X$, we produce differentials in the $kq$-resolution by comparing to the $bo$-resolution using Betti realization. 

Over more general base fields, comparison to the classical case will not produce all of the relevant differentials. Instead, we expect that a combination of equivariant Betti realization (with respect to $C_2 = \Gal(\c/\r)$) and base-change functors will suffice for producing these differentials. This approach relies on analyzing the $C_2$-equivariant Adams spectral sequenced based on the $C_2$-spectrum $ko_{C_2}$ of \cite[Def. 10.1]{GHIR19}, as well as the construction of a $C_2$-equivariant analog of $X$. We plan to carry out this program in future work. 
\end{rem2}

The vanishing line requires an additional argument which we sketch now. Let $A_1$ be a spectrum with $H^*(A_1) \cong A(1)$. Then the $E_2$-page of the $bo$-resolution for $A_1$ may be identified with the Adams spectral sequence for $A_1$. The latter has a vanishing line of slope $1/5$ on its $E_2$-page which can be proven using the May spectral sequence. If one takes $A_1 = S^0/(2,\eta,v_1)$, then one might hope to produce the vanishing line in the $bo$-resolution for $S^0$ through a series of Bockstein spectral sequences. However, this runs into the issue that one only has a Bockstein spectral sequence at the level of $E_1$-pages in general. 

Instead, Mahowald relates the $E_2$-page of the $bo$-resolution for the mod two Moore spectrum $S^0/2$ to the $E_2$-page of the $bo$-resolution for $A_1$ above a line of slope $1/5$. This leads to a vanishing line of slope $1/5$ in the $E_2$-page of the $bo$-resolution for $S^0/2$. To obtain the vanishing line for the sphere, Mahowald then uses the cofiber sequence $$S^0 \overset{2}{\to} S^0 \to S^0/2$$
along with the explicit calculation of $2$-torsion and periodicity in the $E_2$-page above a line of slope $1/5$. This explicit calculation essentially follows from the expression for the $E_1$-page in terms of Adams covers of $bo$ and $bsp$. 

\subsection{Differentials}
In this section, we prove the motivic analog of Mahowald's Bounded Torsion Theorem.

The relevant differentials in the $kq$-resolution 
$$\kqE_1^{n,t,u} = \pi_{t,u}(kq \wedge \overline{kq}^{\wedge n}) \Rightarrow \pi_{t-n,u}(S^{0,0})$$
can be determined via Betti realization, $\tau$-linearity, and $\eta$-linearity. Recall that $\pi_*(X(\c)) \otimes \f_2[\tau,\tau^{-1}] \cong \pi_{**}(X)[\tau^{-1}]$, so the realization of a $\tau$-torsion class is zero and the realization of a $\tau$-free class is nonzero. This fact will allow us to deduce many of the relevant differentials in the $kq$-resolution from the differentials calculated by Mahowald in the $bo$-resolution \cite{Mah81}. We obtain the remaining differentials using the $\pi_{**}(kq)$-module structure of the $kq$-resolution. 

Recall the isomorphism
$$\pi_{**}(kq \wedge \overline{kq}^{\wedge n})[\tau^{-1}] \cong \pi_*(bo \wedge \overline{bo}^{\wedge n}) \otimes \f_2[\tau,\tau^{-1}].$$
We have shown above that $\pi_{**}(kq \wedge kq)$ is a direct sum of combinations of $\pi_*(bo^{\langle j \rangle})$ (in appropriate bidegrees) and $\pi_{**}(kq)$ for various values of $j$, together with $\m_2$-summands in Adams filtration zero. The same holds when we consider $\pi_{**}(kq \wedge \overline{kq}^{\wedge n})$ for any $n$ by a similar analysis.

In more detail, Mahowald expresses the $n$-line of the $bo$-resolution using multi-indices in the proof of \cite[Thm. 5.11]{Mah81}. One has
$$\Ext_{A(1)^\vee}^{**}(\overline{bo}^{\wedge n}) \cong \bigoplus_{I \in \mathcal{I}_n} \Sigma^{|I|} \Ext_{A(1)^\vee}^{**}(H\underline{\z}^{cl}_I)$$
where $\mathcal{I}_n = \{ I = (i_1,\ldots,i_n) : i_j \geq 1 \text{ for all } 1 \leq j \leq n\}$ and if $I = (i_1,\ldots,i_n)$, then $\Sigma^{|I|} = \Sigma^{4(i_1 + \cdots + i_n)}$ and $H\underline{\z}^{cl}_I = H\underline{\z}^{cl}_{i_1} \otimes \cdots \otimes H\underline{\z}^{cl}_{i_n}.$ This follows easily from the K{\"u}nneth isomorphism for the classical mod two homology and the classical analog of Theorem \ref{Cor:AmodmodA(1)Decomp}. Since motivic mod two homology also has a K{\"u}nneth isomorphism, we obtain an analogous expression
$$\Ext_{A(1)^\vee}^{***}(\overline{kq}^{\wedge n}) \cong \bigoplus_{I \in \mathcal{I}_n} \Sigma^{|I|,|I|/2} \Ext_{A(1)^\vee}^{***}(H\underline{\z}_I)$$
where $\Sigma^{|I|,|I|/2} = \Sigma^{4(i_1+\cdots+i_n), 2(i_1+\cdots+i_n)}$. 

We can say more about the $n$-line. Using the long exact sequences of Lemma \ref{Lem:SES} and the aAHSS, we may show that modulo $v_1$-torsion, each summand of the $n$-line is expressible as a sum of suspensions of $\m_2[h_0]$, one or two copies of $\m_2$, and $\pi_{**}(kq)$ or $\pi_{**}(ksp)$. Moreover, we have the following corollary of Corollary \ref{Cor:Half1}.

\begin{cor}\label{Cor:Half}
Suppose that $x$ is a $\tau$-free class in $\kqE_1^{n,t,u}$ which is not $\tau$-divisible. Then we have $u = \lceil t/2 \rceil$ is $t \equiv 0,1 \mod 4$ and $u = \lceil t/2 \rceil +1$ if $t \equiv 2 \mod 4$. Moreover, there are no $\tau$-tree classes with $t \equiv 3 \mod 4$. 
\end{cor}

\begin{proof}
The class $x$ is a sum of elements in $\pi_{t,u}(kq \wedge H\z_I)$ where $I \in \mathcal{I}_n$. Since $\pi_{t,u}(kq \wedge H\underline{\z}_I) \cong \bigoplus_{s \in \z}\Ext_{A(1)^\vee}^{s,t+s,u}(H\underline{\z}_I)$, the result follows from the calculations in Corollary \ref{Cor:Half1} and the long exact sequences of Lemma \ref{Lem:SES}. 
\end{proof}

We may compare the $kq$-resolution and $bo$-resolution precisely since $kq(\c) \simeq bo$ and Betti realization is strong symmetric monoidal. In particular, the piece of the $n$-line of the $kq$-resolution coming from a multi-index $I \in \mathcal{I}_n$ realizes to the piece of the $n$-line of the $bo$-resolution coming from the same multi-index $I$. The following statements are clear from the above observations:
\begin{enumerate}
\item Any class in $\pi_{**}(kq \wedge \overline{kq}^{\wedge n})$ coming from $\pi_*(bo^{\langle j \rangle})$ (placed in the correct bidegrees) is $\tau$-torsion free. In particular, the subset of these classes which are not $\tau$-divisible are in one-to-one correspondence with a subset $S$ of classes in $\pi_*(bo \wedge \overline{bo}^{\wedge n})$. 
\item The classes in $\pi_{**}(kq \wedge \overline{kq}^{\wedge n})$ coming from $\pi_{**}(\Sigma^{?,?}kq)$ satisfy precisely one of the following:
\begin{enumerate}
\item The class is $\tau$-torsion free. The subset of such $\tau$-torsion free classes is in one-to-one correspondence with the complement of $S$ in $\pi_*(bo \wedge \overline{bo}^{\wedge n })$. 
\item The class is simple $\tau$-torsion. Any such class is $\eta$-torsion free. 
\end{enumerate}
\end{enumerate}

\begin{thm}\label{Prop:kqDifferentials}\mbox{}
\begin{enumerate}
\item Suppose that a $\tau$-torsion free class $x \in \kqE_1^{n,t,u}$ with $n \geq 2$ is a cycle under $d_1$ and is represented in $\pi_{t,u}(kq \wedge \overline{kq}^{\wedge n})$ by an element of $H\f_2$-Adams filtration $f \geq 2$. Then $x$ is a boundary under $d_1$. 

\item Suppose that a $\tau$-torsion class $x \in \kqE_1^{n,t,u}$ with $n \geq 2$ is a cycle under $d_1$ and is represented in $\pi_{t,u}(kq \wedge \overline{kq}^{\wedge n})$ by an element of $H\f_2$-Adams filtration $f \geq 2$. Then $x$ is a boundary under $d_1$.
\end{enumerate}
\end{thm}

\begin{proof}\mbox{}
\begin{enumerate}
\item This follows from Theorem \ref{Thm:boDifferentials} and Betti realization. Assume that $x$ is not $\tau$-divisible. Then Theorem \ref{Thm:boDifferentials} implies that there is a differential $d_1(y(\c)) = x(\c)$ for some $y(\c) \in \kqE_1^{n-1,t}$. The properties of Betti realization described above imply that there is a differential $d_1(y) = \tau^{u'-u} x$ for some $y \in \kqE_1^{n-1,t,u'}$. We may assume that $y$ is not $\tau$-divisible; indeed, if $y=\tau^j y'$ then $\tau$-lineariity of $d_1$-differentials implies that $d_1(y') = \tau^{u'-u-j}x$. In this case, we know that $stem(x) = t-n$ and $wt(x) = \lceil t/2 \rceil$ if $t \equiv 0,1 \mod 4$ and $wt(x) = \lceil t/2 \rceil + 1$ if $t \equiv 2 \mod 4$ by Corollary \ref{Cor:Half}. We also know that $stem(y') = t-n+1$ and $wt(y') = \lceil t/2 \rceil$ if $t \equiv 0,1 \mod 4$ and $wt(y') = \lceil t/2 \rceil + 1$ if $t \equiv 2 \mod 4$ by the same corollary. Therefore $u = u'$ and the differential has the form $d_1(y') = x$. 

\item The only $\tau$-torsion classes contain $\eta^j$ for $j \geq 3$ as a factor. The multiplication map $kq \wedge kq \to kq$ equips the $kq$-resolution with a $\pi_{**}(kq)$-module structure, so the fate of $\eta^j$ for $j \geq 3$ is the same as the fate of $\eta$ and $\eta^2$. This was determined in the first part of the proposition. 
\end{enumerate}
\end{proof}

\begin{thm}\label{Prop:kqDiffl2}
Let $k \geq 1$. The $d_1$-differential $d_1 : \kqE^{0,4k,\ell}_1 \cong \z \to \z \cong \kqE^{1,4k,\ell}_1$ is given by multiplication by $2^{\rho(k)}$ for any $\ell \leq 2k$, where $\rho(k)$ is the $2$-adic valuation of $8k$. 
\end{thm}

\begin{proof}
This follows immediately from Theorem \ref{Thm:boDiffl2} and Betti realization. 
\end{proof}

These propositions completely determine the $0$- and $1$-lines of the $kq$-resolution; see Parts (a) and (b) of Theorem \ref{Thm:kqLines}. 

\begin{rem2}
It is unlikely that Betti realization and base-change will suffice to produce all of the differentials involving the $0$- and $1$-lines of the $kq$-resolution over general base fields. 
\end{rem2}

\subsection{A vanishing line of slope $1/5$}

As we saw in the last section, there is a vanishing line of slope $1/3$ in the $E_1$-page of the $kq$-resolution. It turns out that this naive vanishing line suffices to calculate $\pi_{**}(S^{0,0})[\eta^{-1}]$, but in order to calculate $\pi_{**}(S^{0,0})[v_1^{-1}]$,  we will need a stronger vanishing result. To prove this stronger vanishing statement, we generalize Mahowald's arguments \cite[pp. 380-381]{Mah81} and \cite[(c')]{Mah84}. 

We begin by defining some motivic analogs of classical finite complexes. 

\begin{defin}
Let $M := \mathrm{cofib}(S^{0,0} \overset{\cdot 2}{\to} S^{0,0})$  be the motivic mod two Moore spectrum. Using the long exact sequence in homotopy, one can show that there exists a unique lift of $\eta \in \pi_{1,1}(S^{0,0})$ to a map $\tilde{\eta} : \Sigma^{1,1} M \to M$. 

 Let $Y := \mathrm{cofib}(\Sigma^{1,1} M \overset{\tilde{\eta}}{\to} M)$. Then $Y \simeq V(0) \wedge C\eta$ where $C\eta := \mathrm{cofib}(S^{1,1} \overset{\eta}{\to} S^{0,0})$. Using the Atiyah-Hirzebruch spectral sequence, one can show that $Y$ admits a self-map $v_1 : \Sigma^{2,1} Y \to Y$. Finally, define $A_1$ by $A_1 := \mathrm{cofib}(\Sigma^{2,1} Y \overset{v_1}{\to} Y)$. 
\end{defin}

\begin{rem2}
The Betti realization of $Y$ is the classical spectrum $Y$, and a choice motivic self-map $v_1 : \Sigma^{2,1} Y \to Y$ realizes to a corresponding classical self-map $v_1^{cl} : \Sigma^2 Y \to Y$. We will fix these choices in the sequel. Since the classical self-map is non-nilpotent and the motivic self-map is $\tau$-torsion free, the motivic self-map is non-nilpotent. 
\end{rem2}

\begin{lem}
The motivic spectrum $A_1$ realizes $A(1)$. In other words,  there is an isomorphism of $A(1)$-modules $H^{**}(A_1) \cong A(1)$. 
\end{lem}

\begin{proof}
This follows from the long exact sequences in motivic cohomology associated to the cofiber sequences defining $A_1$, along with the facts that $Sq^1$ detects $2$, $Sq^2$ detects $\eta$, and $Q_1$ detects $v_1$. 
\end{proof}
%

\begin{lem}\label{Lem:A1}
The $kq$-resolution for $A_1$ coincides with the motivic $H\f_2$-Adams resolution for $A_1$. In particular, we have
$$^{kq}E_2(A_1) \cong \mbox{}^{mASS}E_2(A_1) \cong \Ext_{A^\vee}^{***}(A(1)^\vee).$$
\end{lem}

\begin{proof}
Our argument follows Mahowald's in \cite{Mah81}, but we add some details. The motivic Adams spectral sequence 
$$E_2 = \Ext_{A^\vee}^{***}(H_{**}(A_1 \wedge kq)) \cong \Ext_{A(1)^\vee}(A(1)^\vee) \Rightarrow \pi_{**}(A_1 \wedge kq)$$
collapses to show that $A_1 \wedge kq \simeq H\f_2$. We claim that the canonical $kq$-Adams resolution is an $H\f_2$-Adams resolution. To show this, it needs to be checked that 
\begin{enumerate}
	\item The spectra $kq\wedge \overline{\kq}^{\wedge n}\wedge A_1$ are generalized Eilenberg-MacLane spectra, and 
	\item The canonical maps $\overline{kq}^{\wedge n}\wedge A_1\to kq\wedge \overline{kq}^{\wedge n}\wedge A_1$ are injections in motivic homology. 
\end{enumerate} 
Since $kq\wedge A_1\simeq H\f_2$, the first condition follows immediately. For the second, it is enough to show this in the case that $n=1$. In this case, this is the map 
\[
S^{0,0}\to kq
\]
smashed with $A_1$. Since $H_{**}A_1\cong A(1)^\vee$, the map 
\[
A_1 \to kq\wedge A_1
\]
gives the map 
\[
A(1)^\vee=\m_2\otimes_{\m_2}A(1)^\vee\to A/\kern-0.25em/A(1)^\vee\otimes_{\m_2}A(1)^\vee \cong A^\vee.
\]
Thus, this map is injective, proving that the canonical $kq$-Adams resolution for $A_1$ is an $H\f_2$-Adams resolution. 
\end{proof}

\begin{prop}\label{prop: vanishing line, A(1)}
We have $\Ext_{A^\vee}^{n,t,*}(A(1)) = 0$ for $6n > t +5$ 
\end{prop}

\begin{proof}
Let $I = \{(i,j) \in \z \times \z : i>0, j \geq 0\} \setminus \{(1,0),(1,1),(2,0)\}.$ The motivic May spectral sequence \cite{DI10} converging to $\Ext_A^{***}(A(1))$ has $E_1$-page $\m_2[h_{i,j} | (i,j) \in I\}$. Define $s(i,j) = 1 / stem(h_{i,j})$ to be the slope of the line spanned by $h_{i,j}$ in Adams grading, i.e. with Adams filtration on the vertical axis and stem on the horizontal axis. Then it is easy to check that 
$$\max_{(i,j) \in I} s(i,j) = 1/3, \quad \text{ and } \quad \max_{(i,j) \in I \setminus \{(1,2)\}} s(i,j) = 1/5.$$
Moreover, the maximal $s(i,j)$ for $(i,j) \in I$ is realized by $(i,j) = (1,2)$. Since $h_{1,2}^4 = h_{2} \cdot h_1^2 h_3=  0$ by \cite[Tables 2-3]{DI10}, we conclude that aside from a finite number of $h_{1,2}$ multiples, the May spectral sequence is generated by classes on or below a line of slope $1/5$. Since $h_{1,2}^4=0$, it follows that if $n>\frac 15 (t-n)+1$ then $\Ext^{n,t,*}(A(1))=0$. Thus we conclude that the $\Ext$-group is trivial when $6n>t+5$.
\end{proof}

Combining the previous proposition with Lemma \ref{Lem:A1}, we obtain the following vanishing region in the $kq$-resolution for $A_1$. 

\begin{cor}\label{cor: vanishing line for A_1}
We have $^{kq}E^{n,t,*}_2(A_1) = 0$ for $6n > t+5$ 
\end{cor}

Our goal is to establish a $1/5$-vanishing line for $S^{0,0}$. The idea is to begin with the vanishing line for $A_1$, and use various cofiber sequences to boot-strap our way up from $A_1$ to $S^{0,0}$. 



At this point, we would like to produce a ``$v_1$-Bockstein spectral sequence'' on the $kq$-ASS $E_2$-term, since then we would be able to immediately obtain a 1/5-vanishing line for $Y$. In order for this to occur, we would need a short exact sequence on the $\kqE_1$-terms. Indeed, then we would obtain a long exact sequence on $\kqE_2$-terms from which we could produce a genuine $v_1$-Bockstein spectral sequence. Unfortunately, as we will see below, this does not happen. However, there is a work around which we will explain in due course. The following lemma will be useful. 

\begin{lem}
	The Adams spectral sequences for $kq\wedge \overline{kq}^{\wedge n}\wedge Y$ and $kq\wedge \overline{kq}^{\wedge n}\wedge A_1$ collapse at $E_2$.
\end{lem}
\begin{proof}
	The statement is immediate for $A_1$ since $kq\wedge A_1\simeq H$, and so the Adams spectral sequences for $kq\wedge \overline{kq}^{\wedge n}\wedge A_1$ are concentrated in Adams filtration 0.
	
	For $Y$, note that we have the change of rings isomorphism 
	\[
	\Ext_{A^\vee }(H_{**}(kq\wedge \overline{kq}^{\wedge n}\wedge Y))\cong \Ext_{E(Q_1)}((\overline{A/\kern-0.25em/A(1)}^{\vee})^{\otimes n}).
	\]
	The latter term is given by 
	\[
	\Ext_{E(Q_1)}((\overline{A/\kern-0.25em/A(1)}^{\vee})^{\otimes n})\cong \left(\m_2[v_1]\otimes M_*((\overline{A/\kern-0.25em/A(1)}^{\vee})^{\otimes n};Q_1)\right)\oplus V
	\]
	where $V$ is a free $\m_2$-module concentrated in Adams filtration 0 and $M_*(-;Q_1)$ denotes Margolis homology with respect to $Q_1$. The formulas for the coaction in Theorem \ref{thm: dual steenrod algebra} show that the action of $Q_1$ on $A/\kern-0.25em/A(1)^{\vee}$ is determined by 
	\[
	Q_1(\otau_k) = \oxi_{k-1}^2
	\]
	for $k\geq 2$ and $Q_i\oxi_k=0$ for all $j$. This shows that 
	\[
	M_*(A/\kern-0.25em/A(1)^\vee; Q_1)\cong E(\oxi_2, \oxi_3, \ldots)
	\]
	and hence that 
	\[
	M_*(\overline{A/\kern-0.25em/A(1)}^\vee;Q_1)\cong \overline{E}(\oxi_2, \oxi_3, \ldots):= E(\oxi_2, \oxi_3, \ldots)/\m_2\{1\}.
	\]
	Since the Margolis homology for $(\overline{A/\kern-0.25em/A(1)}^{\vee})^{\otimes n}$ is torsion free over $\m_2$, we can apply the K\"unneth formula to obtain
	\[
	M_*((\overline{A/\kern-0.25em/A(1)}^{\vee})^{ \otimes n};Q_1)\cong \overline{E}(\oxi_2, \oxi_3, \ldots)^{\otimes n}.
	\]
	Thus the Adams $E_2$-term for $kq\wedge \overline{kq}^{\wedge n}\wedge Y$ consists of a $v_1$-torsion free component concentrated in even stems and a free $\m_2$-module concentrated in Adams filtration 0. Since the Adams differentials are $v_1$-linear, it follows that the Adams spectral sequence must collapse at $E_2$.
	\end{proof}
	
	We record for later use the following immediate corollary from the proof of the previous lemma. 
	
	\begin{cor}\label{cor: Margolis homology}
		The $Q_1$-Margolis homology of $A/\kern-0.25em/A(1)$ and $\overline{A/\kern-0.25em/A(1)}$ are respectively given by 
		\[
		E(\oxi_2, \oxi_3, \ldots)
		\]
		and 
		\[
		\overline{E}(\oxi_2, \oxi_3, \ldots).
		\]
		Here and below we write $\overline{E}(x,y,\ldots)$ to denote the quotient $\m_2$-module $E(x,y,\ldots)/\m_2\{1\}$.
	\end{cor}
	
	The preceding lemma shows that in order to get short exact sequence on $\kqE_1$-terms, it is enough if we get a SES in $\Ext$-groups. We now investigate whether or not this is the case. At this point, we find that it is convenient to introduce some notation. 
	
	\begin{defin}
		We write $S_j$ for $\overline{kq}^{\wedge j}$. Note that $kq\wedge S_n\wedge X$ gives the $n$th line in the $E_1$-term of the $kq$-resolution for $X$.  We will also often abbreviate  $\Ext_{A^\vee}^{s,t}(\m_2, H_*(kq\wedge S_j\wedge X))$ by $\Ext_{A(1)^{\vee}}(S_j\wedge X)$.
	\end{defin}
	
	Applying homology to the cofiber sequence for $A_1$, obtain the short exact sequence
	\[
	0\to H_{**}(Y)\to H_{**}(A_1)\to H_{**}(\Sigma^{3,1}Y)\to 0
	\]
	since $v_1$ has Adams filtration 1. This induces a long exact sequence
	\[
	0\to \Ext^{0,t,w}_{A(1)^\vee}(S_n\wedge Y)\to \Ext^{0,t,w}_{A(1)^\vee}(S_n\wedge A_1)\to \Ext^{0,t,w}_{A(1)^\vee}(S_n\wedge \Sigma^{3,1}Y)\to \Ext^{1,t,w}_{A(1)^\vee}(S_n\wedge Y)\to \cdots .
	\]
	In this case, the connecting homomorphism is multiplication by $v_1$. Since $\Ext_{A(1)^\vee}^{s,t,w}(S_n\wedge A_1)=0$ if $s>0$, multiplication by $v_1$ is an isomorphism in this region. Thus we have a short exact sequence 
	\[
	0\to \Ext^{0,t,w}_{A(1)^\vee}(S_n\wedge Y)\to \Ext_{A(1)^\vee}^{0,t,w}(S_n\wedge A_1)\to \Ext_{A(1)^\vee}^{0,t,w}(S_n\wedge \Sigma^{3,1}Y)[v_1^{\infty}]\to 0
	\]
	where the last term denotes the $v_1$-torsion in $\Ext_{A(1)^\vee}^{0,t,w}(S_n\wedge \Sigma^{3,1}Y)$. The second map in this short exact sequence corresponds to the connecting homomorphism 
	\[
	\partial: \kqE_1^{n,t,w}(A_1)\to \kqE_1^{n,t-1,w}(\Sigma^{2,1}Y).
	\]
	This means, unfortunately, that we do not obtain a short exact sequence on $\kqE_1$-terms in general. 
	
	However, we do get a short exact sequence in a region. From Corollary \ref{cor: Margolis homology}, we find that the lowest degree term in the Margolis homology of $(\overline{A/\kern-0.25em/A(1)}^\vee)^{ \otimes n}$ is $\oxi_2^{\otimes n}$. The degree of $\oxi_2$ is 6, and hence the first non-zero term in $M_*((\overline{A/\kern-0.25em/A(1)}^\vee)^{ \otimes n};Q_1)$ is in degree $6n$. Thus we have the following. 
	
	\begin{prop}
		If $t<6n$, then $\Ext^{***}_{A(1)^\vee}(S_n\wedge Y)$ consists only of $v_1$-torsion concentrated in Adams filtration 0. Consequently we have isomorphisms
		\[
		\kqE_1^{n,t,w}\cong \Ext^{0,t,w}_{A(1)^\vee}(S_n\wedge Y).
		\]
		Moreover, this is an identification of cochain complexes with respect to the $d_1$-differential. 
	\end{prop}
	
	\begin{rem2}
		This observation is similar to the one found in \cite[Obs. 1.3]{BBBCX17}.
	\end{rem2}
	
	This proposition then tells us that we do in fact have the following short exact sequence in the region $t<6n$:
	\[
	0\to \kqE_1^{n,t,w}(Y)\to \kqE_1^{n,t,w}(A_1)\to \kqE_1^{n,t-1,w}(\Sigma^{2,1}Y)\to 0.
	\]
	This does lead to a long exact sequence on $E_2$-terms 
	\[
	\cdots \to \kqE_2^{n,t,w}(Y)\to\kqE_2^{n,t,w}(A_1)\to \kqE_2^{n,t-1,w}(\Sigma^{2,1}Y)\to \kqE_2^{n+1,t,w}(Y)\to \cdots . 
	\]
	
	Now in the region $t+5<6n$ we have already established that $\kqE_2^{n,t,w}(A_1)=0$ (Corollary \ref{cor: vanishing line for A_1}). Thus, so long as $t+5<6n$ we have 
	\[
	\kqE_2^{n,t,w}(Y)\cong \kqE_2^{n+1,t-1,w}(\Sigma^{2,1}Y)\cong \kqE_2^{n+1,t-3,w-1}(Y).
	\]
	Since if $t+5<6n$ then $(t-3)+5<6(n+1)$, we can iterate to get a string of isomorphism
	\[
	\kqE_2^{n,t,w}(Y)\cong \kqE_2^{n+j,t-3j,w-j}(Y).
	\]
	For $j\gg 0$, we will have that 
	\[
	t-3j< 4(n+j)
	\]
	and hence reside in the na\"ive vanishing region. Thus we have shown the following proposition.


\begin{prop}\label{prop: vanishing line for Y}
	The $E_2$-term of $kq$-based Adams spectral sequence for $Y$ has a vanishing line of slope 1/5. In particular, if $t+5<6n$, then $\mbox{}^{kq}E_2^{n, t, w}(Y)=0$. 
\end{prop}

Recall next that there is a cofiber sequence 
\begin{equation}\label{eqn: cofiber sequence M to Y}
	\begin{tikzcd}
	\Sigma^{1,1}M\arrow[r,"\eta"] & M \arrow[r] & Y.
\end{tikzcd}
\end{equation}
We will mimic the the preceding argument for this cofibre sequence. In particular, we will see that there is a ``$\eta$-Bockstein spectral sequence'' in a region, similar to the above situation. 
%

Observe that applying $\mbox{}^{kq}E_1^{n,*,*}$ to the above cofiber sequence gives a long exact sequence
\begin{equation}\label{eqn: LES Y to M}
\begin{tikzcd}
	\cdots \arrow[r] & \mbox{}^{kq}E_1^{n,t-1,w-1}(M)\arrow[r,"\cdot \eta"] & \mbox{}^{kq}E_1^{n,t,w}(M)\arrow[r] & \mbox{}^{kq}E_1^{n,t,w}(Y) \\
	\arrow[r,"\partial"] & \mbox{}^{kq}E_1^{n,t-2,w-1}(M)\arrow[r,"\cdot \eta"] & \mbox{}^{kq}E_1^{n,t-1, w}(M)\arrow[r] & \cdots . 
\end{tikzcd}
\end{equation}
Whenever multiplication by $\eta$ is 0, this long exact sequence breaks up into short exact sequences which give rise to long exact sequences on $\mbox{}^{kq}E_2$. The following lemma says that this occurs above a line of slope $1/5$.
\begin{lem}
	In $\mbox{}^{kq}E_1^{n, t, w}(M)$, multiplication by $\eta$ is trivial so long as $t<6n-4$.
\end{lem}
\begin{proof}
	Consider the cofiber sequence for the mod 2 Moore spectrum. This induces a long exact sequence 
	\[
	\cdots \to \kqE_1^{n,t,w}(S^0)\to \kqE_1^{n,t,w}(S^0)\to \kqE_1^{n,t,w}(M)\to \kqE_1^{n,t-1,w}(S^0)\to \cdots .
	\]
	Recall from Section \ref{Section:Analysis} that 
	\[
	\kqE_1^{n,*,*}(S^0)\cong \bigoplus_{\ell(I)=n}\Ext_{A(1)^\vee}^{*,*}(\Sigma^{4|I|}H\underline{Z}_I)
	\]
	where the sum runs over multi-indices $I$ of positive integers whose length is $n$, and where
	\[
	H\ul{\z}_I:= H\ul{\z}_{i_1}\otimes \cdots \otimes H\ul{\z}_{i_n}.
	\]
	It follows from Lemma \ref{Lem:ExtOfOtimes} and Theorem \ref{Lem:ExtOfHZn} that when 
	\[
	t-s+4< \kappa(I):= \begin{cases}
 	4|I|+2|I|-\alpha(I) & |I| \equiv 0 \mod 2,\\
 	4|I|+2|I|-\alpha(I) - 1 & |I|\equiv 1\mod 2,
 \end{cases}
	\]
	where 
	\[
	\alpha(I) = \sum_{j=1}^n\alpha(i_j)
	\]
	that 
	\[
	\Ext_{A(1)^\vee}^{s,t,w}(\Sigma^{4|I|}H\ul{\z}_I)
	\]
	consists only of $h_0$-towers and $v_1$-torsion classes in Adams filtration zero. Thus multiplication by 2 is injective in this region. Hence the long exact sequence above becomes short exact sequence
	\[
	0\to \Ext_{A(1)^\vee}^{s,t,w}(\Sigma^{4|I|}H\ul{\z}_I)\to \Ext_{A(1)^\vee}^{s,t,w}(\Sigma^{4|I|}H\ul{\z}_I)\to \Ext^{s,t,w}_{A(1)^\vee}(A(0)\otimes \Sigma^{4|I|}H\ul{\z}_I)\to 0
	\]
	for $s>0$. This shows that for $s>0$ and $t-s<6|I|-\alpha(I)$, the modules $\Ext_{A(1)^\vee}^{s,t,w}(A(0)\otimes H\ul{\z}_I)$ consist only of simple $\eta$-torsion. Since the simple $v_0$-torsion classes are also simple $\eta$-torsion, we can actually relax the condition so that $s\geq 0$. 
	
	Since the Adams spectral sequence computing $\kqE_1^{n, *, *}(M)$ collapses at $E_2$, we conclude from the above, that $\kqE_1^{n, t, w}(M)$ consists only of simple $\eta$-torsion whenever 
	\[
	t<\min\{\kappa(I)\mid \ell(I)=n\}.
	\]
	Note that the lowest norm multi-index is 
	\[
	I = (\underbrace{1, 1, \ldots, 1}_{\text{$n$ 1s}}).
	\]
	It follows from Lemma \ref{Lem:ExtOfOtimes} that 
	\[
	6n<\min\{\kappa(I)\mid \ell(I)=n\}.
	\]
	This proves the lemma. 
\end{proof}

We can now obtain a vanishing line for $M$. 

\begin{prop}
	For $t<6n-4$, we have short exact sequences
	\[
	0\to \kqE_1^{n,t,w}(M)\to \kqE_1^{n,t,w}(Y)\to \kqE_1^{n,t-2, w-1}(M)\to 0.
	\]
\end{prop}
\begin{proof}
	From the cofiber sequence \eqref{eqn: cofiber sequence M to Y}, we have the long exact sequence \eqref{eqn: LES Y to M}. In the case $t+4<6n$, it follows from the previous lemma that multiplication by $\eta$ is trivial. The result follows. 
\end{proof}

Since the maps in the short exact sequence of the proposition commute with the $d_1$-differential, we obtain the following corollary. 

\begin{cor}
	For $t<6n-4$, we have a short exact sequence of cochain complexes
	\[
		0\to \kqE_1^{n,t,w}(M)\to \kqE_1^{n,t,w}(Y)\to \kqE_1^{n,t-2, w-1}(M)\to 0
	\]
	where the differential is the Adams $d_1$-differential. Thus, in this region, we get a long exact sequence
	\begin{equation}\label{eqn: LES on E2}
	\cdots \to \kqE_2^{n,t,w}(M)\to \kqE_2^{n,t,w}(Y)\to \kqE_2^{n,t-2,w-1}(M)\to \kqE_2^{n+1, t,w}(M)\to \cdots. 
	\end{equation}
	Moreover, the connecting homomorphism detects multiplication by $\eta$ in $\pi_{**}$. 
\end{cor}

We will use this to prove that the $E_2$-page of the $kq$-resolution for the mod 2 Moore spectrum has a vanishing line of slope $1/5$. Fix a pair $(t,w)$. Then from the long exact sequence \eqref{eqn: LES on E2} and the vanishing line of slope $1/5$ for $Y$, we see that $\kqE_2^{n,t,w}(Y)=0$ whenever $t+5<6n$. Hence there is a sequence of isomorphisms 
\[
\kqE_2^{n,t-2,w-1}(M)\cong \kqE_2^{n+1,t,w}(M)\cong\cdots \cong \kqE_2^{n+j, t+2(j-1), w+j-1}(M)\cong \cdots  
\]
for all $j\geq 0$. Eventually, these terms will lie in the region above the na\"ive vanishing line of slope $1/3$. Indeed, for $j\gg0$, we have 
\[
t+2(j-1)< 4(n+j). 
\]
Therefore all of the groups in this sequence vanish. This shows that whenever $t+5<6n$, we have that 
\[
\kqE_2^{n+j, t+2(j-1),w+j-1}(M)=0
\]
for all $j\geq 0$. Thus, we have shown the following. 

\begin{cor}\label{cor: vanishing for Moore spectrum}
	For $t+7<6n$, we have that 
	\[
	\kqE_2^{n,t,w}(M)=0.
	\]
\end{cor}

We are now in a position to prove the vanishing line for the sphere. In light of Proposition \ref{Prop:kqDifferentials}, the only non-trivial classes on $E_\infty^{n,t,w}(S^0)$ are detected by elements of Adams filtration 0 or 1. For this reason, it suffices to show a vanishing line for the \emph{algebraic $kq$-resolution} for the sphere. This is a spectral sequence which is obtained by applying the functor $\Ext^{s,t,w}_{A^\vee}(-)$ to the $kq$-Adams resolution. It is a convergent spectral sequence of the form
\[
\algE_1^{n,s,t,w}(X) = \Ext_{A^\vee}^{s,t,w}(H_*(kq\wedge \overline{kq}^{\wedge n}\wedge X))\implies \Ext_{A^\vee}^{s+n,t,w}(X).
\]
Our argument is an adaption of the one found in \cite{Mah84}.

\begin{rem2}
The same techniques applied above can also be used to derive a 1/5 vanishing line for the algebraic $E_2$-term for $M$. In particular, we have that $\algE_2^{n,0,t,w}(M)=0$ whenever $t+7<6n$.
\end{rem2}

In \cite{Mah84} Mahowald defines a function $a(j)$ whose values are given by the following table:
\begin{center}
\begin{tabular}{c | c | c | c | c }
$a(j)$ & $0$ & $-2$ & $-2$ & $-1$ \\
$j\mod 4$ & $0$ & $1$ & $2$ & $3.$	
\end{tabular}
\end{center}

In the arguments below we make the following definition for ease of notation. 


\begin{prop}[Compare with Prop 3.6 of \cite{Mah84}]
	If $t<6j+a(j)$ , then:
	\begin{enumerate}
		\item If $a\in \Ext^{0,t,*}(H_*S_j)$, then either $h_0^ia\neq 0$ for all $i$ or $h_0a=0$;
		\item If $a\in \Ext^{s,t,*}(H_*S_j)$, $a\neq 0$, then $a = h_0^s a'$ for some $a'$.
	\end{enumerate}
\end{prop}
\begin{proof}
	We can conclude this our calculations in Section \ref{sec: kq-cooperations}.
\end{proof}

\begin{defin}
	For a motivic spectrum $X$, let $\TExt_{A(1)^\vee}^{s,t,*}(S_j\wedge X)$ denote the subspace of $\Ext_{A(1)^\vee}^{***}(S_j\wedge X)$ spanned by the classes $a$ for which $h_0^ia=0$ for some $i$. Let $\overline{C}_j^{s,t,*}(X) := \TExt^{s,t,*}_{A(1)^\vee}(S_j\wedge X)$.
\end{defin}

\begin{prop}\label{prop: V(0) splitting}
	For $t<6j+a(j)$ we get a short exact sequence of cochain complexes
	\[
	0\to \algE_1^{j,0,t,*}(S^{0,0})\to \algE_1^{j,0,t,*}(M)\to \overline{C}_j^{j,0,t,*}(S^{1,0})\to 0
	\]
\end{prop}
\begin{proof}
	The proof is the same as the proof of \cite[Prop. 3.7]{Mah84}. The cofiber sequence
	\[
	S^{0,0}\to M\to S^{1,0}
	\]
	induces a long exact sequence in $\Ext(S_j\wedge -)$. In particular, we have 
	\[
	\Ext^{0,t,*}_{A(1)^\vee}(S_j)\to \Ext^{0,t,*}_{A(1)^\vee}(S_j\wedge M)\to \Ext^{0, t, *}_{A(1)^\vee}(S_j\wedge S^{1,0})\to \Ext^{1,t,*}_{A(1)^\vee}(S_j)\to \cdots.
	\]
	The connecting homomorphism is given by multiplication by $h_0$. By the previous proposition, in the range $t<6j+a(j)$, the group $\TExt$ is precisely the kernel of multiplication by $h_0$. The proposition follows.
\end{proof}

\begin{prop}\label{prop: cohomology with bars or no bars}
	For $t<6j+a(j)-4$, we have $H^*(\overline{C}_\bullet^{0,t,w}(S^0)) = \algE_2^{*,0,t,w}(S^0)$ and $\algE_2^{*,s,s+t,w}(S^0) = 0 = H^*(\overline{C}_\bullet^{s,s+t,w}(S^0))$ for $s>0$.
\end{prop}
\begin{proof}

The classical analog of this result (cf. \cite[Proposition 3.8]{Mah84}) follows from the proof of Theorem 5.14 in \cite{Mah81}. In particular, the $h_0$-towers in this range form an acyclic complex with respect to the $d_1$-differential classically. The motivic version follows from Betti realization, $\tau$-linearity of the $d_1$-differential,  and motivic weight considerations (Corollary \ref{Cor:Half1}). 

%
%
\end{proof}

We deduce the following. 

\begin{prop}[compare with Prop 3.9\cite{Mah84}]
	If $t<6j+a(j)-4$, then there is an exact sequence
	\[
	\cdots \to H^j(C^{0,t,*}_\bullet(S^0))\to \algE_2^{j,0,t,*}(M)\to \algE_2^{j,0,t,*}(S^1)\to \algE_2^{j+1,0,t,*}(S^0)\to \cdots 
	\]
\end{prop}
\begin{proof}
	From Proposition \ref{prop: V(0) splitting} we have the short exact sequence of cochain complexes
	\[
	0\to \algE_1^{\bullet, 0,t,*}(S^0)\to \algE_1^{\bullet, 0,t,*}(M)\to \overline{C}_\bullet^{0,t,*}(S^1)\to 0
	\]
	and from this it follows that we get a long exact sequence in cohomology. In the range $t<6n+a(n)-4$, it follows from Proposition \ref{prop: cohomology with bars or no bars} we can remove a bar. This gives the desired long exact sequence.
\end{proof}

From these, it follows that $^{kq}E_\infty(S^0)$ has a vanishing line of slope $1/5$. 

\begin{thm}\label{Thm:kqVanishing}
	We have $^{kq}E_2^{n,t,*}(S^0) = 0$ if $t+7<6n$. 
\end{thm}
\begin{proof}
	The previous result implies that if $\algE_2^{n', 0, t,*}(M)=0$ for all $n'\geq n$, then $\algE_2^{n',0,t,*}(S^{1,0})\cong \algE_2^{n'+1,0,t,*}(S^{0,0})$ for all $n'>n$. In particular, if we have fixed $t$, then in the long exact sequence of the previous proposition, the groups $\algE_2^{n, 0,t,*}(M)$ are eventually all zero. For example, it follows from Corollary \ref{cor: vanishing for Moore spectrum} that  once $t<6n-7$, these groups are all 0 for later $n'$. In other words, we have 
	\[
	\algE_2^{n',0,t-1,*}(S^{0,0})\cong \algE_2^{n', 0, t, *}(S^{1,0}) \cong \algE_2^{n'+1,0,t,*}(S^{0,0})
	\]
	so long as $t<6n'-7$. 
	
Now from the 1/3-vanishing line, we know that $\algE_2^{n, 0, t,*}(S^0) = 0$ whenever $t<4n$. Combining these observations, we can ``push'' the vanishing of $\algE_2^{n,0,t,*}(S^0)$. For example, consider $t= 4n$. Clearly $4n<6n-7$, and so we have the isomorphism
	\[
	\algE_2^{n, 0, 4n,*}(S^0)\cong \algE_2^{n+1,0,4n-1,*}(S^0)
	\]
	and the latter group is trivial as $4n-1<4n+4$. More generally, if $4n\leq t<6n-14$, then we have the string of isomorphisms
	\[
	\algE_2^{n,0,t,*}(S^0)\cong \algE_2^{n+1,0,t-1,*}\cong \cdots 
	\]
	and at some point the topological degree falls within the na\"ive vanishing region above a line of slope $1/3$, and so all of these groups are trivial. 
	
	From Proposition \ref{Prop:kqDifferentials}(1) above, we know that when $6n>t+7$ the group $^{kq}E_2^{n,t,*}(S^0)$ could only possibly possess a class which was detected by an element of Adams filtration 0 or 1. Now a class in Adams filtration 1 necessarily must be connected to one in Adams filtration 0 by multiplication by $h_0$ or $h_1$. But the argument above shows that an element of Adams filtration 0 cannot appear in the region $6n>t+7$.
\end{proof}

\section{Main theorem, $v_1$-periodicity, and $\eta$-periodicity}\label{Sec:MR}

In this section, we state our main result on the $kq$-resolution and state its two main applications to motivic periodicity.

\subsection{The main theorem}
We state our main computational result in the following theorem. 

\begin{thm}\label{Thm:kqLines} We have the following:
\begin{enumerate}
\item The $0$-line of the $kq$-resolution is given by 
$$E^{0,*,*}_\infty \cong \m_2[h_0,h_1,v_1^4]/(h_0h_1, h_0v^4_1, \tau h_1^3)$$
where $|h_0| = (0,0)$, $|h_1| = (1,1)$, and $|v^4_1| = (8,4)$. 
\item The $1$-line of the $kq$-resolution is given by 
$$E^{1,*,*}_\infty \cong \bigoplus_{k \geq 0} \Sigma^{4k} \z/2^{\rho(k)}[\tau] \oplus \m_2[h_1,v^4_1]/(h_1^3 \tau)\{y\},$$
where $|y| = (9,5)$ and $\rho(k)$ is the $2$-adic valuation of $8k$. All of these classes are $v_1$-periodic. 
\item $E^{n,t,u}_\infty = 0$ whenever $6n > t+7$.
\end{enumerate}
\end{thm}

\begin{proof}
Parts $(1)$ and $(2)$ follow from Propositions \ref{Prop:kqDifferentials} and \ref{Prop:kqDiffl2}. Part $(3)$  is Theorem \ref{Thm:kqVanishing}. 
\end{proof}

\subsection{The $\eta$-local sphere}
Recall that $\eta : S^{1,1} \to S^{0,0}$ is not nilpotent. 

\begin{defin}
Let $X$ be a motivic spectrum with a non-nilpotent self-map $v : X \to \Sigma^{-r,-s} X$. We define the \emph{$v$-telescope of $X$} to be the colimit
$$v^{-1} X := \colim ( X \overset{v}{\to} \Sigma^{-r,-s} X \overset{v}{\to} \Sigma^{-2r,-2s}X \overset{v}{\to} \cdots ).$$
If $X = S^{0,0}$, we will refer to a $v$-telescope $v^{-1}S^{0,0}$ as the $v$-local sphere. 
\end{defin}

The motivic homotopy groups of $\pi_{**}(\eta^{-1}S^{0,0})$ were conjectured by Guillou-Isaksen in \cite[Conj. 1.3(2)]{GI15}. This conjecture was confirmed by Andrews-Miller in \cite{AM17}. This serves as an ``exotic" motivic analog of the $v_1$-periodic computation of the previous subsection. 

\begin{thm}\label{Thm:EtaInv}
The motivic stable homotopy groups of the $\eta$-local sphere are 
$$\pi_{**}(\eta^{-1}S^{0,0}) \cong \f_2[h_1^{\pm 1},v_1^4]\{x,y\}$$
where $|\eta| = (1,1)$, $|v_1^4| = (8,4)$, $x$ is detected by the generator of $\kqE_\infty^{0,0,0}$, and $y$ is detected by the generator of $\kqE_\infty^{1,9,5}$.  
\end{thm}

\begin{proof}
We claim that $\pi_{**}(S^{0,0}[\eta^{-1}])$ is detected in the $0$- and $1$-lines of the $kq$-resolution. If this holds, then the theorem is clear from inspection of the $0$- and $1$-lines described in Theorem \ref{Thm:kqLines}. 

By part $(2)$ of Proposition \ref{Prop:kqDifferentials}, every $\tau$-torsion class in $kq$-Adams filtration greater than one has Adams filtration zero or one, is killed by a differential, or supports a nontrivial differential. The classes in Adams filtration zero correspond to the classes in the zero line of the Adams spectral sequence for a connective cover of $bo$ or $bsp$ (placed in appropriate tridegrees). Within a fixed filtration, then, there can be at most two powers of $\eta$ detected. Therefore an infinite $\eta$-tower above $kq$-Adams filtration one would be detected along a line of slope at least $1/2$ in the $kq$-resolution, contradicting the vanishing line of slope $1/3$ proven in Lemma \ref{Lem:OneThird}. 
\end{proof}

\begin{rem2}
In the terminology of Andrews \cite{And18} and Gheorghe \cite{Ghe17b}, the computation of $\pi_{**}(\eta^{-1}S^{0,0})$ completely identifies all $w_0$-periodic classes in the motivic stable stems. 

Gheorghe-Isaksen-Krause-Ricka have constructed a $\c$-motivic modular forms spectrum $\mathit{mmf}$ \cite{GIKR18} which serves as a computational analog of the classical topological modular forms spectrum $\mathit{tmf}$ \cite{DFHH14}. The $\mathit{mmf}$-based Adams spectral sequence might serve as useful tool for understanding the $v_2$-periodic and $w_1$-periodic motivic stable stems. 
\end{rem2}

\subsection{The $v_1$-periodic $\c$-motivic stable stems}
Our goal in this section is to identify the $v_1$-periodic part of the motivic stable stems. We begin with some definitions following \cite[Sec. 6]{Mah81}.

\begin{defin}
Let $\gamma \in \pi_{j,k}(X)$ be represented by $\gamma: S^{j,k} \to X$ and let $Y := S^{0,0}/(2,\eta)$. Then there are potentially four maps
\begin{enumerate}
\item $\gamma_i^{\#}: \Sigma^{j-3,k-1} Y \overset{p_1}{\to} S^{j,k} \overset{\gamma}{\to} X$ where $p_1$ is the collapse onto the top cell,
\item $\Sigma^{j-2,k-1} Y \overset{p_2}{\to} \Sigma^{j,k} S/2 \overset{\gamma^{\#}}{\to} X$ where $p_2$ is the collapse onto the top two cells and $\gamma^{\#}$ is an extension of $\gamma$ (if it exists),
\item $\Sigma^{j-1,k}Y \overset{p_3}{\to} \Sigma^{j-2,k-1} B^4_2 \overset{\gamma^{\#}}{\to} X$ where $B^4_2$ is the simplicial $2$-coskeleton of the $4$-skeleton of $B_{gm}\mu_2$ and $p_3$ and $\gamma^{\#}$ are analogously defined, and
\item $\Sigma^{j,k} Y \overset{\gamma^{\#}}{\to} Y$. 
\end{enumerate}
If a map of type $(i)$ exists and the composite
$$\Sigma^{j-4+i+2\ell, k-1+h+\ell} Y \overset{v_1^\ell}{\to} \Sigma^{j-4+i,k-1+h} Y \overset{\gamma^{\#}}{\to} X$$
is essential for all $\ell \geq 0$ for all $\gamma^{\#}$, then we say that $\gamma$ is $v_1$-periodic of type $i$. 
\end{defin}

\begin{exm}(compare with \cite[Exm. 6.2]{Mah81}
The Atiyah-Hirzebruch spectral sequence for $Y$ and Betti realization imply that $\eta \in \pi_{1,1}(S^{0,0})$ is $v_1$-periodic of type $2$. Similarly, one can show that $\nu \in \pi_{3,2}(S^{0,0})$ is $v_1$-periodic of type $3$. The generator of the image of the two torsion in the image of the motivic unitary J-homomorphism (see \cite{HKO11} and below) in $\pi_{8k-1,4k}(S^{0,0})$ is $v_1$-periodic of type $1$. 
\end{exm}

\begin{thm}\label{Thm:v1periodic}
The only $v_1$-periodic classes in $\pi_{**}(S^{0,0})$ are those described in parts $(1)$ and $(2)$ of Theorem \ref{Thm:kqLines}.
\end{thm}

\begin{proof}
The proof is identical to the proof of \cite[Thm. 6.3]{Mah81}. Above $kq$-Adams filtration one, every class is $v_1$-torsion. There can be at most two classes in a $v_1$-periodic family detected in a fixed filtration, so any $v_1$-periodic family is detected on or above a line of slope $1/4$. This contradicts the vanishing line of slope $1/5$ in part $(c)$ of Theorem \ref{Thm:kqLines}. 
\end{proof}

\section{Motivic Telescope Conjectures}\label{Sec:Tel}

In this section, we place the computations of Theorem \ref{Thm:EtaInv} and Theorem \ref{Thm:v1periodic} in the context of chromatic motivic homotopy theory. 
\subsection{Localization functors}
We begin with some motivic analogs of the results from \cite{Mil92}. Let $E$ be any motivic spectrum.

\begin{defin}(compare with \cite{Bou79})
\leavevmode
\begin{enumerate}
\item A motivic spectrum $W$ is \emph{$E$-local} if and only if $[T,W]=0$ for every $E$-acyclic spectrum  $T$.
\item A map $f : X \to Y$ is an \emph{$E$-equivalence} if and only if $E_*f$ is an isomorphism. 
\end{enumerate}
\end{defin}

\begin{thm}\cite{Bou79}\cite{Hor06}
For any motivic spectra $E$ and $X$, there is a unique (up to canonical equivalence) $E$-equivalence  $\eta : X \to L_E X$ where $L_E X$ is an $E$-local motivic spectrum. 
\end{thm}

The motivic spectrum $L_E X$ is called the \emph{Bousfield localization of $X$ with respect to $E$}. 

\begin{defin}(compare with \cite[Def. 3]{Mil92})
Let $\ca$ be a set of motivic spectra. 
\begin{enumerate}
\item A motivic spectrum $W$ is \emph{finitely $\ca$-local} if and only if $[\Sigma^{m,n} A , W]=0$ for every $A \in \ca$ and every $m,n \in \z$. 
\item A motivic spectrum $Z$ is \emph{finitely $\ca$-acyclic} if and only if $[Z,W]=0$ for every finitely $\ca$-local motivic spectrum $W$.
\item A map $f : X \to Y$ is a \emph{finite $\ca$-equivalence} if and only if its mapping cone is finitely $\ca$-acyclic. 
\end{enumerate}
\end{defin}

The proof of \cite[Thm. 4]{Mil92} carries over to the motivic setting without change to prove the following:

\begin{thm}
For any set $\ca$ of finite motivic spectra and any motivic spectrum $X$, there is a unique (up to canonical equivalence) finite $\ca$-equivalence $\eta : X \to L_{\ca}^f X$ where $L_\ca^f X$ is an $\ca$-local spectrum. 
\end{thm}

The motivic spectrum $L_\ca^f X$ is called the \emph{finite localization of $X$ with respect to $\ca$}. If $\ca$ is the set of finite $E$-acyclic spectra for some motivic spectrum $E$, then we will write $L_E^f$ for $L_\ca^f$. Any $E$-local spectrum is finitely $E$-local, so we obtain a unique morphism $L^f_E X \to L_E X$ under $X$. This gives rise to a natural transformation $L_E^f \to L_E$, and the discussion from \cite[Sec. 2]{Mil92} carries over to the motivic setting without alteration to give the following results.

\begin{cor}
The natural transformation $L^f_E \to L_E$ is an equivalence if and only if $E$ is smashing (i.e. the map $L_EX \cong L_E( X \wedge S^{0,0}) \to X \wedge L_E S^{0,0}$ is an equivalence for all $X$) and the natural map $L_E^f S^{0,0} \to L_E S^{0,0}$ is an equivalence. 
\end{cor}

\begin{cor}
Finite $\ca$-localization is Bousfield localization with respect to the spectrum $L_\ca^f S^{0,0}$. 
\end{cor}

We conclude our discussion of localization with the following definition and lemma.

\begin{defin}
The \emph{Bousfield class of $E$}, denoted $\langle E \rangle$, is the set of spectra $X$ such that $E_{**}(X) = 0$. 
\end{defin}

We record one more lemma, which follows immediately from the definitions. 

\begin{lem}\label{Lem:BL}
If $\langle E \rangle = \langle F \rangle$, then $L_E \simeq L_F$ and $L_E^f \simeq L_F^f$. 
\end{lem}

\subsection{Recollection of the classical Telescope Conjecture}

In order to motivate the motivic Telescope Conjecture, we recall three equivalent formulations of the classical Telescope Conjecture. Let $K(n)^{cl}$ denote the classical $n$-th Morava K-theory and let $E(n)^{cl}$ denote the classical $n$-th Johnson-Wilson theory. Recall that a classical finite complex $X$ is of \emph{type $n$} if $K(i)^{cl}_*(X) = 0$ for $i < n$ and $K(n)_*^{cl}(X) \neq 0$. If $X$ is of type $n$, then there exists a non-nilpotent $v_n$-self-map $v : \Sigma^k X \to X$ for some $k \geq 0$ which induces an isomorphism in $K(n)^{cl}$-homology. 

The classical Telescope Conjecture first appeared in \cite[10.5]{Rav84}, where it had the following form:

\begin{conj}[Classical Telescope Conjecture (Telescopic Formulation)]
Let $X$ be any finite complex of type $n$ with non-nilpotent $v_n$-self-map $v : \Sigma^k X \to X$. Then $\langle v^{-1} X \rangle$ depends only on $n$, and $\langle v^{-1} X \rangle = \langle K(n)^{cl} \rangle$. 
\end{conj}

Let $K(\leq n)^{cl} = \bigvee_{i=0}^n K(i)^{cl}$. Miller provided two new formulations of the classical Telescope Conjecture in \cite[Sec. 3]{Mil92}.

\begin{conj}[Classical Telescope Conjecture (Localization Formulation)]
The natural transformation $L^{f}_{K(\leq n)^{cl}} \to L_{K(\leq n)^{cl}}$ is an equivalence. 
\end{conj}

\begin{conj}[Classical Telescope Conjecture (Smashing Formulation)]
The map $L^f_{K(\leq n)^{cl}} S^0 \to L_{K(\leq n)^{cl}} S^0$ is an equivalence.
\end{conj}

These three formulations were shown to be equivalent in \cite{Mil92}. There are two key ideas in this identification. Using the fact that all finite localizations are smashing plus the fact that $L_{K(\leq n)^{cl}}$ is smashing \cite{Rav16}, Miller showed that the Localization and Smashing Formulations are equivalent. Then, Miller used the periodicity theorem, asymptotic uniqueness of $v_n$-self maps, and a thick subcategory argument to identify the Telescopic and Localization Formulations. In particular, one needs the following identification.

\begin{prop}
If $X$ is a $K(n-1)^{cl}$-acyclic finite complex with $v_n$-self-map $v : \Sigma^k X \to X$, then the map $X \to v^{-1} X$ is a finite $K(n)^{cl}$-localization. 
\end{prop}

\subsection{The motivic Telescope Conjecture}

Our goal now is to propose motivic analogs of the three formulations of the classical Telescope Conjecture described above. We begin with some background from chromatic motivic homotopy theory. 

First, recall that Borghesi defined $\c$-motivic Morava K-theories $K(n)$ in \cite{Bor03} satisfying
$$K(n)_{**} \cong \m_2[v_n^{\pm 1}]$$
where $|v_n| = (2^{n+1}-2,2^n-1)$. These were also defined by Hornbostel in \cite{Hor06}. 

\begin{defin}
A finite motivic spectrum $X$ is of \emph{classical type $n$} if $K(i)_{**}(X) = 0$ for $i < n$ and $K(n)_{**}(X) \neq 0$.  
\end{defin}

\begin{lem}\label{Lem:ClassicalTypen}
If $X$ is of classical type $n$ and $H_{**}(X)$ is $\tau$-torsion free, then $X$ admits a non-nilpotent $v_n$-self map $v : X \to \Sigma^{-r,-s} X$, i.e. a self-map $v$ which induces an isomorphism on $K(n)$-homology. Moreover, any two such $v$ coincide after raising them to suitable powers. 
\end{lem}

\begin{proof}
The Betti realization $Re(X)$ is a classical spectrum of (classical) type $n$, so it admits a non-nilpotent $v_n$-self-map $v^{cl} : X \to \Sigma^{-r} X$ which induces an isomorphism on $K(n)^{cl}$-homology. Betti realization is strong symmetric monoidal, so the map
$$\pi_{s,t}(X) \to \pi_s(Re(X))$$
is a homomorphism of graded rings. In particular, there exists some non-trivial element $v \in \pi_{-r,?}(X)$ which maps to $v^{cl}$. Since $v^{cl}$ is non-nilpotent, the element $v$ must also be non-nilpotent. The isomorphism $H_{**}(X)[\tau^{-1}] \cong H_{*}(Re(X))[\tau^{\pm 1}]$ and the universal coefficient theorem show that $v$ induces an isomorphism in $K(n)$-homology. Asymptotic uniqueness follows similarly from Betti realization. 
\end{proof}

\begin{lem}\label{Lem:ClassicalAU}
Let $X$ and $Y$ be finite motivic spectra of classical type $n$ with $\tau$-torsion free homology with non-nilpotent $v_n$-self-maps $\psi$ and $\phi$, respectively, and let $f : X \to Y$ be any map. Then there are positive integers $i$ and $j$ for which the diagram
\[
\begin{tikzcd}
X \arrow{r}{f} \arrow{d}{\psi^i} & Y \arrow{d}{\phi^j} \\
\Sigma^{-r,-s} X \arrow{r}{\Sigma^{-r,-s} f} & \Sigma^{-r,-s} Y
\end{tikzcd}
\]
commutes. 
\end{lem}

\begin{proof}
The result follows from Betti realization and the analogous classical result of Hopkins and Smith \cite{HS98}. 
\end{proof}

\begin{rem2}
As a first guess, one might take $K(\leq n) := \bigvee_{i = 0}^n K(i)$ and conjecture that $L^f_{K(\leq n)} \to L_{K(\leq n)}$ is an equivalence. The previous lemma allows one to identify $L^f_{K(\leq n)}X \simeq v^{-1} X$ for any finite motivic spectrum $X$ of classical type $n$ with $\tau$-torsion free homology, where $v$ is a non-nilpotent $v_n$-self map of $X$. This provides part of the input needed to relate the Telescopic and Localization Formulations of the resulting Telescope Conjectures.

However, there are still several issues with this approach. We do not know if $L_{K(\leq n)}$ is smashing, and we do not expect a thick subcategory argument to hold in this context. Moreover, our calculations in Section \ref{Sec:TCv1} suggest that the motivic Telescope Conjecture at height one should carry more information that that which is detected by $K(0) \vee K(1)$. 
\end{rem2}

Gheorghe constructed exotic $\c$-motivic Morava K-theory spectra $K(w_n)$ in \cite[Cor. 3.14]{Ghe17b} with
$$K(w_n)_{**} \cong \f_2[w_n^{\pm 1}]$$
where $|w_n| = (2^{i+2}-3,2^{i+1}-1)$. These detect certain ``exotic" forms of periodicity in the $\c$-motivic stable stems. For example, the spectrum $K(w_0)$ detects $\eta$-periodicity, and $K(w_1)$ detects the non-nilpotent self-map $w^4_1 : \Sigma^{20,12} S^{0,0}/\eta \to S^{0,0}/\eta$ which Andrews used to produce various exotic periodic families in \cite{And18}.

Gheorghe's exotic Morava K-theories fit into a larger family of exotic Morava K-theories $K(\beta_{ij})$ defined by Krause in \cite{Kra18}. The following proposition specializes \cite[Prop. 6.9]{Kra18} to the case $p=2$. 

\begin{prop}\cite{Kra18}
For each $i> j \geq 0$, there is a $\c$-motivic $2$-complete cellular $C\tau$-module $K(\beta_{ij})$ with 
$$K(\beta_{ij})_{**} \simeq \f_2[\alpha_{ij}, \beta_{ij}^{\pm 1}]/(\alpha_{ij}^2 = \beta_{ij}),$$
with $|\alpha_{ij}| = (2^{j+1}(2^i-1)-1, 2^j(2^i-1))$ and $|\beta_{ij}| = (2^{j+2}(2^i-1)-2,2^{j+1}(2^i-1))$. For $j=0$, they admit an $E_\infty$-ring structure. 
\end{prop}

Krause observes that $K(w_{n-1}) \simeq K(\beta_{n,0})$. Following \cite[Prop. 4.35]{Kra18} and \cite[Pg. 124]{Kra18}, we define
$$d_{ij} := \begin{cases} \dfrac{1}{2^{i+1}-2} \quad & \text{ if } j=-1, \\
\dfrac{2^{j+1}(2^i-1)}{2^{j+2}(2^i-1)-2} \quad & \text{ else.}
\end{cases}
$$
We set $K(\beta_{i,-1}) := K(i)$ for $i \geq 0$. 

\begin{defin}\label{Def:Type}
A finite motivic spectrum $X$ is of type $(m,n)$ if $H_{**}(X)$ is $\tau$-torsion free, and $K(\beta_{ij})_{**}(X) = 0$ for all $i > j \geq -1$ with $d_{ij} > d_{mn}$ and $K(\beta_{mn})_{**}(X) \neq 0$.
\end{defin}

\begin{rem2}
Unlike in the classical setting, it is not known if the thick subcategories defined by vanishing of $K(\beta_{ij})$ are all of the prime thick subcategories of finite $p$-local cellular motivic spectra. For further discussion, see \cite[Introduction]{Kra18}. 
\end{rem2}

\begin{prop}\label{Prop:NonNilp}
Suppose that $X$ is a finite motivic spectrum of type $(m,n)$. Then there is a non-nilpotent self-map $v : \Sigma^{|\theta|} X \overset{\theta}{\to} X$ of slope $d_{mn}$ which induces an isomorphism on $K(\beta_{mn})_{**}(X)$. Moreover, any two such $v$ coincide after raising them to suitable powers. 
\end{prop}

\begin{proof}
This follows from the combination of Lemma \ref{Lem:ClassicalTypen} and \cite[Thm. 6.12]{Kra18}. 
\end{proof}

\begin{defin}
We will refer to a non-nilpotent self-map as described in Proposition \ref{Prop:NonNilp} as a \emph{self-map of type $(m,n)$}. 
\end{defin}

\begin{exm}
\leavevmode
\begin{enumerate}
\item A non-nilpotent self-map of type $(m,-1)$ is a $v_m$-self-map. 
\item A non-nilpotent self-map of type $(m,0)$ is a $w_{m-1}$-self-map. 
\end{enumerate}
\end{exm}

Applying Lemma \ref{Lem:ClassicalAU} and \cite[Thm. 6.12]{Kra18} gives the following corollary. 

\begin{cor}\label{Cor:Unique}
Let $X$ and $Y$ be finite motivic spectra of type $(m,n)$ with non-nilpotent self-maps of type $(m,n)$ $\psi$ and $\phi$, respectively, and let $f : X \to Y$ be any map. Then there are positive integers $i$ and $j$ for which the diagram
\[
\begin{tikzcd}
X \arrow{r}{f} \arrow{d}{\psi^i} & Y \arrow{d}{\phi^j} \\
\Sigma^{-r,-s} X \arrow{r}{\Sigma^{-r,-s} f} & \Sigma^{-r,-s} Y
\end{tikzcd}
\]
commutes. 
\end{cor}

The proof of \cite[Prop. 14]{Mil92} now carries over to prove the following:

\begin{cor}\label{Cor:Tel}
If $X$ is a finite motivic spectrum of type $(m,n)$ with non-nilpotent self-map $v : X \to \Sigma^{-r,-s} X$ of type $(m,n)$, then the map $X \to v^{-1} X$ is a finite $K(\beta_{mn})$-localization. 
\end{cor}

In other words, we can identify the finite $K(\beta_{mn})$-localization of a type $(m,n)$ spectrum $X$ with the telescope of its non-nilpotent self-map of type $(m,n)$. This leads to the following conjectures. 

\begin{conj}[Motivic Telescope Conjecture (Telescopic Formulation)]\label{Conj:MotTCTF}
Let $X$ be a finite motivic spectrum of type $(m,n)$ with non-nilpotent self-map $v : \Sigma^{-r,-s} X \to X$ of type $(m,n)$. Then $\langle v^{-1} X \rangle $ depends only on $m$ and $n$, and $\langle v^{-1} X \rangle = \langle K(\beta_{ij}) \rangle$. 
\end{conj}

Let $K(\leq \beta_{mn}) := \bigvee_{(i,j) \in S_{mn}} K(\beta_{i,j})$ where $S_{mn}$ is the set of $(i,j)$ such that $d_{ij} > d_{mn}$. We define $L^f_{mn} := L^f_{K(\leq \beta_{mn})}$ and $L_{mn} :=  L_{K(\leq \beta_{mn})}$. 

\begin{conj}[Motivic Telescope Conjecture (Localization Formulation)]\label{Conj:MotTCLF}
The natural transformation $L^f_{mn} \to L_{mn}$ is an equivalence. 
\end{conj}

\begin{conj}[Motivic Telescope Conjecture (Smashing Formulation)]\label{Conj:MotTCSF}
The map $L^f_{mn} S^{0,0} \to L_{mn} S^{0,0}$ is an equivalence. 
\end{conj}

The equivalence of these three formulations would follow from Corollary \ref{Cor:Tel} along with the following conjecture which may be viewed as a motivic analog of Ravenel's Smashing Conjecture \cite[10.6]{Rav84}.

\begin{conj}[Motivic Smashing Conjecture]\label{Conj:MotSC}
For each $(m,n)$, the localization functor $L_{mn}$ is smashing, i.e. $L_{mn} X \simeq X \wedge L_{mn} S^{0,0}$. 
\end{conj}

\subsection{The motivic Telescope Conjecture for $K(\leq \beta_{10})$}\label{Sec:TCeta}

Our goal in this section is to relate our calculations to Conjecture \ref{Conj:MotTCSF} in the case where $(m,n) = (1,0)$. In  Proposition \ref{Prop:etaTel}, we show that $L_{10}^f S^{0,0} \simeq  \eta^{-1}S^{0,0}$. We then show that $\eta^{-1} S^{0,0} \simeq L_{cKW} S^{0,0}$ in Proposition \ref{Prop:cKW}, where $KW \simeq \eta^{-1}KQ$ is Witt theory \cite{Ana12} and $cKW \simeq \eta^{-1} kq$ is connective Witt theory. Finally, we conjecture that $\langle cKW \rangle = \langle K(0) \vee K(w_0) \rangle$ in Conjecture \ref{Conj:cKWBous}, which would imply that $L_{cKW} S^{0,0} \simeq L_{10} S^{0,0}$. In summary, we will show
$$L_{10}^f S^{0,0} \simeq \eta^{-1}S^{0,0} \simeq L_{cKW} S^{0,0} \overset{?}{\simeq} L_{10} S^{0,0}.$$

\begin{prop}\label{Prop:etaTel}
There is an equivalence of motivic spectra $L^f_{10} S^{0,0} \simeq \eta^{-1} S^{0,0}$.
\end{prop}

Note that $S^{0,0}$ is of type $(0,-1)$ and not $(1,0)$, so this lemma is not automatic from Definition \ref{Def:Type} and Corollary \ref{Cor:Tel}.

\begin{proof}
Recall that $S^{0,0}$ has (at least) two non-nilpotent self-maps:
\begin{enumerate}
\item The degree $2$ self-map $2 : S^{0,0} \to S^{0,0}$ is a non-nilpotent self-map of type $(0,-1)$.
\item The first Hopf map $\eta : S^{1,1} \to S^{0,0}$ is a non-nilpotent self-map of type $(1,0)$. 
\end{enumerate}
In particular, $S^{0,0}$ is of type $(0,-1)$. In order to apply Corollary \ref{Cor:Tel}, we realize $S^{0,0}$ as a homotopy limit of finite spectra of type $(1,0)$, then verify that this homotopy limit commutes with $\eta$-localization. 

Let $S^{0,0}/2^k := \mathrm{cofib}(S^{0,0} \overset{2^k}{\to} S^{0,0})$ be the mod $2^k$ Moore spectrum. Then we have
$$S^{0,0} = (S^{0,0})^\wedge_2 \simeq \lim_k S^{0,0}/2^k,$$
where the first equality is a reminder that everything is implicitly $2$-complete. The long exact sequence in homotopy arising from the cofiber sequence defining $S^{0,0}/2^k$ implies that $2 : S^{0,0}/2^k \to S^{0,0}/2^k$ is nilpotent and that $\eta : \Sigma^{1,1} S^{0,0}/2^k \to S^{0,0}/2^k$ is non-nilpotent. Therefore $S^{0,0}/2^k$ is of type $(1,0)$ for all $k \geq 1$, so by Corollary \ref{Cor:Tel}, we have $L^f_{10}(S^{0,0}/2^k) \simeq \eta^{-1}( S^{0,0}/2^k)$ and thus
$$S^{0,0} \simeq \lim_k \eta^{-1}(S^{0,0}/2^k).$$
The lemma will follow if we can show that
$$\lim_k \eta^{-1}(S^{0,0}/2^k) \simeq \eta^{-1} (\lim_k S^{0,0}/2^k) \simeq \eta^{-1} S^{0,0}.$$

The homotopy groups of the left-hand side can be computed via an inverse limit of localized motivic Adams spectral sequences
$$E_2 = \lim_k h_{1}^{-1} \Ext^{***}_{A_*}(S^{0,0}/2^k) \Rightarrow \pi_{**}(\lim_k \eta^{-1}(S^{0,0}/2^k))$$
and the homotopy groups of the right-hand side can be computed via a localized inverse limit of motivic Adams spectral sequences
$$E_2' = h_1^{-1} \lim_k \Ext^{***}_{A_*}(S^{0,0}/2^k) \Rightarrow \pi_{**}(\eta^{-1} S^{0,0}).$$
Convergence of the relevant inverse limit spectral sequences follows from \cite[Prop. 4.2.22]{Gre12}, and convergence of the relevant localized spectral sequences follows from adapting \cite[Thm. 2.13]{MRS01}. 

Explicit calculation using the algebraic Atiyah-Hirzebruch spectral sequence and calculations of Guillou-Isaksen \cite[Thm. 1.1]{GI15} shows that the colimit-limit interchange map
$$\eta^{-1} S^{0,0} \to \lim_k \eta^{-1} S^{0,0}/2^k$$
induces an isomorphism of spectral sequences beginning with $E_2' \cong E_2$. The lemma follows by our previous remarks on convergence of both spectral sequences. 
\end{proof}

In particular, the calculation of Theorem \ref{Thm:EtaInv} may be interpreted as a calculation of $\pi_{**}(L_{10}^fS^{0,0})$. 

\begin{prop}\label{Prop:cKW}
There is an equivalence of motivic spectra
$$\eta^{-1} S^{0,0} \simeq L_{cKW}S^{0,0}.$$
\end{prop}

\begin{proof}
The left-hand side may be computed by inverting $\eta : S^{1,1} \to S^{0,0}$ in the $E_\infty$-page of the $kq$-resolution. On the other hand, the right-hand side may be computed using the $cKW$-based motivic Adams spectral sequence
$$E_1^{n,t,u} = \pi_{t,u}(\overline{cKW}^{\wedge n} \wedge cKW) \Rightarrow \pi_{**}(L_{cKW} S^{0,0}),$$
where convergence follows from \cite[Thm. 7.3.4]{Man18}. We have 
$$cKW = KW_{\geq 0} \simeq \eta^{-1} kq$$
by \cite{Ana12} and the definition of the homotopy t-structure. Therefore we may identify the $cKW$-based motivic Adams spectral sequence with the $\eta$-localized $kq$-resolution, where $\eta \in \pi_{1,1}(kq)$. This converges to $\eta^{-1} \pi_{**}(S^{0,0})$ by the obvious motivic analog of \cite[Thm. 2.13]{MRS01}; the proposition follows. 
\end{proof}

\begin{conj}\label{Conj:cKWBous}
There is an equivalence of Bousfield classes
$$\langle cKW \rangle = \langle K(0) \vee K(w_0) \rangle.$$
\end{conj}

We do not know how to prove this conjecture. However, Lemma \ref{Lem:BL} and Conjecture \ref{Conj:cKWBous} would imply that $L_{cKW} \simeq L_{10}$. Applying our previous computations, this would imply that $L_{10}^f S^{0,0} \to L_{10} S^{0,0}$ is an equivalence. If Conjecture \ref{Conj:MotSC} also holds in the case $(m,n) = (1,0)$, then all three formulations of the motivic Telescope Conjecture would hold for $(m,n) = (1,0)$. 

\begin{rem2}
It is also be possible to formulate an ``exotic motivic Telescope Conjecture" using only wedges of Krause's exotic Morava K-theories (but not Borghesi's motivic Morava K-theories). Indeed, if one sets $K(\leq \beta_{mn})^{ex} := \bigvee_{(i,j) \in S_{mn}^{ex}} K(\beta_{ij})$ where $S_{mn}^{ex} := S_{mn} \setminus \{(i,j) \in S_{mn} : j=-1\}$, then \cite[Thm. 6.12]{Kra18} identifies $L_{K(\leq \beta_{mn})^{ex}}^f X$ as the telescope of a non-nilpotent self-map of type $(m,n)$ for any $X$ such that $K(\beta_{ij})_{**}(X) = 0$ for $(i,j) \in S_{mn}^{ex} \setminus \{(m,n)\}$. One can therefore relate the various formulations of such an exotic motivic Telescope Conjecture. The case $(m,n) = (1,0)$ would then follow from our calculations if one can show that $\langle cKW \rangle = \langle K(w_0) \rangle$ and that $L_{K(w_0)}$ is smashing. 
\end{rem2}

\begin{rem2} 
An exotic motivic Telescope Conjecture as suggested above may be possible to study in the category of cellular $C\tau$-modules using \cite{GWX18}. For example, one sees that $\pi_{**}(\eta^{-1}S^{0,0})$ consists entirely of simple $\tau$-torsion classes. Therefore an exotic motivic Telescope Conjecture might be equivalent to a conjecture in the category of $C\tau$-modules. The latter is equivalent to the derived category of $BP_*BP$-comodules by \cite{GWX18}. Telescope conjectures always hold in the derived category of a commuative Hopf algebra by \cite[Sec. 6.3]{HPS97}, so it seems plausible that this exotic analog of the Telescope Conjecture holds. 
\end{rem2}

\subsection{The motivic Telescope Conjecture for $K(\leq \beta_{1,-1})$}\label{Sec:TCv1}

We now turn to Conjecture \ref{Conj:MotTCSF} in the case where $(m,n) = (1,-1)$. In Proposition \ref{Prop:v1Tel}, we show that $L_{1,-1}^fS^{0,0} \simeq v^{-1}_1 S^{0,0}$ using the inverse limit of localized Adams spectral sequences. We then show that $v^{-1}_1 S^{0,0} \simeq L_{KQ}S^{0,0}$ in Proposition \ref{Prop:KQLoc} by comparing localizations of the $kq$-resolution. Using the Wood cofiber sequence and the isotropy separation sequence for $KGL$, we then prove an equivalence of Bousfield classes $\langle KQ \rangle = \langle KGL \vee KW \rangle$ in Proposition \ref{Prop:Bousfield}, which implies $L_{KQ}S^{0,0} \simeq L_{KGL \vee KW} S^{0,0}.$ Finally, we outline the steps needed to prove Conjecture \ref{Conj:Bous} which would imply $\langle KGL \vee KW \rangle = \langle K(0) \vee K(w_0) \vee K(1) \rangle$ and thus $L_{KGL \vee KW} S^{0,0} \simeq L_{1,-1}S^{0,0}$. In summary, we will show
$$L_{1,-1}^f S^{0,0} \simeq v^{-1}_1 S^{0,0} \simeq L_{KQ} S^{0,0} \simeq L_{KGL \vee KW} S^{0,0} \overset{?}{\simeq} L_{1,-1}S^{0,0}.$$

We begin by defining a sequence of finite motivic spectra of type $(1,-1)$ whose homotopy limit is $S^{0,0}$. 

\begin{defin}
Recall from the previous section that $S^{0,0}/2^k$ admits a non-nilpotent self-map $\eta : S^{1,1}/2^k \to S^{0,0}/2^k$.  Let $Y_k := S^{0,0}/(2^k,\eta^k).$ 
\end{defin}

We observe that $\lim_k Y_k \simeq (S^{0,0})^\wedge_{2,\eta} \simeq (S^{0,0})^\wedge_2 = S^{0,0}$, where the second equivalence follows from \cite[Thm. 1]{HKO11a} and the last equality is a reminder that we have implicitly $2$-completed everything. 

\begin{lem}
For each $k \geq 1$, the motivic spectrum $Y_k$ admits a non-nilpotent self-map of type $(1,-1)$
$$v_1 : \Sigma^{2,1} Y_k \to Y_k.$$
\end{lem}

\begin{proof}
The map can be constructed using the long exact sequence in homotopy groups associated to the cofiber sequence
$$S^{1,1}/2^k \overset{\eta^k}{\to}S^{0,0}/2^k$$
as in the classical case. It is straightforward to verify that the map is $\tau$-torsion free, so it realizes to the analogous self-map of $S^0/(2^k,\eta^k)$. Since the classical self-map is non-nilpotent and Betti realization is strong symmetric monoidal, we conclude that $v_1$ is non-nilpotent motivically. 
\end{proof}

\begin{defin}
We define the \emph{$v_1$-inverted motivic sphere spectrum} $v^{-1}_1 S^{0,0}$ by setting
$$v^{-1}_1 S^{0,0} := \lim_k v^{-1}_1 Y_k.$$
\end{defin}

\begin{prop}\label{Prop:v1Tel}
There is an equivalence of motivic spectra
$$ L^f_{1,-1}S^{0,0} \simeq v^{-1}_1  S^{0,0}.$$
\end{prop}

\begin{proof}
It suffices to show that $v^{-1}_1 S^{0,0}$ is finitely $K(\leq \beta_{1,-1})$-local. Since $Y_k$ is of type $(1,-1)$, we have $L^f_{1,-1} Y_k \simeq v^{-1}_1 Y_k$ for all $k \geq 1$. Let $W$ be a finitely $K(\leq \beta_{1,-1})$-acyclic motivic spectrum. Then if $f \in [W,v^{-1}_1 S^{0,0}] = [ W,\lim_k v^{-1}_1 Y_k]$ is nontrivial, we must have $f : W \to v^{-1}_1 Y_N$ nontrivial for some $N \geq 1$. But this is a contradiction since $v_1^{-1}Y_N \simeq L^f_{1,-1} Y_N$ is finitely $K(\leq \beta_{1,-1})$-acyclic. 
\end{proof}

\begin{prop}\label{Prop:KQLoc}
There is an equivalence of motivic spectra
$$v^{-1}_1 S^{0,0} \simeq L_{KQ}S^{0,0}.$$
\end{prop}

\begin{proof}
The proof is similar to the proof of Proposition \ref{Prop:cKW}. The homotopy groups of the left-hand side may be obtained by inverting $v_1$ in the $E_\infty$-page of the $kq$-resolution. The homotopy groups of the right-hand side may be obtained using the $KQ$-based motivic Adams spectral sequence. Since $KQ \simeq v^{-1}_1 kq$, we may identify the $v_1$-localized $kq$-resolution with the $KQ$-based motivic Adams spectral sequence. The result follows. 
\end{proof}

\begin{prop}\label{Prop:Bousfield}
There is an equivalence of Bousfield classes $\langle KQ \rangle = \langle KGL \vee KW \rangle$. 
\end{prop}

We thank the motivic homotopy theory group in Osnabr{\"u}ck for help with the following proof. 

\begin{proof}
We start by showing that $\langle KQ \rangle \subseteq \langle KGL \vee KW \rangle$. Let $X$ be a motivic spectrum with $KQ_{**}(X) = 0$, so $X \in \langle KQ \rangle$. The Wood cofiber sequence \cite[Thm. 3.4]{RO16}
$$\Sigma^{1,1} KQ \overset{\eta}{\to} KQ \to KGL$$
implies that $KGL_{**}(X) = 0$. Since $KW \simeq \eta^{-1} KQ$, we also see that $KW_{**}(X) = 0$. Therefore $X \in \langle KGL \vee KW \rangle$.

Now suppose $X$ satisfies $(KGL \vee KW)_{**}(X) = 0$. Then $KGL_{**}(X) = 0$ and $KW_{**}(X) = 0$. Recall the cofiber sequence \cite[Sec. 5.2]{HKO11b}
$$KGL_{hC_2} \to KGL^{C_2} \to \Phi^{C_2} KGL.$$
which can be further identified with
$$KGL_{hC_2} \to KQ \to KW.$$
Since homotopy orbits is a filtered colimit, it commutes with homology and smash products. Therefore $KGL_{**}(X) = 0$ implies that $(KGL_{hC_2})_{**}(X) = 0$. Since $KW_{**}(X) = 0$, we conclude that $KQ_{**}(X) = 0$ and therefore $X \in \langle KQ \rangle$. 
\end{proof}

\begin{rem2}
More generally, Proposition \ref{Prop:Bousfield} follows from the motivic analog of \cite[Lem. 1.34]{Rav84} by taking $X = KQ$ and $g = \eta$, so $Y \simeq KGL$ and $\hat{X} \simeq KW$. 
\end{rem2}

\begin{conj}\label{Conj:Bous}
There are equivalences of Bousfield classes
$$\langle KGL \rangle = \langle K(0) \vee K(1) \rangle \quad \text{ and } \quad \langle KW \rangle = \langle K(w_0) \vee K(1) \rangle.$$
\end{conj}

\begin{rem2}
The first equivalence would follow from a motivic analog of \cite[Thm. 2.1.(d)]{Rav84}, which in turn would rely on a motivic analog of Johnson and Wilson's formula for $K(n)^{cl}_*(X)$ in \cite{JW75}. We do not know how to approach the second equivalence. 
\end{rem2}

If Conjecture \ref{Conj:Bous} holds, then we would have $L_{KQ} \simeq L_{KGL \vee KW} \simeq L_{K(\leq \beta_{1,-1})}$. In particular, we would have $L_{KGL \vee KW} S^{0,0} \simeq L_{1,-1} S^{0,0}$. If we also knew that $L_{1,-1}$ (or equivalently, $L_{KQ}$ or $L_{KGL \vee KW}$) is smashing, then all three formulations of the motivic Telescope Conjecture for $(m,n) = (1,-1)$ would hold. 

\begin{rem2}\label{Rem:TelGen}
The formulation of the $\c$-motivic Telescope Conjectures depended on the existence of exotic Morava K-theories $K(\beta_{ij})$. These are only known to exist over $\Spec(\c)$, so it is unclear how to formulate the motivic Telescope Conjecture over general base fields. With that said, many of the spectra appearing in this section, such as $KGL$, $KQ$, $KW$, and $cKW$, are defined over general base fields. If one supposes Conjecture \ref{Conj:Bous} holds over general base fields, then for example, the Telescope Conjecture at height $(1,1)$ would ask that the map $L_{KGL \vee KW}^fS^{0,0} \to L_{KGL \vee KW} S^{0,0}$ is an equivalence. We plan to analyze this conjecture in future work using the $kq$-resolution. 
\end{rem2}

\section{Motivic $J$-homomorphisms}\label{Sec:JHom}

We conclude by discussing the relation of our calculations to the motivic J-homomorphism. We begin by recalling the work of Hu-Kriz-Ormsby on the motivic unitary J-homomorphism in Section \ref{Sec:Ju}. In Section \ref{Sec:GIKR}, we recall and build upon recent work of Gheorghe-Isaksen-Krause-Ricka \cite{GIKR18} which provides a topological model for the category of $\c$-motivic cellular spectra. Using these results, we propose a model for the motivic stable orthogonal J spectrum in Section \ref{Sec:Jo} and determine its location in the $kq$-resolution. 

\subsection{Recollection of the motivic unitary J-homomorphism}\label{Sec:Ju}
Recall that Hu-Kriz-Ormsby defined a motivic unitary J-homomorphism 
$$J_u: \pi_k(K^{alg}(K)) \to \pi^K_{2\ell+k-1,\ell}(S^{0,0})$$
in \cite[Sec. 1]{HKO11} (see also \cite[16.2]{BH17}), where $k + \ell > 1$ and $K$ is any algebraically closed field of characteristic zero. They computed its image using the motivic Adams-Novikov spectral sequence:

\begin{thm}\cite[Thm. 3]{HKO11}
The image of $J_u$ is isomorphic to the $2$-primary component of the image of the classical unitary J-homomorphism in dimension $2\ell + k -1 $. 
\end{thm}

In particular, the image of the unitary J-homomorphism only contains $\tau$-free classes in $\pi_{**}(S^{0,0})$. The corresponding elements in the classical stable stems are contained in the image of the classical orthogonal J-homomorphism since the unitary J-homomorphism factors through the orthogonal J-homomorphism.  By \cite{DM89}, the image of the orthogonal J-homomorphism is detected in the $0$- and $1$-lines of the $bo$-resolution. Translating back to the motivic setting using Betti realization and $\tau$-localization, we obtain the following corollary:

\begin{cor}\label{Cor:Ju}
The image of the motivic unitary J-homomorphism $J_u$ is detected in the $1$-line of the $kq$-resolution. More precisely, it is contained in the cokernel of the nontrivial $d_1$-differentials from $kq$-Adams filtration zero to one. 
\end{cor}

\subsection{The Gheorghe-Isaksen-Krause-Ricka equivalence}\label{Sec:GIKR}
Let $Sp^{\z^{op}} = Fun(\z^{op},Sp)$, where $\z$ is regarded as a poset, be the category of \emph{filtered spectra} \cite[Def. 2.1]{GIKR18}. In \cite{GIKR18}, Gheorghe-Isaksen-Krause-Ricka construct a functor 
$$\Gamma_\star : Sp \to Mod_{\Gamma_\star S^0}$$
from the category of classical spectra to the category of left modules over $\Gamma_\star S^0 $ \cite[Def. 3.10]{GIKR18}. Concretely, this functor is defined by
$$\Gamma_\star X : w \mapsto Tot(\tau_{\geq 2w} (X \wedge MU^{\bullet+1}))$$
where $MU^{\bullet + 1}$ is the usual cosimplicial resolution of the $E_\infty$-ring spectrum $MU$, $\tau_{\geq 2w}$ is the $2w$-connective cover, and $Tot$ is totalization. 

The category of $\Gamma_{\star}S^0$-modules is equivalent to the category of cellular $\c$-motivic spectra \cite[Thm. 6.12]{GIKR18} after $2$-completion. The equivalence is induced by adjoint functors
$$- \otimes S^{0,\star} : Mod_{\Gamma_\star S^0} \leftrightarrows \Motc : \Omega^{0,\star}_s.$$
The right adjoint is defined by 
$$\Omega^{0,w}_s(X) := F_s(S^{0,w},X)$$
where $F_s(S^{0,w},X)$ is the (classical) mapping spectrum between motivic spectra. The left-adjoint is given by
$$\hocolim_{i \geq j} Y_i \otimes S^{0,j}$$
where $Y_i \otimes -$ is defined using the fact that $\Motc$ is enriched over classical spectra. 

The following lemma and construction are due to Mark Behrens.

\begin{lem}
Let $Re : \Motc \to Sp$ denote Betti realization and let $E_\star$ be a filtered spectrum with $E_j = \colim_{i \to \infty} E_{-i}$ for all $j \ll 0$. The functor $\widetilde{Re} : Mod_{\Gamma_\star S^0} \to Sp$ in the commutative diagram
\[
\begin{tikzcd}
\Motc \arrow[r,bend left=5] \arrow{d}{Re} & Mod_{\Gamma_\star S^0} \arrow{dl}{\widetilde{Re}} \arrow[l, bend left=5]\\
Sp
\end{tikzcd}
\]
is given by
$$\widetilde{Re}(E_\star) = \colim_{i \to \infty} E_{-i}.$$
\end{lem}

\begin{proof}
Betti realization commutes with colimits and is strong symmetric monoidal, so we have 
$$Re(E_\star \otimes S^{0,\star}) = Re(\colim_{i \geq j} E_i \otimes S^{0,j}) = \colim_{i \geq j} Re(E_i) \wedge S^0 = \colim_{i \to \infty} E_{-i}.$$
\end{proof}

\begin{constr}\label{Constr:Map}
Suppose that $E_\star \in Mod_{\Gamma_\star S^0}$ satisfies 
$$Tot(\tau_{\geq 2w}(\colim_{i \to \infty} E_{-i} \wedge MU^{\bullet + 1})) \simeq Tot(\tau_{\geq 2w} E_w \wedge MU^{\bullet + 1})$$ 
for all $w \in \z^{op}$ and that under this equivalence, the filtered spectrum structure maps factor as
\begin{equation}\label{Eqn:StrComp}
Tot(\tau_{\geq 2w}(E_w \wedge MU^{\bullet+1})) \to Tot(\tau_{\geq 2w-2}(E_w \wedge MU^{\bullet+1})) \to Tot(\tau_{\geq 2w-2}(E_{w-1} \wedge MU^{\bullet+1})).
\end{equation}
Suppose further that the Adams-Novikov spectral sequence for $E_w$ converges for all $w \in \z^{op}$. Then we define a map
$$\Gamma_\star(\colim_{i \to \infty} E_{-i}) \to E_\star$$
which is given for $\star = w$ by
$$Tot(\tau_{\geq 2w} (\colim_{i \to \infty} E_{-i} \wedge MU^{\bullet +1})) \simeq Tot(\tau_{\geq 2w} (E_w \wedge MU^{\bullet +1})) \hookrightarrow Tot(E_w \wedge MU^{\bullet+1}) \simeq E_w.$$
The filtered spectrum structure maps on the left-hand side are given by Equation \ref{Eqn:StrComp}
and the filtered spectrum structure maps on the right-hand side are given by 
$$Tot(E_w \wedge MU^{\bullet+1}) \to Tot(E_{w-1} \wedge MU^{\bullet+1}),$$
so the map above commutes with the filtered spectrum structure maps. 
\end{constr}

\begin{cor}
If $E$ is a motivic spectrum such that $\Omega^{0,\star}_s(E)$ satisfies the above conditions, then there exists a map of motivic  spectra
$$(\Gamma_\star( Re(E)) )\otimes S^{0,\star} \to E.$$
\end{cor}

\begin{lem}\label{Lem:Conditions}
Suppose that $E$ is a motivic spectrum for which the motivic Adams-Novikov spectral sequence converges and for which there exists a $u \in \z$ such that $\pi_{*,u}(E) \cong \pi_{*,u-n}(E)$ for all $n \geq 0$. Then $\Omega^{0,\star}_s(E)$ satisfies the conditions of Construction \ref{Constr:Map}. 
\end{lem}

\begin{proof}
This is immediate from the properties of the function spectrum $F_s(-,-)$. 
\end{proof}

Part of the following corollary is asserted without proof in \cite[Table 1]{GIKR18}. 

\begin{cor}
There are equivalences of motivic spectra
$$H\f_2 \simeq \Gamma_\star(H\f_2^{cl}), \quad H\z \simeq \Gamma_\star(H\z^{cl}), \quad kgl \simeq \Gamma_\star(bu), \quad kq \simeq \Gamma_\star(bo), \quad ksp\simeq \Gamma_\star(bsp).$$
\end{cor}

\begin{proof}
The Betti realization of each motivic spectrum on the left-hand side of the equivalence coincides with the classical spectrum that $\Gamma_\star$ is being applied to on the right-hand side. Each motivic spectrum satisfies the hypotheses of Lemma \ref{Lem:Conditions}, so we obtain maps from the right-hand side of each equivalence to the left-hand side by Construction \ref{Constr:Map}. These maps induce an isomorphism between motivic Adams-Novikov spectral sequences by construction. The motivic Adams-Novikov spectral sequence converges for each motivic spectrum above, so the map is an equivalence of motivic spectra. 
\end{proof}

\begin{cor}
There is an isomorphism between the $kq$-resolution and the $\Gamma_\star(bo)$-resolution. 
\end{cor}

\begin{rem2}
In other words, we could have replaced $kq$ by $\Gamma_\star(bo)$ in all of the analysis above and arrived at the same conclusions. On the other hand, we do not claim that $\Gamma_\star$ applied to the $bo$-resolution is equivalent to the $kq$-resolution. The functor $\Gamma_\star$ is not generally strong symmetric monoidal \cite[Rmk. 3.8]{GIKR18}, and we cannot apply \cite[Prop. 3.25]{GIKR18} to obtain an equivalence $\Gamma_\star(bo) \wedge_{\Gamma_\star(S^0)} \Gamma_\star(bo) \simeq \Gamma_\star(bo \wedge bo)$.
\end{rem2}

\subsection{A model for the motivic stable orthogonal J-homomorphism}\label{Sec:Jo}

We now study the orthogonal J-homomorphism using $\Gamma_\star$. Following the philosophy that $\Gamma_\star$ of a classical spectrum gives its correct motivic analog, one might try to construct the $\c$-motivic orthogonal J-homomorphism by applying $\Gamma_\star$ to the classical orthogonal J-homomorphism. Unfortunately, this definition runs into the following technical difficulty. Recall that the classical J-homomorphism is obtained by taking the colimit of maps $O(n) \to \Omega^n S^n$ to obtain 
$$J_o^{cl} : O \to GL_1(S^0).$$
This is an infinite loop map by construction, so we can deloop once and take the corresponding map of $\Omega$-spectra to obtain
$$J^{cl}_o : \Sigma^{-1} bo \to gl_1S^0$$
(see e.g. \cite{AHR10}). Applying $\Gamma_\star$ produces a map of motivic spectra 
$$\Gamma_\star J^{cl}_o : \Sigma^{-1,-1} kq \to \Gamma_\star gl_1S^0.$$
This map has the correct source, but it is not clear that $\Gamma_\star(gl_1(S^0))$ corresponds to the spectrum of units of the complex motivic sphere. 

To avoid questions about motivic infinite loop spaces, we employ a different description of the orthogonal $J$ spectrum. In \cite{DM89}, Davis and Mahowald define a connective J spectrum $j$ (denoted $J$ in their paper) via the fiber sequence
\begin{equation}\label{eq: connective j spectrum fiber sequence}
\begin{tikzcd}
j\arrow[r] & bo \arrow[r,"\psi^3-1"] & \Sigma^4 bsp. 
\end{tikzcd}
\end{equation}
The $d_1$-differential in the $bo$-resolution has the form $d_1 : bo \to bo \wedge \overline{bo}$. Pinching onto the first summand in the spectral splitting
$$bo \wedge \overline{bo} \simeq bo \wedge (\Sigma^4 bsp \vee \bigvee_{i \geq 2} \Sigma^{4i} H\z_i^{cl}),$$
shows that
$$\coker(d_1: bo \to bo \wedge \overline{bo}) \cong \coker(\psi^3-1: bo \to \Sigma^4 bsp),$$
$$\ker(d_1: bo \to bo \wedge \overline{bo}) \cong \ker(\psi^3-1 : bo \to \Sigma^4 bsp).$$
There are no further differentials involving the $0$- and $1$-line, so one finds that $\pi_*(j)$ is isomorphic to the $0$- and $1$-lines of the $bo$-resolution. Now, the $0$-line of the $bo$-resolution is just the Hurewicz image of $bo$, and the $1$-line is the image of the orthogonal J-homomorphism. Therefore we can determine the image of J from $\pi_*(j)$ by removing the Hurewicz image of $bo$. 

\begin{defin}
We define the \emph{$\c$-motivic connective orthogonal J spectrum} $j_o$ by applying $\Gamma_\star$ to the connective orthogonal $J$ spectrum, i.e. $j_o := \Gamma_\star(j)$. 
\end{defin}

We now show that $j_o$ fits into a fiber sequence with $kq$ and $\Sigma^4 ksp$ (this is not immediately obvious since $\Gamma_\star$ does not generally preserve fiber sequences of spectra). 

\begin{lem}
The groups $MU_*(bo)$ are concentrated in even degrees. 
\end{lem}

We learned the following proof from Mark Behrens. 

\begin{proof}
Consider the cofiber sequence
$$\Sigma bo \wedge MU \overset{\eta}{\to} bo \wedge MU \to bo \wedge MU \wedge C\eta.$$
Since multiplication by $\eta$ is trivial in $MU_*$, turning triangles gives a short exact sequence
$$0 \to \pi_*(bo \wedge MU) \to \pi_*(bo \wedge MU \wedge C\eta) \to \pi_*(\Sigma^2 bo \wedge MU).$$
Recall that $bo \wedge C\eta \simeq bu$. Since $MU_*(bu)$ is concentrated in even degrees, the result follows. 
\end{proof}

\begin{cor}
The groups $MU_*(bsp)$ and $MU_*(j)$ are concentrated in even degrees.
\end{cor}
\begin{proof}
This is because $bsp = \tau_{\geq 0}\Sigma^{-4}bo$. The corresponding result for $MU_*(j)$ follows from the fiber sequence \eqref{eq: connective j spectrum fiber sequence}.
\end{proof}

Since all the required $MU$-homology groups are concentrated in even degrees, it follows from \cite[Lem. 3.17]{GIKR18} that we have a fiber sequence
\[
\begin{tikzcd}
\Gamma_\star j\arrow[r] & \Gamma_\star bo \arrow[r] & \Gamma_\star \Sigma^4 bsp
\end{tikzcd}
\]
of $\c$-motivic spectra. Also, we note from \cite[Lem 3.12]{GIKR18} that $\Gamma_\star(\Sigma^4 bsp)\simeq \Sigma^{4,2}\Gamma_\star bsp$. We have shown the following.

\begin{thm}\label{Thm:Jo}
There is a fiber sequence
\begin{equation}\label{Eqn:J}
j_o\to kq \to \Sigma^{4,2}ksp.
\end{equation}
The homotopy groups $\pi_{**}(j_o)$ are isomorphic to the $0$- and $1$-lines of the $kq$-resolution. 
\end{thm}

%
\begin{rem2}
The Hurewicz image of $kq$ is detected in the $0$-line. We therefore expect the image of a $\c$-motivic orthogonal J-homomorphism to be isomorphic to the $1$-line of the $kq$-resolution. 
\end{rem2}

\begin{rem2}\label{Rem:BH}
A variant of \eqref{Eqn:J} exists over general base fields after $p$-localization. Indeed, forthcoming work of Bachmann and Hopkins shows that there is a cofiber sequence of $2$-local spectra
$$\eta^{-1}S^{0,0} \to kq \xrightarrow{(\psi^3-1)} \Sigma^{4,2}ksp.$$

\end{rem2}

%

%
%
%
%
%
%
%
%
%
%
%

\bibliographystyle{plain}
\bibliography{master}

\end{document}